\documentclass[11pt,a4paper]{amsart}
\usepackage{amsthm,amsmath,amssymb,latexsym,epsfig,graphicx,color,verbatim}    
\usepackage[matrix,arrow,curve]{xy}

\newtheorem{PARA}{}[section]
\newtheorem{theorem}[PARA]{Theorem}
\newtheorem{corollary}[PARA]{Corollary}
\newtheorem{lemma}[PARA]{Lemma}
\newtheorem{proposition}[PARA]{Proposition}
\newtheorem{definition}[PARA]{Definition}

\theoremstyle{definition}
\newtheorem{remark}[PARA]{Remark}
\newtheorem{example}[PARA]{Example}

\newcommand{\para}{\begin{PARA}\rm}
\newcommand{\arap}{\end{PARA}\rm}
\newcommand{\dfn}{\begin{definition}\rm}
\newcommand{\nfd}{\end{definition}\rm}
\newcommand{\rmk}{\begin{remark}\rm}
\newcommand{\kmr}{\end{remark}\rm}
\newcommand{\xmpl}{\begin{example}\rm}
\newcommand{\lpmx}{\end{example}\rm}
\newcommand{\cA}{\mathcal{A}}

\newcommand{\cM}{\mathcal{M}}

\newcommand{\cP}{\mathcal{P}}

\newcommand{\one}
{{{\mathchoice \mathrm{ 1\mskip-4mu l} \mathrm{ 1\mskip-4mu l}
\mathrm{ 1\mskip-4.5mu l} \mathrm{ 1\mskip-5mu l}}}}

\newcommand{\C}{{\mathbb{C}}}

\newcommand{\Q}{{\mathbb{Q}}}
\newcommand{\R}{{\mathbb{R}}}

\newcommand{\Z}{{\mathbb{Z}}}
\newcommand{\im}{\mathrm{ im }}        
\newcommand{\id}{\mathrm{ id}}         

\newcommand{\eps}{{\varepsilon}}
\newcommand{\om}{{\omega}}
\newcommand{\Om}{{\Omega}}

\newcommand{\tp}{{\widetilde{p}}}

\newcommand{\tPhi}{{\widetilde{\Phi}}}
\newcommand{\tSH}{{\widetilde{SH}}}

\newcommand{\CZ}{\mathrm{CZ}}
\def\NABLA#1{{\mathop{\nabla\kern-.5ex\lower1ex\hbox{$#1$}}}}
\def\Nabla#1{\nabla\kern-.5ex{}_{#1}}
\def\Tabla#1{\Tilde\nabla\kern-.5ex{}_{#1}}
\renewcommand{\Tilde}{\widetilde}

\newcommand{\p}{{\partial}}

\newcommand{\wh}{\widehat}
\newcommand{\MM}{\mathcal{M}}
\newcommand{\Ad}{{\rm Ad}}
\newcommand{\Sp}{{\rm Sp}}

\newcommand{\into}{\hookrightarrow}
\newcommand{\Spec}{{\rm Spec}}
\newcommand{\Crit}{{\rm Crit}}
\newcommand{\LL}{{\mathcal{L}}}

\begin{document}

\title{Rabinowitz Floer homology and symplectic homology}

\author{Kai {\sc Cieliebak}}
\address{Ludwig-Maximilians-Universit\"at, D-80333, M\"unchen,
Germany}
\email{kai@mathematik.uni-muenchen.de}

\author{Urs {\sc Frauenfelder}}
\address{Ludwig-Maximilians-Universit\"at, D-80333, M\"unchen,
Germany}
\email{frauenf@math.lmu.de}

\author{Alexandru \sc{Oancea}}
\address{IRMA, Universit\'e Louis Pasteur, F-67084, Strasbourg, France}
\email{oancea@math.u-strasbg.fr}

\date{18 February 2009}

\maketitle

\setcounter{tocdepth}{1}
\tableofcontents


\begin{abstract} The first two authors have recently defined
  Rabinowitz-Floer homology groups $RFH_*(M,W)$ associated to an exact
  embedding of a contact manifold $(M,\xi)$ into a symplectic manifold
  $(W,\omega)$. These depend only on the bounded component $V$ of
  $W\setminus M$. We construct a long exact sequence in which
  symplectic cohomology of $V$ maps to symplectic homology of $V$,
  which in turn maps to Rabinowitz-Floer homology $RFH_*(M,W)$, which
  then maps to symplectic cohomology of $V$. We 
  compute $RFH_*(ST^*L,T^*L)$, where $ST^*L$ is the unit
  cosphere bundle of a closed manifold $L$. As an application, we
  prove that the image of an exact contact embedding of $ST^*L$
  (endowed with the standard contact structure) cannot be displaced
  away from itself by a Hamiltonian isotopy, provided $\dim\,L\ge 4$
  and the embedding induces an injection on $\pi_1$. In particular,
  $ST^*L$ does not admit an exact contact embedding into a subcritical
  Stein manifold if $L$ is simply connected. We also prove that
  Weinstein's conjecture holds in symplectic manifolds which admit
  exact displaceable codimension $0$ embeddings. 
\end{abstract}


\section{Introduction} 

Let $(W,\lambda)$ be a complete convex exact symplectic manifold, with
symplectic form $\omega=d\lambda$ (see Section~\ref{sec:RFH} for the
precise definition). An embedding $\iota :M \hookrightarrow W$ of a
contact manifold $(M,\xi)$ is called {\it exact contact embedding} if
there exists a 1-form $\alpha$ on $M$ such that such that
$\ker\alpha=\xi$ and $\alpha-\lambda|_M$ is exact.
We identify $M$ with its image $\iota(M)$. We assume that $W\setminus
M$ consists of two connected components and denote the bounded
component of $W\setminus M$ by $V$. 
One can classically~\cite{V} associate to such an exact contact
embedding the {\it symplectic (co)homology groups} $SH_*(V)$ and
$SH^*(V)$. We refer to Section~\ref{sec:SH} for the definition and
basic properties, and
to~\cite{Se} for a recent survey. 

The first two authors have recently defined for such an
exact contact embedding Floer homology groups $RFH_*(M,W)$ for the
Rabinowitz action functional~\cite{CF}. We refer to Section~\ref{sec:RFH} for a
recap of the definition and of some useful properties. We will show
in particular that these groups do not depend on $W$, but only on $V$
(the same holds for $SH_*(V)$ and $SH^*(V)$). We
shall use in this paper the notation $RFH_*(V)$ and call them {\it
Rabinowitz Floer homology} groups. 

\begin{remark}\label{rem:grading}
All (co)homology groups are taken with field coefficients.
Without any further hypotheses on the first Chern class $c_1(V)$ of
the tangent bundle, the
symplectic (co)homology and Rabinowitz Floer homology groups are
$\Z_2$-graded. If $c_1(V)=0$ they are $\Z$-graded, and if $c_1(V)$
vanishes on $\pi_2(V)$ the part constructed from contractible loops is
$\Z$-graded. This $\Z$-grading on Rabinowitz Floer homology differs
from the one in~\cite{CF} (which takes values in $\frac 1 2 + \Z$) by
a shift of $1/2$ (see Remark~\ref{rmk:grading}).
\end{remark}

Our purpose is to relate these two constructions. The relevant 
object is a new version of symplectic homology, denoted $\Check{SH}_*(V)$,
associated to ``$\bigvee$-shaped'' Hamiltonians like the one in
Figure~\ref{fig:Vshape} on page~\pageref{fig:Vshape} below.
This version of symplectic homology is related to the usual ones via
the long exact sequence in the next theorem. 

\begin{theorem} \label{thm:main_exact}
There is a long exact sequence 
\begin{equation} \label{eq:long_seq}
\xymatrix{
\ldots \ar[r] &  \!SH^{-*}(V)\! \ar[r] & \!SH_*(V)\! \ar[r] &
\!\Check{SH}_*(V)\! \ar[r]
& \!SH^{-*+1}(V)\! \ar[r] & \ldots
}
\end{equation}
\end{theorem}

The long exact sequence~\eqref{eq:long_seq} can be seen as measuring
the defect from being an isomorphism of the canonical map
$SH^{-*}(V)\to SH_*(V)$, which we define in Section~\ref{sec:Vshaped}. An
interesting fact is that we have a very
precise description of this map. To state it, let us recall that
there are canonical morphisms induced by truncation of the range of
the action $H_{*+n}(V,\partial V)\stackrel {c_*} \longrightarrow
SH_*(V)$ and $SH^*(V)\stackrel {c^*} \longrightarrow
H^{*+n}(V,\partial V)$ (see~\cite{V} or Lemma~\ref{lem:sing} below).  

\begin{proposition} \label{prop:map}
 The map $SH^{-*}(V)\to SH_*(V)$ fits into a commutative diagram 
\begin{equation} \label{eq:map}
\xymatrix
@R=15pt
{
SH^{-*}(V) \ar[r] \ar[d]_{c^*} & SH_*(V) \\
H^{-*+n}(V,\partial V) \ar[r] & H_{*+n}(V,\partial V) \ar[u]_{c_*}
}
\end{equation}
in which the bottom arrow is the composition of the map induced by the 
inclusion $V\hookrightarrow (V,\partial V)$ with the Poincar\'e duality
isomorphism 
$$
H^{-*+n}(V,\partial
V)\stackrel {PD} \longrightarrow H_{*+n}(V) \stackrel {incl_*}
\longrightarrow H_{*+n}(V,\partial V).
$$
\end{proposition}

We also define in
Section~\ref{sec:Vshaped} truncated versions
$\Check{SH}_*^{\ge 0}(V)$ and $\Check{SH}_*^{\le 0}(V)$ of the
symplectic homology groups $\Check{SH}_*(V)$. 

\begin{proposition} \label{prop:main2}
There are commuting diagrams of long exact sequences as below, where
$PD$ denotes Poincar\'e duality and the top exact sequence is the
(co)homological long exact sequence of the pair $(V,M)$:
$$
\xymatrix
@C=10pt
{
\ldots \ar[r] &  \!H_{*+n}(V)\! \ar[r] \ar@{=}[d]^{\textrm{PD}} &
\!H_{*+n}(V,M)\! \ar[r] \ar[d] & \!H_{*+n-1}(M)\! \ar[r] \ar[d] &
\!H_{*+n-1}(V)\! \ar[r] \ar@{=}[d]^{\textrm{PD}} & 
\ldots \\ 
\ldots \ar[r] &  \!H^{-*+n}(V,M)\! \ar[r] & \!SH_*(V)\! \ar[r] &
\!\Check{SH}_*^{\ge 0}(V)\! \ar[r] & \!H^{-*+1+n}(V,M)\! \ar[r] & \ldots
}
$$
and
$$
\xymatrix
@C=10pt
{
\ldots \ar[r] &  H^{-*+n}(V,M) \ar[r] &
H^{-*+n}(V) \ar[r] \ar@{=}[d]^{\textrm{PD}} & H^{-*+n}(M) \ar[r] &
H^{-*+n+1}(V) \ar[r] & \ldots \\ 
\ldots \ar[r] &  SH^{-*}(V) \ar[r] \ar[u] & H_{*+n}(V,M) \ar[r] &
\Check{SH}_*^{\le 0}(V) \ar[r] \ar[u] & SH^{-*+1}(V) \ar[r] \ar[u] & \ldots
}
$$
\end{proposition} 

The main result of this paper is the following. 

\begin{theorem} \label{thm:main1}
 We have an isomorphism 
$$
RFH_*(V) \simeq \Check{SH}_*(V).
$$
\end{theorem}

Theorem~\ref{thm:main1} is proved in Section~\ref{sec:proof}. It
follows that the Rabinowitz Floer homology groups fit into a long
exact sequence    
\begin{equation} \label{eq:RFHseq}
\xymatrix{
\ldots \ar[r] &  \!SH^{-*}(V)\! \ar[r] & \!SH_*(V)\! \ar[r] &
\!RFH_*(V)\! \ar[r]
& \!SH^{-*+1}(V)\! \ar[r] & \ldots
}
\end{equation}

We also recall the following vanishing result for Rabinowitz Floer homology
from~\cite{CF}. 

\begin{theorem}[{\cite[Theorem~1.2]{CF}}] \label{thm:vanishingRFH}
If $M=\p V$ is Hamiltonianly displaceable in $W$, then 
$$
RFH_*(V)=0.
$$
\end{theorem}

To state the next corollary, we recall that the symplectic (co)homology
and Rabinowitz Floer homology groups decompose as direct sums 
$$
SH_*(V)=\oplus _c SH_*^c(V), \ \  SH^*(V)=\oplus_c SH^*_c(V), \ \
RFH_*(V)=\oplus_c RFH_*^c(V) 
$$
indexed over free homotopy classes of loops \emph{in $V$}. We denote
the free homotopy class of the constant loops by $c=0$. 

\begin{corollary} \label{cor:vanishingSH}
 Assume $M=\partial V$ is Hamiltonianly displaceable in $W$. 
\begin{itemize}
\item For $c\neq 0$ we have  
$$
SH_*^c(V)=0, \quad SH^*_c(V)=0.
$$
\item Suppose that $c_1(W)|_{\pi_2(W)}=0$. Then for $c=0$ we have 
\begin{equation} \label{eq:vanishingc=0}
SH_*^{c=0}(V)=0, \quad SH^*_{c=0}(V)=0 
\end{equation}
if $*\ge n$ or $*\le -n$. Moreover, if $V$ is Stein
then~\eqref{eq:vanishingc=0} holds for $*\neq 0$, and if $V$ is Stein
subcritical then~\eqref{eq:vanishingc=0} holds for all $*\in \Z$. 
\end{itemize}
\end{corollary}

\begin{proof} The long exact sequence~\eqref{eq:RFHseq}
splits into a direct sum of long exact sequences, indexed over free
homotopy classes of loops in $V$. The assumption that $M$ is
Hamiltonianly displaceable implies $RFH_*(V)=0$, hence the map
$SH^{-*}_{c}(V)\to SH_*^c(V)$ is an isomorphism for any $c$.  

We now use the commutative diagram in Proposition~\ref{prop:map} and
the fact that the canonical map $c_*:H_{*+n}(V,\p V)\to SH_*(V)$
takes values into the direct summand $SH_*^{c=0}(V)$, and similarly,
the map $c^*:SH^*(V)\to H^{*+n}(V,\p V)$ factors through 
$SH^*_{c=0}(V)$ (see Lemma~\ref{lem:sing} and Lemma~\ref{lem:sing-coh}
below).  

Let us assume $c\neq 0$. Then the above discussion shows that the map 
$SH^{-*}_{c}(V)\to SH_*^c(V)$ is at the same time an isomorphism and
vanishes. This implies the conclusion. 

Let us now assume $c=0$ and $c_1(W)|_{\pi_2(W)}=0$, so that all
homology groups are $\Z$-graded. By Proposition~\ref{prop:map}, the
map $SH^{-*}_{c=0}(V)\to SH_*^{c=0}(V)$ is the composition 
$$
SH^{-*}_{c=0}(V)\to H^{-*+n}(V,\p V)\simeq H_{*+n}(V)\to H_{*+n}(V,\p
V)\to SH_*^{c=0}(V),
$$
and therefore vanishes if $H_{*+n}(V)=0$ or $H_{*+n}(V,\p
V)\cong H^{n-*}(V)=0$. This is always the case if $*\ge n$ or $*\le
-n$. If $V$ is Stein, this holds if $*\neq 0$, and if $V$ is Stein
subcritical, this holds for all $*\in \Z$. The conclusion follows. 
\end{proof} 

\begin{corollary}[\cite{C}] \label{cor:Kai} 
If $V$ is Stein subcritical and $c_1(V)|_{\pi_2(V)}=0$,
then $SH_*(V)=0$.  
\end{corollary}

\begin{proof} Any compact set in a subcritical Stein
manifold is Hamiltonianly displaceable~\cite{BC}. Thus $V$ is
displaceable in $\widehat V$, and therefore $SH_*(V)=0$ by 
Corollary~\ref{cor:vanishingSH}. 
\end{proof}

\noindent {\bf Remark.} The original proof of Corollary~\ref{cor:Kai}
in~\cite{C} uses a handle decomposition for $W$. The proof given above
only uses the fact that the subcritical skeleton can be displaced from
itself~\cite{BC}. On the other hand, the proof given above uses the
grading in an essential way and hence only works under the hypothesis
$c_1(V)|_{\pi_2(V)}=0$, whereas the original proof does not need this
assumption.  

\begin{corollary}[Weinstein conjecture in displaceable manifolds]
  Assume that $V$ is Hamiltonianly displaceable 
in~$W$ and $c_1(W)|_{\pi_2(W)}=0$. Then any hypersurface of contact
type $\Sigma\subset V$ carries a closed characteristic. 
\end{corollary} 

\begin{proof}
This follows from the fact that $SH^n_{c=0}(V)=0$, as proved in
Corollary~\ref{cor:vanishingSH} above. In particular the canonical map
$SH^n_{c=0}(V)\to H^{2n}(V,\partial V)$ vanishes, and thus $V$
satisfies the Strong Algebraic Weinstein Conjecture in the sense of
Viterbo~\cite{V}. The conclusion is then a consequence of the Main
Theorem in~\cite{V} (see also~\cite[Theorem~4.10]{Osurvey} for details). 
\end{proof} 

We now turn to the computation of the Rabinowitz Floer homology groups for
cotangent bundles. Let $L$ be a connected closed Riemannian manifold, and let
$DT^*L\subset T^*L$ be the unit disc bundle with its canonical
symplectic structure. Note that $c_1(T^*L)=0$, so its symplectic
(co)homology and Rabinowitz Floer homology groups are $\Z$-graded. 
Given a free homotopy
class $c$ of loops in $L$, we denote by $\Lambda^c L$ the corresponding
connected component of the free loop space of $L$. Rabinowitz Floer
homology $RFH_*(DT^*L)$ decomposes as a direct sum of homology groups
$RFH_*^c(DT^*L)$ which only take into account loops in the class
$c$. We denote the free homotopy class of the constant loops by $c=0$. We
denote the Euler number of the cotangent bundle $T^*L\to L$ by
$e(T^*L)$ (if $L$ is non-orientable we work with $\Z/2$-coefficients).

\begin{theorem} \label{thm:cotangent}
In degrees $*\neq 0,1$ the Rabinowitz Floer homology of $DT^*L$ is given by 
$$
RFH_*(DT^*L) = 
\left\{\begin{array}{rl} 
H_*(\Lambda L), & *> 1, \\
H^{-*+1}(\Lambda L), & *< 0. 
\end{array} \right.
$$
In degree $0$ we have 
$$
RFH_0^c(DT^*L)=\left\{\begin{array}{ll}
H_0(\Lambda^c L)\oplus H^1(\Lambda^c L), & if \ c\neq 0,\\
H_0(\Lambda^0 L)\oplus H^1(\Lambda^0 L), & if \ c=0 \ and \ e(T^*L)= 0,\\
H^1(\Lambda^0 L), & if \ c=0 \ and \ e(T^*L)\neq 0.
\end{array}\right.
$$
In degree $1$ we have 
$$
RFH_1^c(DT^*L)=\left\{\begin{array}{ll}
H_1(\Lambda^c L)\oplus H^0(\Lambda^c L), & if \ c\neq 0,\\
H_1(\Lambda^0 L)\oplus H^0(\Lambda^0 L), & if \ c=0 \ and \ e(T^*L)= 0,\\
H_1(\Lambda^0 L), & if \ c=0 \ and \ e(T^*L)\neq 0.
\end{array}\right.
$$
\end{theorem}

The proof is based on the isomorphisms 
\begin{equation} \label{eq:SHcotangent}
SH_*^c(DT^*L)\simeq H_*(\Lambda^c L), \qquad SH^*_c(DT^*L)\simeq
H^*(\Lambda^c L), 
\end{equation}
proved in~\cite{Vcotangent,AS,SW}. In particular $SH_0^c(DT^*L)$ and
$SH^0_c(DT^*L)$ are isomorphic to the ground field. 
We also need the following Lemma.

\begin{lemma} \label{lem:deg0}
 The map $SH^0_c(DT^*L)\to SH_0^c(DT^*L)$ in the exact sequence of
Theorem~\ref{thm:main_exact} vanishes if $c\neq 0$, and
is multiplication by the Euler number $e(T^*L)$ if $c=0$. 
\end{lemma}

\begin{proof}
That the map $SH^0_c(DT^*L)\to SH_0^c(DT^*L)$ vanishes if 
$c\neq 0$ follows from the same argument as in
Corollary~\ref{cor:vanishingSH}. 

Let us focus on the map $SH^0_{c=0}(DT^*L)\to
SH_0^{c=0}(DT^*L)$. Modulo the isomorphisms with $H_0(\Lambda^0 L)$,
$H^0(\Lambda^0 L)$, the commutative diagram~\eqref{eq:map} becomes 
\begin{equation*} 
\xymatrix
@R=15pt
{
H^0(\Lambda^0 L) \ar[r] \ar[d] & H_0(\Lambda^0 L) \\
H^{n}(DT^*L,ST^*L) \ar[r] & H_{n}(DT^*L,ST^*L) \ar[u]
}
\end{equation*}
The vertical map on the left factors as 
$$
H^0(\Lambda^0L)\stackrel {incl^*}\longrightarrow H^0(L) \stackrel
\simeq \longrightarrow H^0(DT^*L) \stackrel
{\cup \tau} \longrightarrow H^n(DT^*L,ST^*L),
$$
where $incl^*$ is the isomorphism induced by the inclusion of constant
loops, and $\cup \tau$ is the Thom isomorphism given by cup-product
with the Thom class $\tau\in H^n(DT^*L,ST^*L)$ (see~\cite{V}). The
vertical map on the right factors as 
$$
H_n(DT^*L,ST^*L) \stackrel {\cap\tau} \longrightarrow H_0(DT^*L) \stackrel
{\simeq} \longrightarrow H_0(L)
\stackrel {incl_*} \longrightarrow H_0(\Lambda^0 L),
$$
where $\cap\tau$ is the isomorphism given by cap-product with the Thom
class, and $incl_*$ is the isomorphism induced by the inclusion of
constant loops. We also recall that the bottom map is the composition 
$$
H^n(DT^*L,ST^*L)\stackrel {PD} \longrightarrow H_n(DT^*L)\stackrel
{incl_*} \longrightarrow H_n(DT^*L,ST^*L). 
$$
The Poincar\'e dual of the Thom class is the fundamental class
$[0_L]$, and the evaluation of the Thom class on the fundamental class
is the Euler number $e(T^*L)$~\cite[Ch.VI,~\S11-12]{Bredon}. The
successive images of the generator 
$1\in H^0(\Lambda^0 L)$ via the maps described above are therefore
$$
1\mapsto \tau \mapsto [0_L] \mapsto e(T^*L),
$$
and the conclusion of the Lemma follows.
\end{proof}

\begin{proof}[Proof of Theorem~\ref{thm:cotangent}]
Inserting~\eqref{eq:SHcotangent} in the long exact
sequence~\eqref{eq:RFHseq} we obtain 
$$
\def\objectstyle{\scriptstyle}
\xymatrix
@C=10pt
{
\ldots \ar[r] &  \!H^{-*}(\Lambda^c L)\! \ar[r] & \!H_*(\Lambda^c L)\! \ar[r] &
\!RFH_*^c(DT^*L)\! \ar[r]
& \!SH^{-*+1}(\Lambda^cL)\! \ar[r] & \! SH_{*-1}(\Lambda^c L) \ar[r] & \ldots
}
$$
This immediately implies the result for $*\neq 0,1$. 
For small values of the
degree the above long exact sequence takes the form 
$$
\def\objectstyle{\scriptstyle}
\xymatrix
@C=7pt
{
0 \ar[r] & H_1(\Lambda^c L) \ar[r] &  RFH_1^c(DT^*L) \ar[r] &
H^0(\Lambda^c L) \ar[r] & H_0(\Lambda^c L) \ar[r] & RFH_0^c(DT^*
L) \ar[r] & H^1(\Lambda^c L) \ar[r] & 0.
}
$$
The middle map is given by Lemma~\ref{lem:deg0}, and the conclusion follows. 
\end{proof}

\begin{corollary} \label{cor:nonzero}
If $\dim L\geq 1$ we have 
$$
   RFH^{c=0}_*(DT^*L)\neq 0.
$$
\end{corollary}

\begin{proof} 
Denote by $p:\Lambda^0L\to L$ the evaluation at $t=0$ and by
$i:L\to\Lambda^0L$ the inclusion as constant loops. Since $p\circ
i=\id$, the induced map $i_*:H_*(L)\to H_*(\Lambda^0L)$ is injective. 
So $\Lambda^0L$ has nonvanishing homology in some positive degree
(take for example the image under $i_*$ of the fundamental class of
$L$), and the corollary follows from Theorem~\ref{thm:cotangent}. 
\end{proof}

\begin{remark}
If $L$ is simply connected the homology of $\Lambda^0L$, hence by
Theorem~\ref{thm:cotangent} also $RFH^{c=0}_*(DT^*L)$, is nontrivial
in an infinite number of degrees. This follows from Sullivan's minimal
model for $\Lambda L$, as explained for example in~\cite{VS}. 

If $L$ is not simply connected this need not be the case. For example,
if the universal cover of $L$ is contractible the inclusion
$i:L\to\Lambda^cL$ induces an isomorphism $i_*:H_*(L)\to
H_*(\Lambda^0L)$. On the other hand, $H_0(\Lambda^cL)$, and hence also
$RFH^c_*(DT^*L)$, is nonzero for each nontrivial free homotopy class
$c$.  
\end{remark}

Corollary~\ref{cor:nonzero} is used in~\cite{CFP} to study the
dynamics of magnetic flows. In order to apply it to exact contact
embeddings, we need the a criterion for independence of Rabinowitz
Floer homology of the symplectic filling $V$ given in the following
result.  

\begin{theorem} \label{thm:independence}
Let $V$ be an exact symplectic manifold of dimension $2n$ with
convex boundary $M=\p V$ and $c_1(V)|_{\pi_2(V)}=0$.  
Then $RFH_*(V)$ is independent of $V$ if $M$ admits a contact form for
which the closed characteristics $\gamma$ which are contractible in $V$
are nondegenerate and satisfy
$$
   CZ(\gamma)>3-n. 
$$
Here $CZ(\gamma)$ denotes the Conley-Zehnder index of $\gamma$ with
respect to the trivialization of $\gamma^*TV$ that extends over a
spanning disk in $V$. 
\end{theorem}

\begin{proof} Let $J$ be a time-independent cylindrical almost complex
  structure on $\widehat V$, as defined in Section~\ref{sec:RFH}. The virtual
  dimension of the moduli space of $J$-holomorphic planes in $\widehat
  V$ asymptotic to a closed Reeb 
  orbit $\gamma$ is $CZ(\gamma) +n-3$, so that our 
assumption guarantees that it is strictly positive. Thus, no rigid
holomorphic planes exist in $\widehat V$. Since the generators of the
  complex giving rise to $\Check{SH}_*(V)$ are located near $\p V$, it
  is a consequence of the stretch-of-the-neck argument
  in~\cite[\S5.2]{BOexact} that, in this situation, the symplectic
  homology groups $\Check{SH}_*(V)$ depend only on $\p V$. Using
  Theorem~\ref{thm:main1} we infer that the same is true for the
  Rabinowitz Floer homology groups $RFH_*(V)$. 
\end{proof}

\begin{corollary}\label{cor:independence}
Let $V$ be an exact symplectic manifold of dimension $2n$ with
convex boundary $M=\p V$ and $c_1(V)|_{\pi_2(V)}=0$.  
Assume in addition that the inclusion map $\iota:M\into V$ induces an
{\em injective} map $\iota_\#:\pi_1(M)\to\pi_1(V)$. 
Then $RFH_*(V)$ is independent of $V$ if $M$ admits a contact form for
which the closed characteristics $\gamma$ which are contractible in $M$
are nondegenerate and satisfy
$$
   CZ(\gamma)>3-n. 
$$
Here $CZ(\gamma)$ denotes the Conley-Zehnder index of $\gamma$ with
respect to the trivialization of $\gamma^*T(\R\times M)$ that extends
over a spanning disk in $M$. 
\end{corollary}

\begin{remark} \label{rmk:ST*L} 
Corollary~\ref{cor:independence} is the most useful consequence of
Theorem~\ref{thm:independence} because here the condition on the
Conley-Zehnder indices can be verified entirely in $M$.
For example, let $M=ST^*L=\{v\in T^*L \ : \ |v|=1\}$ be the unit
cosphere bundle in $T^*L$. Then the Conley-Zehnder indices (in $ST^*L$)
of contractible closed characteristics are equal to the Morse indices
of the underlying closed geodesics. Hence the condition
$CZ(\gamma)>3-n$ in Corollary~\ref{cor:independence} is satisfied if
either of the conditions below holds:  
\begin{itemize}
\item $\dim L \ge 4$;
\item $\dim L = 3$, and $L$ admits a nondegenerate metric such that
the Morse index of each contractible closed geodesic is at least $1$;
\item $\dim L =2$ and $L$  admits a nondegenerate metric such that the
Morse index of each contractible closed geodesic is at least $2$. 
\end{itemize}
For $\dim L=2$ the condition is satisfied for all closed surfaces
except $S^2$ and $\R P^2$. For $\dim L=3$ the condition holds e.g.~for
all manifolds which admit a metric of nonpositive sectional curvature,
as well as for the 3-sphere. 
\end{remark}

Let us call an embedding $\iota:M\into W$ {\em $\pi_1$-injective} if
the induced map $\iota_\#:\pi_1(M)\to\pi_1(W)$ is injective. Note that
this condition is automatically satisfied if $M$ is simply connected. 

\begin{theorem} \label{thm:ST*L}
Let $L$ be a closed Riemannian manifold satisfying one of the
conditions in Remark~\ref{rmk:ST*L}. 
Then any $\pi_1$-injective exact contact embedding of $ST^*L$ into
an exact symplectic manifold $W$ with $c_1(W)|_{\pi_2(W)}=0$ is
non-displaceable. In particular, the image of a $\pi_1$-injective
exact contact embedding of $ST^*L$ into the cotangent bundle of a
closed manifold must intersect every fiber. 
\end{theorem}

\begin{proof} Assume there exists a displaceable exact contact
embedding of $ST^*L$ into an exact symplectic manifold $W$ with
  $c_1(W)|_{\pi_2(W)}=0$, and denote
  by $V$ the bounded component with boundary $ST^*L$. Then
  $RFH_*(V)=0$ by Theorem~\ref{thm:vanishingRFH}. On the other hand,
  our assumptions on $L$ guarantee via Corollary~\ref{cor:independence}
  and Remark~\ref{rmk:ST*L} that $RFH_*(V)$ depends only on $\p
  V=ST^*L$. Therefore $RFH_*(V)\simeq RFH_*(ST^*L)$ and is nonzero by
  Corollary~\ref{cor:nonzero}, a contradiction. 
 
For the last assertion we use a result of
Biran, stating that
a compact set which avoids the critical coskeleton of a Stein
manifold is displaceable~\cite[Lemma~2.4.A]{Biran} (the critical
coskeleton is the union of the unstable manifolds of critical points
of index equal to half the dimension for an exhausting Morse plurisubharmonic
function). In a cotangent 
bundle the critical co-skeleton can be taken to be one given fiber,
and the result follows. This argument has already appeared in~\cite{CF}.    
\end{proof}

\begin{corollary} \label{cor:nonex}
  Let $L$ be a closed Riemannian manifold satisfying one of the
conditions in Remark~\ref{rmk:ST*L}.
Then $ST^*L$ does not admit any $\pi_1$-injective exact contact
  embedding into a subcritical Stein manifold $W$, or more generally
  into a stabilization $W=V\times \C$, with $c_1(W)|_{\pi_2(W)}=0$. 
\end{corollary}

\begin{proof} Any compact set in a subcritical Stein manifold is
  Hamiltonianly displaceable~\cite{BC}, and this also holds in a
  stabilization $V\times \C$. The conclusion then follows from
  Theorem~\ref{thm:ST*L}. 
\end{proof}

\begin{remark}
Let $L$ be a closed Riemannian manifold satisfying one of the
conditions in Remark~\ref{rmk:ST*L}. 
We explain in this
remark an alternative approach to proving Corollary~\ref{cor:nonex},
using the multiplicative structure in symplectic homology
investigated by McLean~\cite{McLean}. 

Assume there is a $\pi_1$-injective exact contact embedding
$ST^*L\hookrightarrow W$ into a subcritical Stein manifold $W$ with
$c_1(W)|_{\pi_2(W)}=0$, and let $V$ be the bounded component with
boundary $ST^*L$. The exact inclusion $f:V\hookrightarrow W$ induces a
transfer morphism~\cite{V}  
$$
f_!:SH_*(W)\to SH_*(V). 
$$ 
Symplectic homology carries a unital ring structure, with
multiplication given by the pair of pants product~\cite{Se}.
McLean showed that the transfer morphism $f_!$ is a 
{\it unital} ring homomorphism~\cite{McLean}. Since $W$ is subcritical
we have $SH_*(W)=0$~\cite{C}, so that the unit vanishes in $SH_*(W)$. Therefore 
$1=f_!(1)=f_!(0)=0\in SH_*(V)$ and we obtain as in~\cite{McLean} that
\begin{equation} \label{eq:vanishing} 
SH_*(V)=0.
\end{equation} 

Arguing as in Remark~\ref{rmk:ST*L} that there are no rigid
holomorphic planes in $\widehat V$, we deduce from the
stretch-of-the-neck argument in~\cite[\S5.2]{BOexact} that positive
symplectic homology $SH_*^+(V)$ defined in Section~\ref{sec:SH}
depends only on $\p V=ST^* L$, i.e. 
$$
SH_*^+(V)=SH_*^+(DT^*L)=H_*(\Lambda L,L). 
$$ 
Since $SH_*(V)=0$, it follows from the tautological exact
sequence~\eqref{eq:tautological} in Section~\ref{sec:SH} that
$$
H_*(\Lambda L,L)\simeq H_{*+n-1}(V,\p V).
$$ 

If $\pi_1(L)$ is finite, it follows from Sullivan's minimal model for the
free loop space that $H_*(\Lambda L)$ is supported in an infinite set
of degrees, hence the same holds for $H_*(\Lambda L,L)$, a contradiction. 

If $\pi_1(L)$ is infinite and contains an infinite number of conjugacy
classes, we see that $H_0(\Lambda L,L)$ is infinite dimensional, again
a contradiction. 

If $\pi_1(L)$ is infinite but contains only a finite number of conjugacy
classes, we still obtain a contradiction as follows. Note that this
situation can only arise if $\dim L\geq 3$ and thus $\pi_1(ST^*L)\cong
\pi_1(L)$. Pick a nontrivial conjugacy class $c\in\tilde\pi_1(L)$ and
denote by $d$ its image under the injective map
$i:\tilde\pi_1(L)\to\tilde\pi_1(ST^*L)\to \tilde\pi_1(V)$.   
Then $0=SH_*^d(V)\cong SH_*^{d,+}(V)\cong H_*(\Lambda^c L,L)$, but the
latter group is nonzero since $H_0(\Lambda^c L,L)=\Q$. 
\end{remark}

\begin{remark}
All our results remain true if one replaces the hypothesis
$c_1(V)|_{\pi_2(V)}=0$ by the stronger one $c_1(V)=0$  (and
likewise for $W$), and $\pi_1$-injectivity by the weaker assumption
that every contrctible loop in $V$ (resp.~$W$) is null-homologous in
$M$. E.g.~this assumption is automatically satisfied if
$H_1(M;\Z)=0$. For a unit cotangent bundle $M=ST^*L$ the conditions in
Remark~\ref{rmk:ST*L} then need to be replaced by the same conditions
on null-homologous instead of contractible geodesics. 
\end{remark}

The structure of the paper is the following. Section~\ref{sec:SH}
contains all the results concerned exclusively with symplectic
homology. We first recall its definition and main properties,
including the case of autonomous Hamiltonians~\cite{BO}, then
prove a Poincar\'e duality result for Floer homology and cohomology
(Proposition~\ref{prop:Poincare}). We study in 
Section~\ref{sec:HamsuppV} a version of symplectic
homology defined using Hamiltonians which vanish outside $V$, in the
spirit of~\cite{CFHW}. We define in Section~\ref{sec:Vshaped} the
groups $\Check{SH}_*(V)$ and prove Theorem~\ref{thm:main_exact},
Proposition~\ref{prop:map} and Proposition~\ref{prop:main2}. We discuss briefly
in Section~\ref{sec:independence} the fact that symplectic homology
does not depend on the ambient manifold $W$. Section~\ref{sec:RFH}
recalls the definition of Rabinowitz Floer homology. We prove that it
is also independent of the ambient manifold $W$ in
Proposition~\ref{prop:indep}. Sections~\ref{sec:perturbations}
and~\ref{sec:prep} are of a technical nature. We exhibit admissible
deformations of the defining data for Rabinowitz Floer homology, which 
are crucial for relating Rabinowitz Floer homology to symplectic
homology. Most technical work goes into deriving bounds on the
Lagrange multiplier in the Rabinowitz action functional, for which we
establish a maximum principle for a Kazdan-Warner type inequality in
Section~\ref{ss:KW}.   
Our main result, Theorem~\ref{thm:main1}, is proved in
Section~\ref{sec:proof}, and we give in the beginning of that section
a detailed outline.

\section{Symplectic homology} \label{sec:SH}

We use Viterbo's definition of symplectic homology
groups~\cite{V}. We follow the sign conventions in~\cite{C}, 
which match those in~\cite{CF}. We consider an exact
manifold $(V,\lambda)$ with symplectic form $\om:=d\lambda$ and
convex boundary $M=\p V$. 

That $M$ is convex means
that $\lambda|_M$ is a positive contact form when $M$ is oriented as
the boundary of $V$, or, equivalently, that the Liouville vector
field $X$ defined by $\lambda=\iota_X d\lambda$ points
outwards along $M$. We denote by $R_\lambda$ the
Reeb vector field on $M$ defined by $\iota_{R_\lambda} d\lambda|_M=0$ and
$\lambda(R_\lambda)=1$. The set of (positive) periods of closed Reeb
orbits is called the {\bf action spectrum} and is denoted by
$$
   \mathrm{Spec}(M,\lambda).
$$ 
Let $\phi_X^t$ be the flow of the Liouville vector field. We can embed
the negative symplectization of $M$ onto a neighbourhood of $M$ in $V$
by the map  
$$
\left((0,1]\times M,d(r\lambda)\right) \to
\left(V,\om\right), \qquad (r,x)\mapsto \phi_X^{\ln r}(x). 
$$
We denote by $\widehat V$ the symplectic completion of $V$, obtained
by attaching the positive symplectization $\left([1,\infty)\times
M,d(r\lambda)\right)$ along the boundary $M$ identified with
$\{1\}\times M$.  

\subsection{Sign and grading conventions}\label{ss:grading}
Given a Hamiltonian
$H_t:\wh V\to \R$, $t\in S^1=\R/\Z$ the
{\bf Hamiltonian vector field} $X^t_H$ is defined by 
$$
   dH_t=-i_{X^t_H}\om.
$$
An almost complex structure $J$ on $\wh V$ is {\bf $\om$-compatible} 
if $\langle\cdot ,\cdot \rangle \ := \ \om(\cdot,J\cdot)$ 
is a Riemannian metric. The gradient with respect to this
metric is related to the symplectic vector field by $X_H=J\nabla H$. 
The {\bf Hamiltonian action} of a loop $x:S^1\to\wh V$ is 
$$
   \cA_H(x) := \int_0^1x^*\lambda - \int_0^1H\bigl(t,x(t)\bigr)dt. 
$$
A {\bf positive} gradient flow line $u:\R\times S^1\to\wh V$ of
$\cA_H$ satisfies the perturbed Cauchy-Riemann equation
\begin{equation}\label{eq:CR}
   u_s+J(t,u)u_t+\nabla H(t,u) = u_s+J(t,u)\Bigl(u_t-X_H(t,u)\Bigr) =
   0.  
\end{equation}
The {\bf Conley-Zehnder index} $\CZ(x;\tau)\in\Z$ of a nondegenerate
1-periodic orbit $x$ of $X_H$ with respect to a symplectic
trivialization $\tau:x^*T\hat V\to S^1\times\R^{2n}$ is defined as
follows. The linearized Hamiltonian flow along $x$ defines via $\tau$
a path of symplectic matrices $\Phi_t$, $t\in[0,1]$, with $\Phi_0=\id$
and $\Phi_1$ not having $1$ in its spectrum. Then $\CZ(x;\tau)$ is the
Maslov index of the path $\Phi_t$ as defined in~\cite{RS,Sa}.
For a critical point $x$ of a $C^2$-small Morse function $H$ the
Conley-Zehnder index (with respect to the constant trivialization
$\tau$) is related to the Morse index by
\begin{equation}\label{eq:CZ-Morse}
   \CZ(x;\tau) = n - {\rm Morse}(x). 
\end{equation}
If $c_1(V)=0$ we define integer valued Conley-Zehnder indices of all
1-periodic orbits as follows. In each homology class $c\in H_1(V;\Z)$
we choose a loop $\gamma_c$ and a trivialization $\gamma_c^*TV\to
S^1\times\R^{2n}$. This induces trivializations of $T\hat V$ along all
1-periodic orbits $x$ by extension over a 2-chain connecting $x$ to
the reference loop $\gamma_c$ in its homology class and hence
well-defined Conley-Zehnder indices $\CZ(x)\in\Z$. 

If $c_1(V)|_{\pi_2(V)}=0$ we can still define integer valued
Conley-Zehnder indices for contractible 1-periodic orbits with respect
to trivializations that extend over spanning disks. 

Without any hypothesis on $c_1(V)$ we still have well-defined
Conley-Zehnder indices in $\Z_2$ and all the following results hold
with respect to this $\Z_2$-grading.

\subsection{Floer homology} Let $\cP(H)$ be the set of $1$-periodic
orbits of $X_H$. Given 
$x^\pm\in\cP(H)$ we denote by $\widehat \cM(x_-,x_+)$ the space of
solutions of~\eqref{eq:CR} with
$\lim_{s\to\pm\infty}u(s,t)=x_\pm(t)$. Its quotient by the $\R$-action
$s_0\cdot (s,t):=(s+s_0,t)$ on the cylinder is called the {\bf moduli
space of Floer trajectories} and is denoted by 
$$
   \cM(x_-,x_+) := \wh \cM(x_-,x_+)/\R.
$$
Assume now that all elements of $\cP(H)$ are nondegenerate and
contained in a compact set, and also that solutions of~\eqref{eq:CR}
are contained in a compact set. Assume further that the almost complex
structure $J=(J_t)$, $t\in S^1$ is generic, so that $\cM(x_-,x_+)$ is
a smooth manifold of dimension  
$$
\dim \, \cM(x_-,x_+) = \CZ(x^+) - \CZ(x^-) - 1. 
$$
For $k\in\Z$ and $a\in\R\cup\{\pm\infty\}$ the {\bf Floer chain
group} $CF_k^{<a}(H)$ is the $\Q$-vector space generated by the
1-periodic orbits of Conley-Zehnder index $k$ and action less that
$a$. We abbreviate $CF_k(H):=CF_k^\infty(H)$. The boundary operators
$\p:CF_k(H)\to CF_{k-1}(H)$ defined by 
$$
   \p x := \sum_{\CZ(y)=k-1}\#\MM(y,x)y
$$
decrease action and satisfy $\p^2=0$. (Note the reversed order of the
arguments in $\MM(y,x)$, which reflects the fact that we define
homology rather than cohomology.) So for $-\infty\leq
a<b\leq\infty$ they descend to boundary operators $\p^{(a,b)}$ on
$$
   CF^{(a,b)}(H) := CF^{<b}(H)/CF^{<a}(H)
$$  
which give rise to the {\bf filtered Floer homology groups} 
\footnote{\, Here and in the following we tacitly assume that the
  values $a,b,\dots$ are not in the action spectrum.}
$$
   FH_k^{(a,b)}(H) := \ker\p^{(a,b)}/\im \, \p^{(a,b)}. 
$$
For $a<b<c$ we have long exact filtration sequences
$$
   \ldots\to FH^{(a,b)}_*(H) \to FH^{(a,c)}_*(H) \to FH^{(b,c)}_*(H)
   \to FH^{(a,b)}_{*-1}(H) \to\ldots
$$

\subsection{Symplectic homology} Consider a time-independent
Hamiltonian on $(0,\infty)\times M$ of the form $H(r,x)=h(r)$. Then
$X_h=h'(r)R_\lambda$, so that $1$-periodic orbits of $X_h$ on level $r$
are in one-to-one correspondence with closed characteristics on $M$ of
period $h'(r)$. 

Let $\Ad(\wh V)$ be the class of admissible Hamiltonians $H$ which
satisfy $H\le 0$ on $V$, which have only 
nondegenerate $1$-periodic orbits, and which have the form
$H(r,x)=ar+b$ for $r$ large enough, with
$0<a\notin\mathrm{Spec}(M,\lambda)$ and $b\in \R$. The $1$-periodic
orbits of 
such a Hamiltonian are contained in a compact set and, if the almost
complex structure is invariant under homotheties at infinity, 
solutions of~\eqref{eq:CR} are also contained in a compact
set~\cite{Osurvey,V}.  

A monotone increasing homotopy $\wh H$ from $H_-$ to $H_+$ induces
chain maps 
$$
   \sigma_k^{(a,b)}(\wh H):CF^{(a,b)}(H_-)\to CF^{(a,b)}(H_+). 
$$
A standard argument shows that the induced maps $\sigma_k^{(a,b)}$ on
homology are independent of the chosen monotone homotopy $\hat H$.   
We introduce a partial order on Hamiltonians by saying $H\leq K$
iff $H(t,x)\leq K(t,x)$ for all $(t,x)\in S^1\times\wh V$. 
The Floer homologies $FH_k^{(a,b)}(H)$ of Hamiltonians $H\in\Ad(\wh
V)$ form a directed system via the maps $\sigma_k^{(a,b)}$. 
The {\bf symplectic homology groups} of $V$ are the direct limits as
$H\to\infty$ in $\Ad(\wh V)$, 
$$
   SH_k^{(a,b)}(V) := \underrightarrow{\lim} \ FH_k^{(a,b)}(H).
$$
We will be interested in the following groups:
\begin{align*}
   SH_k(V) &:= SH_k^{(-\infty,\infty)}(V), \cr
   SH_k^+(V) &:= \lim_{a\searrow 0}SH_k^{(a,\infty)}(V), \cr
   SH_k^-(V) &:= \lim_{b\searrow 0}SH_k^{(-\infty,b)}(V).
\end{align*}
Here the limits are to be understood as inverse limits
with respect to canonical maps $SH^{(a,b)}\to SH^{(a',b')}$ for $a\leq
a',b\leq b'$. The corresponding directed systems stabilize for $a$ 
respectively $b$ small enough. 
The groups $SH_*^\dagger(V)$, $\dagger=\emptyset,+,-$
are independent of the contact form on $M$. Equivalently, they
are invariant upon replacing $V$ by the subset in $\widehat V$ below
the graph of a function $f:M\to (0,\infty)$. 

For $a<b<c$ we have long exact filtration sequences
$$
   \ldots\to SH^{(a,b)}_*(V) \to SH^{(a,c)}_*(V) \to SH^{(b,c)}_*(V)
   \to SH^{(a,b)}_{*-1}(V) \to\ldots,
$$
hence in particular
\begin{equation} \label{eq:tautological} 
   \ldots\to SH^-_*(V) \to SH_*(V) \to SH^+_*(V)
   \to SH^-_{*-1}(V) \to\ldots
\end{equation} 
As a matter of fact, we have 
\begin{lemma}[\cite{V}]\label{lem:sing}
$$
   SH_*^-(V) \simeq H^{n-*}(V) \simeq H_{*+n}(V,\p V), \qquad
   n=\frac 1 2 \dim \, V. 
$$
\end{lemma}

\begin{proof}
Pick $\eps>0$ smaller than the action of all closed Reeb orbits on
$M$. Let $H\in\Ad(\wh V)$ be time-independent, a $C^2$-small Morse
function on $V$, and 
linear on $(1,\infty)\times M$. Then the only 1-periodic orbits with
action in $(-\infty,\eps)$ are
critical points in $V$, and Floer gradient flow lines $u(s,t)$ are
$t$-independent and satisfy the equation $u_s+\nabla H(u)=0$. Thus the
Floer chain complex agrees with the Morse cochain complex, with the
gradings related by equation~\ref{eq:CZ-Morse}. It follows that
$$
   FH^{(-\infty,\eps)}_*(H)\simeq H^{n-*}(V) \simeq H_{*+n}(V,\p V) 
$$
and the lemma follows by taking the direct limit over $H$, followed by
the limit $\eps\to 0$. 
\end{proof}

\subsection{Autonomous Hamiltonians} \label{sec:autonomous}
Floer or symplectic homology can be defined using autonomous (i.e.~
time-independent) Hamiltonians~\cite{BO}. Let $H:\wh
V\to \R$ be a Hamiltonian whose $1$-periodic orbits are either
constant and nondegenerate, denoted by $\gamma_\tp$ for
$\tp\in\mathrm{Crit}(H)$, or nonconstant and transversally
nondegenerate, denoted by $\gamma$. The geometric images of the latter
are circles $S_\gamma$, which we view as $1$-parameter families of
orbits via the correspondence $\gamma\mapsto \gamma(0)$. Assume
further that the nonconstant orbits appear in the region
$(0,\infty)\times M$, and $H=h(r)$ in their neighbourhood with
$h''(r)\neq 0$. 

Let us choose for each circle $S_\gamma$ a perfect Morse function
$f_\gamma:S_\gamma\to \R$ with two critical points $min$ and $Max$,
and denote by $\gamma_{min}$, $\gamma_{Max}$ the orbits starting at
these critical points. For $a>0$ the Floer chain groups are 
$$
CF_k^{<a}(H)=\bigoplus _{\tp\in\mathrm{Crit}(H), |\gamma_\tp|=k}
\langle \gamma_\tp \rangle \quad \oplus \quad \bigoplus 
_{S_\gamma, p\in\mathrm{Crit}(f_\gamma), |\gamma_p|=k} \langle
\gamma_p \rangle,
$$
with the direct sum running over orbits with action less than $a$. 
The degree $|\gamma_\tp|$ is given by the Conley-Zehnder index,
whereas the degree $|\gamma_p|$ for $p\in\mathrm{Crit}(f_\gamma)$ is
defined by~\cite[Lemma~3.4]{BO}
\begin{equation} \label{eq:degMB}
|\gamma_p|:= \left\{\begin{array}{ll} 
\CZ^\xi(\gamma)+ \frac 1 2 \big(1+sign(h''(r))\big), & p=min, \\
\CZ^\xi(\gamma)+\frac 1 2 \big((1-sign(h''(r))\big), & p=Max. \\
\end{array}\right. 
\end{equation} 
Here $\CZ^\xi(\gamma)$ is the Conley-Zehnder index of the linearized
Hamiltonian flow along $\gamma$ restricted to $\xi$, and
$sign(h''(r))=\pm 1$ is the sign of $h''(r)$ at the level $r$ on which
lives $\gamma$. The 
differential $\p:CF_k^{<a}(H)\to CF_{k-1}^{<a}(H)$ is given by 
$$
\p\gamma_p = \sum_{|\gamma'_{\widetilde q}|=|\gamma_p|-1}
\#\cM(\gamma'_{\widetilde q},\gamma_p)\gamma'_{\widetilde q} + 
\sum_{|\gamma'_q|=|\gamma_p|-1} \#\cM(\gamma'_q,\gamma_p)\gamma'_q.
$$
Here $\cM(\gamma'_{\widetilde q},\gamma_p)$ consists of rigid tuples
$(u_m,u_{m-1},\dots,u_1)$ whose components are solutions
of~\eqref{eq:CR} and satisfy  
\begin{itemize}
\item $\lim_{s\to-\infty} u_m = \gamma'_{\widetilde q}$,
and $\lim_{s\to+\infty} u_1 \in S_\gamma$ belongs to the stable
manifold $W^s(p,f_\gamma)$;
\item for $i=m-1,\dots,1$, the limit orbits
  $\lim_{s\to+\infty}u_{i+1}$ and $\lim_{s\to-\infty} u_i$ belong to
  the same $S_{\gamma_i}$ and are connected (in this order) by a
  positive flow line of $f_{\gamma_i}$.
\end{itemize} 
Similarly, $\cM(\gamma'_q,\gamma_p)$ consists of rigid tuples
$(u_m,u_{m-1},\dots,u_1)$ whose components are solutions
of~\eqref{eq:CR} and satisfy the same conditions as above, except the
first one which is replaced by the requirement that 
\begin{itemize}
\item $\lim_{s\to-\infty} u_m \in
S_{\gamma'}$ belongs to the unstable manifold $W^u(q,f_{\gamma'})$.
\end{itemize} 

To define the symplectic homology groups, one first 
perturbs $M$ inside $\wh V$ so that the closed characteristics are
transversally nondegenerate. The new class of admissible
Hamiltonians, denoted by $\Ad^0(\wh V)$, consists of functions
$H:\wh V\to \R$ which are strictly negative and $C^2$-small on $V$,
and which, on the region $\{r\ge 1\}$, are of the form 
$h(r)$ with $h$ linear at infinity of slope $0<a\notin
\mathrm{Spec}(M,\lambda)$, and $h$ strictly convex elsewhere. The
symplectic homology groups are the direct limits over $H\to\infty$ in
$\Ad^0(\wh V)$.

\subsection{Symplectic cohomology} 

Symplectic cohomology is defined by dualizing the homological chain
complex. More precisely, we denote by $CF^k_{>a}(H)$ the $\Q$-vector
space generated by the $1$-periodic orbits of $H$ of degree $k$ and
action bigger than $a$. In the case of time-dependent Hamiltonians
with nondegenerate $1$-periodic orbits the degree is given by the
Conley-Zehnder index, and the differential $\delta:CF^k_{>a}(H)\to
CF^{k+1}_{>a}(H)$ is 
$$
\delta x := \sum _{\CZ(y)=k+1} \# \cM(x,y) y.
$$
In the case of autonomous Hamiltonians, the degree is given
by~\eqref{eq:degMB} and the differential $\delta:CF^k_{>a}(H)\to 
CF^{k+1}_{>a}(H)$ is 
$$
\delta \gamma_p = \sum_{|\gamma'_q|=|\gamma_p|+1} \#\cM(\gamma_p,\gamma'_q)\gamma'_q,
$$
with $\cM(\gamma_p,\gamma'_q)$ having the same meaning as for
homology. We deduce quotient complexes $CF^k_{(a,b)}(H)$, Floer
cohomology groups $FH^k_{(a,b)}(H)$, and truncation maps
$FH_{(a',b')}^k\to FH_{(a,b)}^k$ for $a\le a'$ and $b\le b'$. 

The {\bf
symplectic cohomology groups} are defined as inverse limits for $H\to
\infty$ in $\mathrm{Ad}(\wh V)$, or $\mathrm{Ad}^0(\wh V)$,
$$
SH^k_{(a,b)}(V):=\underleftarrow{\lim} \, FH^k_{(a,b)}(H).
$$
The inverse limit is considered with respect to the continuation maps
$$
\sigma^k_{(a,b)}:CF^k_{(a,b)}(H_+)\to CF^k_{(a,b)}(H_-), \qquad H_-\le
H_+.
$$

\begin{proposition}[Poincar\'e duality] \label{prop:Poincare} 
For $-\infty\le a < b \le
\infty$ and $H\in \Ad^0(\wh V)$ there is a canonical isomorphism 
$$
PD: FH_k^{(a,b)}(H) \longrightarrow FH^{-k}_{(-b,-a)}(-H) .
$$
Given $H_-\le H_+$ these isomorphisms fit into a commutative diagram 
$$
\xymatrix
@C=50pt
{
FH_k^{(a,b)}(H_-) \ar[r]^{\sigma_k^{(a,b)}} \ar[d]_{PD}^\simeq &
FH_k^{(a,b)}(H_+) \ar[d]^{PD}_\simeq \\ 
FH^{-k}_{(-b,-a)}(-H_-) \ar[r]_{\sigma^{-k}_{(-b,-a)}} &
FH^{-k}_{(-b,-a)}(-H_+) 
}
$$
\end{proposition}

\begin{proof} 
Let $\oplus_{*\in\Z} CF_*^{(a,b)}(H,\{f_\gamma\},(J_t)_{t\in S^1})$ be
the homological Floer complex of $H$, with the collection of perfect
Morse functions $\{f_\gamma\}$ and the time-dependent almost complex
structure $J_t$, $t\in S^1=\R/\Z$. A Floer trajectory $u:\R\times
S^1\to\wh V$ satisfies equation~\eqref{eq:CR}, namely 
\begin{equation} \label{eq:CR1}
\p_s u + J_t(u)\p_t u - J_t(u) X_H(u)=0.
\end{equation} 
Define $v:\R\times S^1\to\wh V$ by $v(s,t):=u(-s,-t)$, so that
it satisfies the equation $\p_s v + J_{-t}(v)\p_t v + J_{-t}(v) X_H(v)=0$.
Denoting $\overline J_t:=J_{-t}$ this equation can be rewritten in the
form~\eqref{eq:CR}, namely 
\begin{equation*} \label{eq:CR2}
\p_s v + \overline J_t(v)\p_t v - \overline J_t(v) X_{-H}(v)=0.
\end{equation*}
The correspondence $u\leftrightarrow v$ therefore determines a canonical
identification 
$$
\oplus_{*\in\Z} CF_*^{(a,b)}(H,\{f_\gamma\},(J_t)_{t\in S^1}) \cong
\oplus_{*\in\Z} CF^*_{(-b,-a)}(-H,\{-f_\gamma\},(\overline J_t)_{t\in
  S^1}). 
$$
The homology groups do not depend on the choice of almost complex
structure, nor of auxiliary Morse functions $\{f_\gamma\}$, so that we
obtain  
$$
\oplus_{*\in\Z} FH_*^{(a,b)}(H) \simeq
\oplus_{*\in\Z} FH^*_{(-b,-a)}(-H). 
$$

This identification is clearly compatible with the
continuation morphisms, and the only issue is to identify the
change in grading under the correspondence
$(H,\{f_\gamma\})\mapsto(-H,\{-f_\gamma\})$. 

\begin{lemma} \label{lem:iRS} 
Let $\Phi:[0,1]\to \Sp(2n)$ be a continuous path
satisfying $\Phi(0)=\one$, and
denote $\tPhi(t):=\Phi(1-t)\Phi(1)^{-1}$. The Robbin-Salamon indices
of $\Phi$ and $\tPhi$ satisfy the relation
$$
i_{RS}(\widetilde \Phi) = -i_{RS}(\Phi). 
$$
\end{lemma} 
\begin{proof} Following~\cite[Proposition~2.2]{CFHW}, a path
$\chi(t)\psi(t)$, $t\in [0,1]$ is homotopic with fixed endpoints to
the catenation $\chi(t)\psi(0)$ and $\chi(1)\psi(t)$, so that 
$$
i_{RS}(\chi(t)\psi(t))=i_{RS}(\chi(t)\psi(0))+ i_{RS}(\chi(1)\psi(t)).
$$
Let $\Phi_-(t):=\Phi(1-t)$. Then $i_{RS}(\Phi_-)=-i_{RS}(\Phi)$
since the crossings of $\Phi_-$ are in one-to-one correspondence with
those of $\Phi$, and the 
crossing forms have opposite signatures~\cite{RS}. Since
$\Phi_-(t)\Phi_-^{-1}(t)=\one$, we obtain 
\begin{eqnarray*}
0&=& i_{RS}(\Phi_-(t)\Phi_-^{-1}(0)) + i_{RS}(\Phi_-(1)\Phi_-^{-1}(t))
\\
&=& i_{RS}(\Phi_-\Phi(1)^{-1}) + i_{RS}(\Phi_-^{-1}) \\
&=& i_{RS}(\tPhi) - i_{RS}(\Phi_-) \\
&=& i_{RS}(\tPhi) + i_{RS}(\Phi). 
\end{eqnarray*} 
The second equality uses that $\Phi_-(1)=\one$. The third equality
uses that, upon replacing a path with its inverse, 
the Robbin-Salamon index changes sign, which follows from the {\it
  (Homotopy)} and {\it (Catenation)} axioms in~\cite{RS}. 
\end{proof}

{\it Proof of Proposition~\ref{prop:Poincare} (continued).} Let 
$\varphi_t$ be the Hamiltonian flow of $H$. Given a periodic point
$x\in\wh V$ such that $\varphi_1(x)=x$, let $\Phi(t):=d\varphi_t(x)$. 
The flow of $-H$ is $\varphi_{-t}$ and we denote
$\tPhi(t):=d\varphi_{-t}(x)$. By differentiating the identity
$\varphi_{1-t}=\varphi_{-t}(\varphi_1(x))$ with respect to $t$ we obtain 
$d\varphi_{1-t}(x)
=d\varphi_{-t}(\varphi_1(x))d\varphi_1(x)=d\varphi_{-t}(x)d\varphi_1(x)$, 
so that $\widetilde \Phi(t)=\Phi(1-t)\Phi(1)^{-1}$. It follows from
Lemma~\ref{lem:iRS} that $i_{RS}(\tPhi)=-i_{RS}(\Phi)$. This proves in
particular that the grading changes sign at a nondegenerate critical
point of $H$. 

Let $\gamma_H$ be a nonconstant orbit of $H$, let $\gamma$ be the
underlying closed characteristic on $M=\partial V$, let
$\gamma_{-H}$ be the same orbit with reverse orientation, 
viewed as an orbit of $-H$, let $-\gamma$ be the underlying closed
characteristic with reverse orientation, and denote by $S_{\pm\gamma}$ the
circle of periodic orbits obtained by reparametrizing
$\gamma_{\pm H}$. Then
$$
i_{RS}(\gamma_H) = \CZ^\xi(\gamma)+\frac 1 2
$$
by~\cite[Lemma~3.4]{BO}, and from Lemma~\ref{lem:iRS} applied on
$\gamma_H$ or $\gamma$ we obtain
$$
i_{RS}(\gamma_{-H})=-\CZ^\xi(\gamma)-\frac 1 2. 
$$
We now perturb $H$ to $H+\delta f_\gamma$ for
some small $\delta>0$, and $-H$ to $-H-\delta f_\gamma$. It is proved
in~\cite[Proposition~2.2]{CFHW} that precisely two orbits survive in
$S_{\pm\gamma}$, 
corresponding to the critical points of $f_\gamma$, and we denote them
by $\gamma_{\pm H,p}$ for
$p\in\mathrm{Crit}(f_\gamma)$. Moreover, 
$$
\CZ(\gamma_{H,p})=\left\{\begin{array}{ll} 
i_{RS}(\gamma_H) + \frac 1 2, & p=min, \\
i_{RS}(\gamma_H) - \frac 1 2, & p=Max.
\end{array}\right.
$$
It follows that
$$
\CZ(\gamma_{-H,p})=\left\{\begin{array}{ll} 
i_{RS}(\gamma_{-H}) + \frac 1 2, & p=Max, \\
i_{RS}(\gamma_{-H}) - \frac 1 2, & p=min.
\end{array}\right.
$$
We obtain $\CZ(\gamma_{H,p})=-\CZ(\gamma_{-H,p})$ for any critical
point $p\in\mathrm{Crit}(f_\gamma)$.
Since the grading in the Morse-Bott description of Floer homology is
precisely the Conley-Zehnder index after perturbation, the conclusion
follows.  
\end{proof}

The following lemma is proved in the same way as Lemma~\ref{lem:sing}.
\begin{lemma}[\cite{V}]\label{lem:sing-coh}
For $\eps>0$ sufficiently small, 
$$
   SH^{-*}_{(-\infty,\eps)}(V) \simeq H_{*+n}(V) \simeq H^{-*+n}(V,\p
   V), \qquad n=\frac 1 2 \dim \, V. 
$$
\end{lemma}

\subsection{Hamiltonians supported in $V$} \label{sec:HamsuppV}

We explain in this section an alternative definition for symplectic
homology/cohomology, using Hamiltonians which are supported in $V$. 
We denote by $\Ad(\widehat V,V)$
the class of Hamiltonians $H:\widehat
V\to \R$ which vanish outside $V$, which satisfy $H\le 0$ on $V$, and
whose $1$-periodic orbits contained in the interior of $V$ are
nondegenerate. For technical simplicity, we also assume that $H$ is a
function of the first coordinate $r$ on a collar neighbourhood
$((1-\delta,1]\times M,d(r\lambda))$ of $M=\p V$. 
We introduce on $\Ad(\widehat V,V)$ an order $\preceq$ defined by 
$$
H\preceq K \quad \mbox{iff} \quad H(\theta,x)\ge K(\theta,x) \mbox{ for
all } (\theta,x)\in S^1\times \widehat V.
$$
Given $-\infty\le a<b\le\infty$ such that $|a|,|b|\notin
\mathrm{Spec}(M,\lambda)$, we define  
$$
\tSH_*^{(a,b)}(V):=\lim_{\stackrel \longleftarrow {\Ad(\widehat V,V)}} FH_{*,V}^{(a,b)}(H), 
\qquad 
\tSH^*_{(a,b)}(V):=\lim_{\stackrel \longrightarrow {\Ad(\widehat V,V)}} FH^{*,V}_{(a,b)}(H).
$$
The subscript/superscript $V$
for the Floer homology groups indicates
that we consider as generators only those $1$-periodic orbits which
are contained in the interior of $V$. 
Lemma~\ref{lem:lap1} below shows that the first coordinate $r$
satisfies the maximum principle along Floer cylinders (here we use
that $H$ is a function of $r$ near $\p V$). It follows that Floer
cylinders connecting orbits in the interior of $V$ cannot break at
constant orbits outside the interior, so these Floer
homology groups are well-defined. 
Moreover, the inverse/direct limits
are considered with respect to the order $\preceq$ on the space of
admissible Hamiltonians $\Ad(\widehat V,V)$.  

As before, we can also give the definition of $\tSH$ using the class
$\Ad^0(\widehat V,V)$ of autonomous Hamiltonians $H:\widehat V\to \R$
which vanish outside $V$, which satisfy $H\le 0$ on $V$, and whose
$1$-periodic orbits in $V$ are either constant and 
nondegenerate, or nonconstant and 
transversally nondegenerate. In this case we define the
Floer homology groups via the Morse-Bott construction
of Section~\ref{sec:autonomous}. 

\begin{proposition} \label{prop:alternative}
 For any $A>0$ such that $A\notin\mathrm{Spec}(M,\lambda)$, we have 
$$
\tSH_*^{(A,\infty)}(V)\simeq SH_*^{(-\infty,A)}(V), \qquad 
\tSH^*_{(A,\infty)}(V)\simeq SH^*_{(-\infty,A)}(V).
$$
\end{proposition}

\begin{proof}
We give the proof only for cohomology, the other case
being similar. We embed the symplectization $M\times
\R_+\hookrightarrow \widehat V$ and, to simplify the discussion, we
consider only autonomous Hamiltonians. We define a cofinal family in
$\Ad^0(\widehat V,V)$ consisting of Hamiltonians
$H=H_{\mu,\delta}:\widehat V\to \R$, $\mu>0$, $0<\delta \le 1$ which, up
to a smoothing, satisfy the following conditions:
\begin{itemize}
 \item $H_{\mu,\delta}=0$ on $\widehat V\setminus V$,
 \item $H_{\mu,\delta}(r,x)=\mu(r-1)$ on $[\delta,1]\times M$, 
 \item $H_{\mu,\delta}$ is a $C^2$-small Morse perturbation of the
constant function $\mu(\delta -1)$ on $\widehat V\setminus
[\delta,\infty)\times M$. 
\end{itemize} 

(this Hamiltonian coincides on $V$ with the Hamiltonian $K$ depicted in
Figure~\ref{fig:HK} on page~\pageref{fig:HK}). 
Given $\mu>0$ and $0<\delta \le 1$, we denote by $K_{\mu,\delta}$ the 
Hamiltonian which is equal to $\mu(r-\delta)$ on $[\delta,\infty)\times M$, and
which is a $C^2$-small Morse perturbation of the constant function $0$
on $\widehat V\setminus [\delta,\infty)\times M$. We denote
$K_\mu:=K_{\mu,1}$. 

Let $A>0$ be fixed as in the statement of the Proposition, and denote
by $\eta_A>0$ the distance to $\mathrm{Spec}(M,\lambda)$. Let us 
choose $0<\eps\le \eta_A/2$, and $0<\delta<\eps/(A+\eps)$. 
We claim the following sequence of isomorphisms
\begin{eqnarray} \label{eq:isos}
FH^*_{(A-\eps,\infty)}(H_{A,\delta}) & \simeq &
FH^*_{(\delta A-\eps,\infty)}(H_{A,\delta}+A(1-\delta)) \\
& \simeq & FH^*_{(\delta A-\eps,\infty)}(K_{A,\delta}) \nonumber \\
& \simeq & FH^*_{(\delta A-\eps,\infty)}(K_A) \nonumber \\
& \simeq & FH^*_{(-\infty,A)}(K_A). \nonumber
\end{eqnarray}
Let us first examine the $1$-periodic orbits of $H_{A,\delta}$. Note
that the action of a 1-periodic orbit on level $r$ of a Hamiltonian
$H(r,y)=h(r)$ is given by
$$
   \cA_H(r,y) = rh'(r)-h(r). 
$$
Denote by $T_0>0$ the minimal period of a closed Reeb orbit on $M$. 
An easy computation shows that the 1-periodic orbits of $H_{A,\delta}$
fall in four classes as follows:  
\begin{enumerate}
\item[(I)] constants in $\widehat V\setminus [\delta,\infty)\times M$,
with action close to $A-\delta A$,
\item[(II)] nonconstant orbits around $\{\delta\}\times M$, with
action in the interval $A(1-\delta)+[\delta
T_0,\delta(A-\eta_A)]=[A-\delta A+\delta T_0,A-\delta\eta_A]$,
\item[(III)] nonconstant orbits around $\{1\}\times M$, with action in
the interval $[T_0,A-\eta_A]$, 
\item[(IV)] constants in $[1,\infty)\times M$, with action $0$. 
\end{enumerate}
Under our assumption $0<\delta A<\eps<\eta_A/2$, the types of orbits
are ordered by action as 
$$
   IV < III < A-\eps < I < II. 
$$   
In particular, the Floer
complex $CF^*_{(A-\eps,\infty)}(H_{A,\delta})$ involves precisely the
orbits of Type~(I) and~(II). 

We can now explain the isomorphisms involved in~\eqref{eq:isos}. The
first isomorphism follows directly from the definitions (the
Hamiltonian and the action interval are simultaneously shifted by
a constant). The second 
isomorphism holds because the obvious increasing homotopy from
$H_{A,\delta}+A(1-\delta)$ to $K_{A,\delta}$ given by convex
combinations is such that the newly created orbits appear outside the
relevant action interval. The third isomorphism holds because one can
deform $K_{A,\delta}$ to $K_A=K_{A,1}$ through $K_{A,\sigma}$,
$\delta\le \sigma\le 1$, keeping the actions positive or very close to
zero. The last isomorphism holds because $K_A=K_{A,1}$ has no $1$-periodic
orbits with action outside the interval $(\delta A-\eps,A)$. 

To conclude, we notice now the sequence of isomorphisms 
\begin{eqnarray*}
 \tSH^*_{(A,\infty)}(V) &\simeq
 &FH^*_{(A,\infty)}(H_{A+\eps,\delta}) \\
 & \simeq & FH^*_{(A-\eps,\infty)}(H_{A,\delta}) \\
 & \simeq & FH^*_{(-\infty,A)}(K_A) \\
 & \simeq & SH^*_{(-\infty,A)}(V).
\end{eqnarray*}
We again use standard continuation arguments. For the first
isomorphism we deform $H_{A+\eps,\delta}$ within the cofinal class
of Hamiltonians of the form $H_{\mu,\sigma}$ such that $\mu\geq A$ and
the condition $\mu(1-\sigma)>A$ always holds. Note that for
$(\mu,\sigma)=(A+\eps,\delta)$ this condition holds due to our
assumption $\delta<\eps(A+\eps)$. During this deformation orbits of
types (I) and (II) always have action $>A$, those of type (IV) have
action $<A$, and new orbits of type (III) appear with action $>A$, so
$FH^*_{(A,\infty)}(H_{\mu,\sigma})$ does not change and converges to
$\tSH^*_{(A,\infty)}(V)$ as $(\mu,\sigma)\to(\infty,1)$. 
The second isomorphism follows by simultaneously shifting the
Hamiltonian and the action interval and the third isomorphism in
equation~\eqref{eq:isos}. For the fourth isomorphism we
use that, upon deforming within the cofinal class of Hamiltonians of
the form $K_\mu$, new orbits have action bigger than $A$. 
\end{proof}

\subsection{$\bigvee$-shaped Hamiltonians in $\wh V$} \label{sec:Vshaped}

In this section we assume that the closed characteristics on $M$ are
transversally nondegenerate. 
We consider the class $\Check\Ad^0(\wh V)$ of 
Hamiltonians $H:\wh V\to\R$ which satisfy the following conditions: 
\begin{itemize} 
\item The $1$-periodic orbits of $H$ are either constant or
  transversally nondegenerate,
\item $H\le 0$ in some tubular neighbourhood of $M\equiv\{1\}\times M$, and
  $H>0$ elsewhere (see Figure~\ref{fig:Vshape}),
\item $H=h(r)$ in the region $\{r\ge 1\}$, with $h(r)=ar+b$
outside a compact set, $0<a\notin\mathrm{Spec}(M,\lambda)$, 
$b\in\R$, and $h$ strictly convex in the region where it is not
linear. 
\end{itemize} 
We define $\Check{SH}_k(V)$ as follows, with limits over $H$ being
taken with respect to the usual partial order on $\Check{\Ad}^0(\wh
V)$. Given $-\infty < a < b < \infty$ we set 

\begin{equation} \label{eq:CheckSHab}
 \Check{SH}_k^{(a,b)}(V) := \lim_{\stackrel
\longrightarrow H} FH_k^{(a,b)}(H), 
\end{equation}

\begin{equation} \label{eq:CheckSH}
\Check{SH}_k(V) := \lim_{\stackrel \longrightarrow b} \lim_{\stackrel
\longleftarrow a} SH_k^{(a,b)}(V). 
\end{equation} 
The last two limits have to be understood as $a\to -\infty$ and $b\to
+\infty$.  We also define 
$$
\Check{SH}_k^{(-\infty,b)}(V) := \lim_{\stackrel
\longleftarrow a} SH_k^{(a,b)}(V), \qquad 
\Check{SH}_k^{(a,\infty)}(V) :=\lim_{\stackrel \longrightarrow b}
SH_k^{(a,b)}(V).  
$$

\begin{remark}
 Let $a<b$ be fixed. It follows from the proof of
Proposition~\ref{prop:SH_exact} below that, if $H$ is the Hamiltonian
in Figure~\ref{fig:Vshape} and the slope $\mu$ is much larger than
$\max \{|a|,|b|\}$, only orbits of types III-V are involved in the
computation of $FH_k^{(a,b)}(H)$. 
\end{remark}

\begin{remark} \label{rmk:exactness}
 We chose to define $\Check{SH}_k(V)$ by first using an inverse limit
and then a direct limit so that the inverse limit is applied to finite
dimensional vector spaces. In this case it is an
exact functor~\cite{ES}, so that truncated exact sequences pass to
the limit.  
\end{remark} 

\begin{remark}
 Whereas we have by definition $\displaystyle \Check{SH}_k(V)=\lim_{\stackrel
\longrightarrow b} \Check{SH}_k^{(-\infty,b)}(V)$, it is a priori not
true that $\displaystyle\Check{SH}_k(V)=\lim_{\stackrel \longleftarrow a}
\Check{SH}_k^{(a,\infty)}(V)$. The universal property of
direct/inverse limits only provides an arrow 
$$
\lim_{\stackrel \longrightarrow b} \lim_{\stackrel
\longleftarrow a} SH_k^{(a,b)}(V) \longrightarrow 
\lim_{\stackrel \longleftarrow a}\lim_{\stackrel \longrightarrow b}
SH_k^{(a,b)}(V).
$$
\end{remark} 

\begin{figure}[htp] 
\centering
\input{vshape.pstex_t}
\caption{A $\bigvee$-shaped Hamiltonian.} \label{fig:Vshape}
\end{figure} 

\begin{proposition} \label{prop:SH_exact}
For any $-\infty<a<0<b<\infty$ such that
$-a,b\notin\mathrm{Spec}(M,\lambda)$, there is a long exact sequence 
\begin{equation} \label{eq:longextrunc}
\def\objectstyle{\scriptstyle}
\xymatrix
@C=15pt
{
\ldots \ar[r] &  \!SH^{-*}_{(-\infty,-a)}(V)\! \ar[r] & \!SH_*^{(-\infty,b)}(V)\! \ar[r] &
\!\Check{SH}_*^{(a,b)}(V)\! \ar[r]
& \!SH^{-*+1}_{(-\infty,-a)}(V)\! \ar[r] & \ldots
}
\end{equation}
\end{proposition}

\begin{proof} 
  We consider a cofinal family in $\Check{\Ad}^0(\wh V)$ consisting of
Hamiltonians $H$ which, up to a smooth approximation, satisy the
following requirements (see Figure~\ref{fig:Vshape}): there exist
constants $\eps>0$, $0<\delta<1$ and
$0<\mu\notin\mathrm{Spec}(M,\lambda)$ such that  
\begin{itemize} 
   \item $H\equiv \mu(1-\delta)-\eps$ on $V\setminus
[\delta,1]\times M$; 
   \item $H=h(r)$ on $[\delta,+\infty)\times M$, where 
$$
\left\{\begin{array}{ll} 
h'(r)=-\mu, & \delta \le r \le 1, \\
h(1)= -\eps, & \\
h'(r)=\mu, & r\ge 1.
\end{array}\right.
$$
\end{itemize} 

Let $\eta_\mu >0$ be the distance between $\mu$ and
$\textrm{Spec}(M,\lambda)$, and let $T_0>0$ be the 
minimal period of a closed characteristic on $M$. The $1$-periodic
orbits of $H$ fall into five classes as follows.  
\renewcommand{\theenumi}{\Roman{enumi}}
\begin{enumerate} 
\item constants in $V\setminus [\delta,1]\times M$, with action
  $-\mu(1-\delta) +\eps = -\mu + \delta\mu+\eps$; 
\item nonconstant orbits in the neighbourhood of $\{\delta\}\times M$,
  corresponding to negatively parametrized closed characteristics on
  $M$ of period at most $\mu-\eta_\mu$, with action in the interval
  $[-\mu(1-\delta)+\eps 
  -\delta(\mu-\eta_\mu),-\mu(1-\delta)+\eps-\delta T_0]
=[-\mu +\eps+\delta\eta_\mu,-\mu +\eps+ \delta\mu-\delta T_0]$;
\item nonconstant orbits in the neighbourhood of $\{1\}\times M$ on
levels $r<1$, corresponding to negatively parametrized closed
characteristics on $M$ of period at most $\mu-\eta_\mu$, with action
in the interval  
  $[-(\mu-\eta_\mu)+\eps,-T_0+\eps]=[-\mu + \eta_\mu+\eps,
-T_0+\eps]$;  
\item constant orbits on $\{1\}\times M$, with action $\eps$; 
\item nonconstant orbits in the neighbourhood of $\{1\}\times M$ on
levels $r>1$, corresponding to positively parametrized closed
characteristics on $M$ of period at most $\mu-\eta_\mu$, with action in
the interval $[T_0+\eps,\mu-\eta_\mu+\eps]$. 
\end{enumerate}

Let $-\infty < a < 0< b < \infty$ be fixed, with
$-a,b\notin\mathrm{Spec}(M,\lambda)$. Let us choose 
$$
\mu\ge \max(|a|,|b|)+1,\qquad \delta\le \min(\eta_\mu/2\mu,1/3) \qquad
\eps\le\min(b,1/3,\eta_{|a|}/2).
$$
The condition $\delta\mu<\eta_\mu$ ensures that the above types of
orbits are ordered by the action as  
$$
   II < I < III_- < a < III_+ < IV < V_- < b < V_+. 
$$
Here the symbols $III_-$, $III_+$ stand for orbits of Type
III which have action smaller resp.~bigger than $a$, and $V_-$,
$V_+$ stand for orbits of Type V which have action smaller
resp.~bigger than $b$. We infer 
the short exact sequence of complexes 
$$
0 \to CF_*^{(-\infty,a)}(H)
\to CF_*^{(-\infty,b)}(H) \to CF_*^{(a,b)}(H) \to 0 
$$
which, in terms of the types of orbits involved, can be rewritten as
$$
0\to CF_*^{I,II,III_-} \to CF_*^{I-V_-} \to
CF_*^{III_+,IV,V_-} \to 0. 
$$
The associated long exact sequence has the form 
$$
\dots \to\!FH_*^{(-\infty,a)}(H)\!
\to \!FH_*^{(-\infty,b)}(H) \!\to\! FH_*^{(a,b)}(H) \!\to\!
FH_{*-1}^{(-\infty,a)}(H) \!\to \dots 
$$
We claim that its entries are isomorphic with the ones of the long
exact sequence in the statement of
Proposition~\ref{prop:SH_exact}. This implies the conclusion of the
proposition, since continuation maps are
compatible with truncation exact sequences and become isomorphisms
if $\mu\ge \max(|a|,|b|)+1$ and $\delta, \eps>0$ are sufficiently
small (depending on $\mu$). So it remains to prove the claim. 

$FH_*^{(-\infty,b)}(H)\simeq SH_*^{(-\infty,b)}(V)$: This
holds because $\mu>b>0$, and because the restriction to $V$ of the
Hamiltonian $H$ can be deformed to the constant Hamiltonian $-\eps$,
in such a way that the action of the newly created $1$-periodic orbits does not
cross the boundary of the action interval $(-\infty,b)$ during the deformation. 

$FH_*^{(a,b)}(H)\simeq \Check{SH}_*^{(a,b)}(V)$: This holds
because $\mu>\max(|a|,|b|)$, and because, upon increasing
the slope in a cofinal family of $\Check\Ad^0(\wh V)$, the action of
the newly created $1$-periodic orbits does not cross the boundary of the
action interval $(a,b)$ during the corresponding homotopies of
Hamiltonians.  

$FH_*^{(-\infty,a)}(H)\simeq SH^{-*}_{(-\infty,-a)}(V)$: To
prove this isomorphism, we denote by 
$H_{\mu,\delta}$ a Hamiltonian which, up to a
smoothing, satisfies the following conditions: 
\begin{itemize}
\item $H_{\mu,\delta}$ vanishes outside $V$;
\item $H_{\mu,\delta}(r,x)=-\mu(r-1)$ on $[\delta,1]\times
M$;
\item $H_{\mu,\delta}=\mu(1-\delta)$ on $\widehat
V\setminus [\delta,1]\times M$.
\end{itemize}
Our definition is such that $-H_{\mu,\delta}\in \Ad^0(\widehat
V,V)$ is as in the proof of Proposition~\ref{prop:alternative}. 
We claim the following sequence of isomorphisms.
\begin{eqnarray*}
FH_*^{(-\infty,a)}(H) & \simeq &
FH_*^{(-\infty,a)}(H_{\mu,\delta}-\eps) \\
& \simeq & FH_*^{(-\infty,a-\eps)}(H_{\mu,\delta})\\
& \simeq & FH_*^{(-\infty,a)}(H_{\mu,\delta})\\
& \simeq & FH^{-*}_{(-a,\infty)}(-H_{\mu,\delta}) \\
& \simeq & \widetilde{SH}^{-*}_{(-a,\infty)}(V) \\
& \simeq & SH^{-*}_{(-\infty,-a)}(V).
\end{eqnarray*} 
The first isomorphism holds is proved by a deformation argument: 
The Hamiltonian $H$ can be deformed outside $V$ via linear
Hamiltonians to 
the constant Hamiltonian $-\eps$, and the action of the newly created
$1$-periodic orbits does not cross the boundary of the action interval
$(-\infty,a)$. The second isomorphism holds trivially from the
definitions, and the third one holds because $|a|\notin
\mathrm{Spec}(M,\lambda)$ and $\eps<\eta_{|a|}/2$. The fourth
isomorphism is implied by Proposition~\ref{prop:Poincare}, the fifth
one holds by continuation because $\mu>|a|$, and the sixth one follows
from Proposition~\ref{prop:alternative}. 
\end{proof} 

\begin{proof}[Proof of Theorem~\ref{thm:main_exact}] 
It follows from the proof of Proposition~\ref{prop:SH_exact} that the
exact sequence~\eqref{eq:longextrunc} is compatible with the
morphisms induced by enlarging the action window, in the following
sense. Given $-\infty<a'<a<0<b<b'<\infty$ and a Hamiltonian $H$ as in
the proof of Proposition~\ref{prop:SH_exact} we have a commutative
diagram of short exact sequences of chain complexes 
\begin{equation}
\def\objectstyle{\scriptstyle}
\xymatrix
@C=15pt
{
0 & \ar[r] &  \!CF_*^{(-\infty,-a')}(H)\! \ar[r] \ar[d] &
\!CF_*^{(-\infty,b)}(H)\! \ar[r] \ar@{=}[d]& 
\!CF_*^{(a',b)}(H)\! \ar[r] \ar[d]
& 0 \\
0 & \ar[r] &  \!CF_*^{(-\infty,-a)}(H)\! \ar[r] \ar@{=}[d] &
\!CF_*^{(-\infty,b)}(H)\! \ar[r] \ar[d] &
\!CF_*^{(a,b)}(H)\! \ar[r] \ar[d] 
& 0 \\
0 & \ar[r] &  \!CF_*^{(-\infty,-a)}(H)\! \ar[r] & 
\!CF_*^{(-\infty,b')}(H)\! \ar[r] &
\!CF_*^{(a,b')}(H)\! \ar[r]
& 0
}
\end{equation}
Passing to Floer homologies and using the isomorphisms in the proof of
Proposition~\ref{prop:SH_exact}, we obtain a commutative diagram of
long exact sequences 
\begin{equation} \label{eq:big_diag}
\def\objectstyle{\scriptstyle}
\xymatrix
@C=15pt
{
\ldots \ar[r] &  \!SH^{-*}_{(-\infty,-a')}(V)\! \ar[r] \ar[d] &
\!SH_*^{(-\infty,b)}(V)\! \ar[r] \ar@{=}[d]& 
\!\Check{SH}_*^{(a',b)}(V)\! \ar[r] \ar[d]
& \!SH^{-*+1}_{(-\infty,-a')}(V)\! \ar[r] \ar[d] & \ldots \\
\ldots \ar[r] &  \!SH^{-*}_{(-\infty,-a)}(V)\! \ar[r] \ar@{=}[d] &
\!SH_*^{(-\infty,b)}(V)\! \ar[r] \ar[d] &
\!\Check{SH}_*^{(a,b)}(V)\! \ar[r] \ar[d] 
& \!SH^{-*+1}_{(-\infty,-a)}(V)\! \ar[r] \ar@{=}[d] & \ldots \\
\ldots \ar[r] &  \!SH^{-*}_{(-\infty,-a)}(V)\! \ar[r] & \!SH_*^{(-\infty,b')}(V)\! \ar[r] &
\!\Check{SH}_*^{(a,b')}(V)\! \ar[r]
& \!SH^{-*+1}_{(-\infty,-a)}(V)\! \ar[r] & \ldots
}
\end{equation}
Using Remark~\ref{rmk:exactness} and passing first to the inverse
limit as $a\to-\infty$, and then to the direct
limit as $b\to \infty$, we obtain the conclusion of
Theorem~\ref{thm:main_exact}.  
\end{proof}

\begin{proof}[Proof of Proposition~\ref{prop:map}]
 We claim that, for any $-\infty<a'<a<0<b<b'<\infty$ such that
$-a',-a,b,b'\notin\mathrm{Spec}(M,\lambda)$, there is a commutative
diagram  
\begin{equation} \label{eq:cd}
\xymatrix
@C=50pt
@R=15pt
{SH^{-*}_{(-\infty,-a')}(V) \ar[r] \ar[d] & SH_*^{(-\infty,b')}(V) \\
SH^{-*}_{(-\infty,-a)}(V) \ar[r] & SH_*^{(-\infty,b)}(V) \ar[u]
}
\end{equation}
in which the vertical maps are the continuation morphisms, and the
horizontal maps are the ones appearing in the long exact sequence of
Proposition~\ref{prop:SH_exact} with $a=-b$. This follows from the
commutative diagram below, which is obtained by rearranging the
leftmost commutative squares in~\eqref{eq:big_diag}:
$$
\xymatrix
@C=50pt
@R=15pt
{SH^{-*}_{(-\infty,-a')}(V) \ar[r] \ar[d] & SH_*^{(-\infty,b')}(V) \ar@{=}[d]
\\
SH^{-*}_{(-\infty,-a)}(V) \ar[r] \ar@{=}[d] & SH_*^{(-\infty,b')}(V) \\
SH^{-*}_{(-\infty,-a)}(V) \ar[r] & SH_*^{(-\infty,b)}(V) \ar[u]
}
$$
Choosing $b=-a=\rho>0$ small enough in~\eqref{eq:cd} and passing first
to the inverse limit as $a'\to-\infty$ and then to the 
direct limit as $b'\to\infty$, we obtain a commutative diagram 
$$
\xymatrix
@C=50pt
@R=15pt
{SH^{-*}(V) \ar[r] \ar[d] & SH_*(V) \\
SH^{-*}_{(-\infty,\rho)}(V) \ar[r] & SH_*^{(-\infty,\rho)}(V) \ar[u]
}
$$
By Lemma~\ref{lem:sing} and Lemma~\ref{lem:sing-coh}, 
the bottom entries of this diagram are
$SH^{-*}_{(-\infty,\rho)}(V)\simeq H^{-*+n}(V,\partial V)$ and 
$SH_*^{(-\infty,\rho)}(V)\simeq H_{*+n}(V,\partial V)$. Moreover, it
follows from the proof of Proposition~\ref{prop:main2} below that the
bottom map is the composition $H^{-*+n}(V,\partial V)\stackrel {PD}
\longrightarrow H_{*+n}(V) \stackrel {incl_*} \longrightarrow
H_{*+n}(V,\partial V)$ of the map induced by 
inclusion $V\hookrightarrow (V,\partial V)$ with the Poincar\'e
duality map. 
\end{proof}

We define two variants of the symplectic homology groups
$\Check{SH}_*(V)$, namely
\begin{eqnarray}
\label{eq:CheckSH+}
\Check{SH}_k^{\ge 0}(V) & := & \lim_{a \nearrow 0}
\Check{SH}_k^{(a,\infty)}(V), \\ 
\label{eq:CheckSH-}
\Check{SH}_k^{\le 0}(V) & := & \lim_{b \searrow 0} 
\Check{SH}_k^{(-\infty,b)}(V).
\end{eqnarray}

\begin{proof}[Proof of Proposition~\ref{prop:main2}]
The two diagrams in the statement of Proposition~\ref{prop:main2}
follow by specializing the commutative diagram~\eqref{eq:big_diag}. 
Let us choose $0<\rho<\min\, \mathrm{Spec}(M,\lambda)$. We set
$a=-\rho$, $b=\rho$ in~\eqref{eq:big_diag} and
let $a'\to-\infty$, $b'\to\infty$ to obtain the commutative diagram of
long exact sequences 
\bigbreak
\begin{equation} \label{eq:proof_main2_diag1} 
\def\objectstyle{\scriptstyle}
\xymatrix
@C=15pt
{
\ldots \ar[r] &  \!SH^{-*}(V)\! \ar[r] \ar[d] &
\!SH_*^{(-\infty,\rho)}(V)\! \ar[r] \ar@{=}[d]& 
\!\Check{SH}_*^{\le 0}(V)\! \ar[r] \ar[d]
& \!SH^{-*+1}(V)\! \ar[r] \ar[d] & \ldots \\
\ldots \ar[r] &  \!SH^{-*}_{(-\infty,\rho)}(V)\! \ar[r] \ar@{=}[d] &
\!SH_*^{(-\infty,\rho)}(V)\! \ar[r] \ar[d]& 
\!\Check{SH}_*^{(-\rho,\rho)}(V)\! \ar[r] \ar[d]
& \!SH^{-*+1}_{(-\infty,\rho)}(V)\! \ar[r] \ar@{=}[d] & \ldots \\
\ldots \ar[r] &  \!SH^{-*}_{(-\infty,\rho)}(V)\! \ar[r] &
\!SH_*(V)\! \ar[r] &
\!\Check{SH}_*^{\ge 0}(V)\! \ar[r] 
& \!SH^{-*+1}_{(-\infty,\rho)}(V)\! \ar[r] & \ldots
}
\end{equation}
Since $SH^{-*}_{(-\infty,\rho)}(V)\simeq
H^{-*+n}(V,M)$ and $SH_*^{(-\infty,\rho)}(V)\simeq
H_{*+n}(V,M)$ (cf. Lemma~\ref{lem:sing} and Lemma~\ref{lem:sing-coh}), the
top and bottom long exact  
sequences in~\eqref{eq:proof_main2_diag1} are the bottom exact
sequences in the diagrams of Proposition~\ref{prop:main2}. To prove the
proposition, we need to show that the middle exact sequence
in~\eqref{eq:proof_main2_diag1} is isomorphic to the homological
(resp.~cohomological) long
exact sequence of the pair $(V,M)$. This essentially follows
from~\cite[Proposition~4.45]{Schwarz}, as we explain now.

For our choice of parameters
$a$ and $b$, this last exact sequence arises by truncating
the range of the action such that only orbits of Type I-IV
for a Hamiltonian $H$ as in Figure~\ref{fig:Vshape} are taken into 
account. 
(Here we take $\eps<\rho$ for the parameter $\eps$ in
the definition of $H$ and the constant $\rho$ above). 
Moreover, with the notation in the proof of
Proposition~\ref{prop:SH_exact}, we have $III_-=III$ and
$III_+=\emptyset$. A deformation argument shows that it is enough to
consider such a Hamiltonian with slope $\mu=\rho$, and for which
$II=III=\emptyset$. Without loss of
generality we can further diminish $\rho$, and assume that $H$
is small enough in $C^2$-norm. Because $V$ is symplectically
aspherical, the Floer 
complex reduces to the Morse complex~\cite[Theorem~6.1]{HS}(see
also~\cite[Theorem~7.3]{SZ}). Our definition of the Floer differential
is such that we consider the Morse complex for the \emph{positive}
gradient vector field $\nabla H$. Equivalently, we are considering
Morse homology for the \emph{negative} gradient vector field $-\nabla
(-H)$ (see Figure~\ref{fig:Vshape_Morse}).  
\begin{figure}[htp] 
\centering
\includegraphics{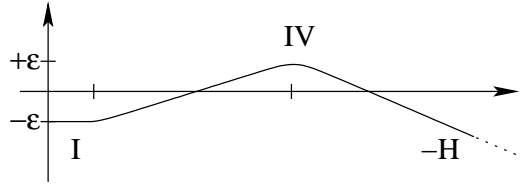}
\caption{Hamiltonian with small slope.} \label{fig:Vshape_Morse}
\end{figure} 
The middle exact sequence
in~\eqref{eq:proof_main2_diag1} is associated to the short exact
sequence of Morse complexes  
\begin{equation} \label{eq:short_Morse} 
0\to C_*^{(-2\rho, 0)}(-H) \to C_*^{(-2\rho,2\rho)}(-H)\to
C_*^{(0,2\rho)}(-H)\to 0. 
\end{equation}
Denote $\widehat V^r:=\{p\in \widehat V \, : \,
-H(p)\le r\}$, $r\in\R$. By~\cite[Prop.~4.45]{Schwarz}
the homology long
exact sequence associated to~\eqref{eq:short_Morse} 
is isomorphic to the long exact sequence of the triple
$(\widehat V^{2\rho},\widehat V^0, \widehat V^{-2\rho})$. By excision
and Poincar\'e duality, the latter is isomorphic to the long exact
sequence of the pair $(V,M)$. 
\end{proof}

\subsection{Definitions in $\wh W$ and dependence only on $\wh V$}
\label{sec:independence} 

Let us now assume that $M\hookrightarrow W$ is an exact contact
embedding of $(M,\xi)$ into the convex exact manifold
$(W,\lambda)$. We denote the bounded component of $W\setminus
M$ by $V$. 

Denote by $\Ad^0(\wh V;\wh W)$ the class of Hamiltonians
$H:\wh W\to \R$ satisfying 
\begin{itemize}
 \item $H\le 0$ on $V$, 
and $H=\mathrm{const}\geq 0$ ouside a compact set,
 \item the periodic orbits of $H$ other than constants at infinity
are transversally nondegenerate if nonconstant, and nondegenerate if
constant. 
\end{itemize}
We define the {\bf symplectic homology groups of $V$ in $W$} by
\begin{align*}
   SH_k(V;W) &:= \lim_{a\nearrow 0}\lim_H FH_k^{(a,\infty)}(H), \cr
   SH_k^+(V;W) &:= \lim_{a\searrow 0}\lim_H FH_k^{(a,\infty)}(H), \cr
   SH_k^-(V;W) &:= \lim_{a\nearrow 0}\lim_{b\searrow 0} \lim_H FH_k^{(a,b)}(H).
\end{align*}

\begin{proposition} \label{prop:SHinW}
 We have $SH_*^\dagger(V;W)\simeq SH_*^\dagger(V)$, for
$\dagger=\emptyset,+,-$. 
\end{proposition}

\begin{proof}
The main ingredients of the proof are the maximum
principle and Gromov's monotonicity principle for the area of
pseudo-holomorphic curves (see also~\cite[Lemma~1]{O}). The arguments
are very similar to the ones in the proof of
Proposition~\ref{prop:indep} below, and we leave the details to the
reader. 
\end{proof} 

Similarly to $\Ad^0(\wh V;\wh W)$, we define the class
$\Check\Ad^0(\wh V;\wh W)$ of admissible Hamiltonians $H:\wh W\to\R$
by requiring the following conditions:  
\begin{itemize} 
\item $H$ coincides with an admisible Hamiltonian
$H'\in\Check{\Ad}^0(\wh V)$ on $V\cup [1,R)\times M$ for some $R>1$,
\item $H=\mathrm{const}$ approximately equal to $a(R-1)$ on $\wh
W\setminus (V\cup[1,R)\times M)$. 
\end{itemize} 
We define the groups $\Check{SH}_k^\dagger(V;W)$,
$\dagger=\emptyset,\ge 0, \le 0$ by
formulas~\eqref{eq:CheckSH}, \eqref{eq:CheckSH+}, \eqref{eq:CheckSH-}
using direct limits over 
$H\to\infty$ in $\Check\Ad^0(\wh V;\wh W)$. We then have 
$$
\Check{SH}_k^\dagger(V;W)\simeq \Check{SH}_k^\dagger(V), \qquad
\dagger= \emptyset, \dagger =``\ge 0\mbox{''}, \mbox{or }\dagger =``\le 0 \mbox{''}. 
$$
Moreover, the statements of Theorem~\ref{thm:main_exact}
and Proposition~\ref{prop:main2} remain valid.

\section{Rabinowitz Floer homology} \label{sec:RFH}

Let us recall from~\cite{CF} the definition of Rabinowitz Floer
homology. 

On an exact symplectic manifold $(W,\lambda)$ with symplectic form
$\om=d\lambda$ define the {\bf Liouville vector field} $X$ by
$i_{X}\om=\lambda$. We say that $(W,\lambda)$ is {\bf complete
  and convex} if the following conditions hold:
\begin{itemize}
\item There exists a compact subset $K\subset W$ with smooth boundary
  such that $X$ points out of $K$ along $\p K$. 
\item The vector field $X$ is complete and has no critical
  points outside $K$. 
\end{itemize}
(This includes the condition of ``bounded topology'' in~\cite{CF}). 
Equivalently, $(W,\lambda)$ is complete and convex iff there exists an
embedding $\phi:N\times[1,\infty)\to W$ such 
that $\phi^*\lambda=r\alpha_N$, where $r$ denotes the coordinate on
$[1,\infty)$ and $\alpha_N$ is a contact form, and
such that $W\setminus\phi\bigl(N\times(1,\infty)\bigr)$ is
compact. (To see this, simply apply the flow of $X$ to $N:=\p
K$. cf.~\cite{CF}). 

Consider now a complete convex exact symplectic manifold $(W,\lambda)$
and a compact subset $V\subset W$ with smooth boundary $M=\p V$ such
that $\lambda|_M$ is a positive contact form with Reeb vector field
$R$. 
We abbreviate by $\mathcal{L}=C^\infty(S^1,W)$ the free loop space
of $W$. A {\bf defining Hamiltonian} for $M$ is a smooth function
$H:W\to\R$ with regular level set $M=H^{-1}(0)$ whose Hamiltonian
vector field $X_H$ (defined by $i_{X_H}\om=-dH$) has compact support
and agrees with $R$ along $M$. Given such a Hamiltonian, the {\bf
Rabinowitz action functional} is defined by  
$$
   A_H: \mathcal{L} \times \R \to \R,
$$
$$
   A_H(x,\eta):= \int_0^1 x^*\lambda - \eta \int_0^1 H(x(t))dt.
$$
Critical points of $A_H$ are solutions of the equations
\begin{equation}\label{crit1}
\left. \begin{array}{cc}
\partial_t x(t)= \eta X_H(x(t)),
& t \in \mathbb{R}/\mathbb{Z}, \\
\int_0^1H(x(t))dt=0. & \\
\end{array}
\right\}
\end{equation}
By the first equation $H$ is constant along $x$, so the second
equation implies $H(x(t))\equiv 0$. Since $X_H=R$ along $\Sigma$, 
the equations (\ref{crit1}) are equivalent to
\begin{equation}\label{crit2}
\left. \begin{array}{cc}
\partial_t x(t)= \eta R(v(t)),
& t \in \mathbb{R}/\mathbb{Z}, \\
x(t)\in \Sigma, & t \in \mathbb{R}/\mathbb{Z}. \\
\end{array}
\right\}
\end{equation}
So there are three types of critical points: closed Reeb orbits on $M$
which are positively parametrized 
and correspond to $\eta>0$, closed Reeb orbits on $M$ which
are negatively parametrized and correspond to $\eta<0$, and constant
loops on $M$ which correspond to $\eta=0$. The action of a critical
point $(x,\eta)$ is $A_H(x,\eta)=\eta$. 

A compatible almost complex structure $J$ on (part of) the symplectization
$\bigl(N\times\R_+,d(r\alpha_N)\bigr)$ of a contact manifold $(N,\alpha_N)$
is called {\bf cylindrical} if it satisfies:
\begin{itemize}
\item $J$ maps the Liouville vector field $r\p_r$ to the Reeb vector
  field $R$;
\item $J$ preserves the contact distribution $\ker\alpha_N$;
\item $J$ is invariant under the Liouville flow
  $(y,r)\mapsto(y,e^tr)$, $t\in\R$. 
\end{itemize}
A compatible almost complex structure $J$ on a complete convex exact
symplectic manifold $(W,\lambda)$ is called cylindrical if $\phi^*J$
is cylindrical on the collar
$\bigl(N\times[1,\infty),d(r\alpha_N)\bigr)$ at infinity.  
For a smooth family $(J_t)_{t\in S^1}$ of cylindrical almost complex
structures on $(W,\lambda)$ we consider the following metric $g=g_J$ on
$\mathcal{L}\times \mathbb{R}$. Given a point 
$(x,\eta) \in \mathcal{L}\times \mathbb{R}$
and two tangent vectors $(\hat{x}_1,\hat{\eta}_1),(\hat{x}_2,\hat{\eta}_2)
\in T_{(x,\eta)}(\mathcal{L}\times \mathbb{R})=\Gamma(S^1,x^*TW)
\times \mathbb{R}$ the metric is given by
$$
   g_{(x,\eta)}\big((\hat{x}_1,\hat{\eta}_1),(\hat{x}_2,\hat{\eta}_2)\big)
=\int_0^1\omega\big(\hat{x}_1(t),J_t(x(t))\hat{x}_2(t)\big)dt+
\hat{\eta}_1\cdot\hat{\eta}_2.
$$
The gradient of the Rabinowitz action functional $A_{H}$ with respect to the
metric $g_J$ at a point $(x,\eta) \in \mathcal{L}\times \mathbb{R}$ reads
$$
   \nabla A_H(x,\eta) = \nabla_J A_{H}(x,\eta)=
   \left(\begin{array}{c}
   -J_t(x)\big(\partial_t x-\eta X_{H}(x)\big)\\
   -\int_0^1H(x(t))dt.
   \end{array}\right)
$$
Hence (positive) gradient flow lines are solutions 
$(x,\eta)\in C^\infty(\mathbb{R}\times S^1,\widehat{V}) \times
C^\infty(\mathbb{R},\mathbb{R})$ of the partial differential equation
\begin{equation}\label{grad1}
\left.\begin{array}{c}
\partial_s x+J_t(x)\big(\partial_t x-\eta X_{H}(x)\big)=0\\
\partial_s\eta + \int_0^1H(x(t))dt=0.
\end{array}\right\}
\end{equation}
It is shown in~\cite{CF} that for $-\infty<a<b\leq\infty$ the
resulting truncated Floer homology groups 
$$
   RFH^{(a,b)}(M,W):=FH^{(a,b)}(A_H,J),
$$
corresponding to action values in $(a,b)$, are well-defined and 
do not depend on the choice of cylindrical $J$ and defining
Hamiltonian $H$. The {\bf Rabinowitz Floer homology} of $(M,W)$ is
defined as the limit 
$$
   RFH_*(M,W) := \lim_{\stackrel \longrightarrow \mu} \lim_{\stackrel
   \longleftarrow \lambda} RFH_*^{(-\lambda,\mu)}(M,W), \qquad
   \lambda,\mu\to\infty.
$$
By~\cite[Theorem~A]{CF-morse}, this definition is equivalent to the
original one in~\cite{CF}. 

\subsection{Independence of the ambient manifold}

Our first new observation on Rabinowitz Floer homology is

\begin{proposition}\label{prop:indep}
The Floer homology groups $RFH^{(a,b)}(M,W)$ for $-\infty<a<b<\infty$
depend only on the exact symplectic manifold $(V,\lambda)$ and not on
the ambient manifold $W$. 
\end{proposition}

\begin{proof}
Since the Liouville vector field $X$ is complete, its flow defines an embedding
$\psi:M\times\R_+\into W$ of the symplectization of
$(M,\lambda_M:=\lambda|_M)$ such that $\psi^*\lambda=r\lambda_M$
(see~\cite{CF}). Pick a cylindrical almost complex structure $J_M$ on
$M\times\R_+$. By Gromov's Monotonicity
Lemma~\cite[Proposition~4.3.1]{Sik}, there exists an $\eps>0$ 
such that every $J_M$-holomorphic curve in $M\times\R_+$ which meets the
level $M\times\{3\}$ and exits the set $M\times[2,4]$ has symplectic
area at least $\eps$. Rescaling by $R>1$, it follows that every
$J_M$-holomorphic curve which meets the 
level $M\times\{3R\}$ and exits the set $M\times[2R,4R]$ has symplectic
area at least $R\eps$. 

Now fix $-\infty<a<b<\infty$ and pick $R>1$ such that $R\eps>b-a$. 
Pick a loop $J_t$ of cylindrical almost complex structures on
$(W,\lambda)$ such that $\psi^*J_t=J_M$ over $M\times[2R,4R]$. Pick a
defining Hamiltonian $H$ which is constant outside
$V\cup\psi(M\times[1,2R])$. We claim that under these conditions the
first component $x$ of every gradient flow line $(x,\eta)$ of
$\nabla_JA_H$ connecting critical points with actions in the interval
$(a,b)$ remains in $V\cup\psi\bigl(M\times[1,3R)\bigr)$.  

To see this, we argue by contradiction. Thus suppose that $(x,\eta)$
is a gradient flow line with asymptotics $(x^\pm,\eta^\pm)$ having
actions in $(a,b)$ whose first component $x$ meets the level
$M\times\{3R\}$. Since the asymptotics of $x$ are contained in
$M\times\{1\}$ it exits the set $M\times[2R,4R]$. Let
$U\subset\R\times S^1$ be a connected component of
$x^{-1}(M\times[2R,4R])$ meeting the level $M\times\{3R\}$. Since
$X_H$ vanishes on $M\times[2R,4R]$, the first equation in~\eqref{grad1}
shows that $x|_U$ is $J_M$-holomorphic, hence by the preceding
discussion it has symplectic area at least $R\eps$. The following
contradiction now proves the claim:
\begin{align*}
   b-a
   &\geq A_H(x^+,\eta^+)-A_H(x^-,\eta^-) \cr
   &= \int_{-\infty}^\infty\|\nabla A_H(x,\eta)(s)\|^2ds \cr
   &\geq \int_U|\p_sx|^2ds\,dt \cr
   &= \int_Ux^*\om \cr
   &\geq R\eps. 
\end{align*}
The claim shows that the Floer homology group $RFH^{(a,b)}(M,W)$ can
be computed from critical points and gradient flow lines in the
completion $V\cup_\psi M\times[1,\infty)$ and is therefore independent
of the ambient manifold $W$. 
\end{proof}

In view of Proposition~\ref{prop:indep} we will denote from now on the
Floer homology groups $RFH^{(a,b)}(M,W)$ by $RFH^{(a,b)}(V)$, and the
Rabinowitz Floer homology by
$$
   RFH_*(V) = \lim_{\stackrel \longrightarrow \mu} \lim_{\stackrel
   \longleftarrow \lambda} RFH_*^{(-\lambda,\mu)}(V), \qquad
   \lambda,\mu\to\infty.
$$
We further introduce 
$$
RFH_*^{\ge 0}(V):= RFH_*^{(-\delta,+\infty)}(V), \qquad 
RFH_*^{\le 0}(V):= RFH_*^{(-\infty,\delta)}(V)
$$
and 
$$
RFH_*^0(V):=RFH_*^{(-\delta,\delta)}(V)
$$
for $\delta>0$ small enough. It then follows from the
definition~\cite{CF} that  
$$
RFH_*^0(V) = H_{*+n-1}(M). 
$$
We note that there are morphisms 
$$
H_{*+n-1}(M) \to RFH_*^{\ge 0}(V)
$$
and 
$$
RFH_*^{\le 0}(V) \to H_{*+n-1}(M) \stackrel \sim \to H^{-*+n}(M)
$$ 
induced by action truncation. However, there is no morphism between
$H_{*+n-1}(M)$ and $RFH_*(V)$ since the latter group is defined using
both a direct and an inverse limit.  

\begin{remark} \label{rmk:grading}
Rabinowitz Floer homology is $\Z$-resp.~$\Z_2$-graded under the same
conditions as symplectic homology, see Section~\ref{ss:grading}. 
The $\Z$-grading on $RFH_*$ (when it is defined) used in this paper is
obtained from the 
original one in~\cite{CF} by adding $\frac 1 2$ for the generators
with positive action and subtracting $\frac 1 2$ for the generators with
negative action. 
\end{remark} 

As usual for $\R$-filtered homology theories, given $-\infty\le
a<b<c\le\infty$ there is a long exact sequence of 
homology groups induced by truncation by the values of the Rabinowitz
action functional   
$$
\dots \to RFH_*^{(a,b)}(V) \to RFH_*^{(a,c)}(V) \to RFH_*^{(b,c)}(V) \to
RFH_{*-1}^{(a,b)}(V) \to \dots
$$

\section{Perturbations of the Rabinowitz action
  functional}\label{sec:perturbations} 

In this section we introduce a family of perturbations of
the Rabinowitz action functional which will be used later
to show that Rabinowitz Floer homology is isomorphic to
symplectic homology. These perturbations are perturbations
in the second variable $\eta$. In particular, the perturbed
Rabinowitz action functional is not linear any more in
$\eta$, so that the interpretation of $\eta$ as a Lagrange multiplier
is not any more true for the perturbed functional.

Assume that $(V,\lambda)$ is a compact exact symplectic manifold with
boundary $M=\p V$ such that $\lambda_M:=\lambda|_M$ is a positive
contact form. Denote by 
$$
   \hat V:=V\cup \bigl(M\times[1,\infty)\bigr)
$$
its completion and extend the 1-form $\lambda$ from $V$ to $\hat V$ by
$\lambda:=r\lambda_M$ on $M\times [1,\infty)$.  
Suppose further that
$H \in C^\infty(\hat V,\mathbb{R})$ is an autonomous Hamiltonian
and $b,c \in C^\infty(\mathbb{R},\mathbb{R})$ are smooth functions.
We abbreviate by $\mathcal{L}=C^\infty(S^1,\hat V)$ the free loop
space of $\hat V$. Consider the perturbed Rabinowitz action functional 
$$
   A_{H,b,c} \colon \mathcal{L}\times \mathbb{R} \to \mathbb{R}
$$
defined for $(x,\eta) \in \mathcal{L}\times \mathbb{R}$ by
$$
   A_{H,b,c}(x,\eta)=\int_0^1x^*\lambda-b(\eta)\int_0^1H(x(t))dt+c(\eta).
$$
Note that
$$
  A_{H,\eta,0}=A_H.
$$
The critical points of $A_{H,b,c}$ are pairs $(x,\eta)$ such that 
$$
\left\{\begin{array}{l} 
\dot x = b(\eta) X_H, \\ 
b'(\eta)\int H(x(t))dt = c'(\eta).
\end{array}\right.
$$
The gradient of the action functional $A_{H,b,c}$ with respect to the
metric $g_J$ induced by a circle $J_t$ of cylindrical almost complex
structures at a point $(x,\eta) \in \mathcal{L}\times \mathbb{R}$ reads
$$
   \nabla_J A_{H,b,c}(x,\eta)=
   \left(\begin{array}{c}
   -J_t(x)\big(\partial_t x-b(\eta)X_{H}(x)\big)\\
   -b'(\eta)\int_0^1H(x(t))dt+c'(\eta).
   \end{array}\right)
$$
Hence (positive) gradient flow lines are solutions 
$(x,\eta)\in C^\infty(\mathbb{R}\times S^1,\widehat{V}) \times
C^\infty(\mathbb{R},\mathbb{R})$ of the partial differential equation
\begin{equation}\label{grad}
\left.\begin{array}{c}
\partial_s x+J_t(x)\big(\partial_t x-b(\eta)X_{H}(x)\big)=0\\
\partial_s\eta + b'(\eta)\int_0^1H(x(t))dt - c'(\eta)=0.
\end{array}\right\}
\end{equation}
Compactness up to breaking for solutions of (\ref{grad}) of fixed
asymptotics was shown in \cite{CF} in the unperturbed case and for $H$
a defining Hamiltonian for $M$. In the general case, this involves the 
following uniform bounds:
\begin{itemize}
 \item a uniform $L^\infty$-bound on the loop $x$,
 \item a uniform $L^\infty$-bound on the Lagrange multiplier $\eta$,
 \item a uniform $L^\infty$-bound for the derivatives of $x$.
\end{itemize}
Given these bounds compactness up to breaking then follows from the
usual arguments in Floer homology. The bound on the derivatives of $x$
is standard once the uniform bounds on $x$ and $\eta$ are established: 
By usual bubbling analysis, an explosion of derivatives would give
rise to a nonconstant $J$-holomorphic sphere, which does not exist 
since the symplectic form $d\hat\lambda$ is exact. 

In the remainder of this section we will establish uniform bounds on
$x$ and $\eta$ under suitable hypotheses. 

Note that the Liouville flow defines an embedding $M\times\R_+
\into\hat V$, where we use the notation $\R_+=(0,\infty)$. In this
section we will restrict to {\bf radial Hamiltonians}
$$
   H(y,r)=h(r)
$$ 
depending only on the coordinate $r\in\R_+$. Here
$h:\R_+\to\R$ is a smooth function which is constant near $0$ and
$H$ is extended to $\hat V$ by this constant. Moreover, we assume
throughout that 
$$
   h'(r)\geq 0\qquad \text{for all }r\in\R_+.  
$$

\subsection{A Laplace estimate}
In contrast to~\cite{CF} we cannot always assume that
the Hamiltonian $H$ has fixed compact support. Instead of that
we also want to consider Hamiltonians which grow linearly
in the symplectization. In this case gradient flow lines
of the Rabinowitz action functional do not reduce to holomorphic
curves outside a compact set and hence we cannot use
convexity at infinity directly. Nevertheless we show in the following
subsections how we can obtain an $L^\infty$-estimate for gradient flow
lines which only depends on the energy of the flow line.

Consider the subset $M\times\R_+\subset\hat V$ as above. 
The symplectic form $\om=d\lambda$ is given on $M\times\R_+$
by 
$$
   \omega=d(r\lambda_M) = dr\wedge\lambda_M+r\,d\lambda_M.
$$ 
The Liouville vector field $X$ is given on
$M\times\R_+$ by 
$$
   X=r\frac{\partial}{\partial r}
$$ 
and its flow is the map
$$
   \phi^\rho(y,r)=(y,re^{\rho}), \quad 
   \rho \in \mathbb{R},\,\,(y,r) \in M\times\R_+.
$$
Fix a smooth family of cylindrical almost complex structures $J_t$. 
We are interested in {\bf partial gradient flow lines} of the
perturbed Rabinowitz action functional, i.e.~solutions 
$$
   w=(x,\eta) \in
   C^\infty([-T,T]\times S^1,\hat V) \times C^\infty([-T,T],\R)
$$
of~\eqref{grad} on some compact time interval $[-T,T]$. 
Here we assume the the Hamiltonian $H(y,r)=h(r)$ is radial. 
Using the formula
$$
   \nabla_t H(x) = h'(r)X(x)
$$
for the gradient of the metric $\omega(\cdot, J_t \cdot)$ on
$M\times\R_+$, we observe that a partial gradient flow line
$(x,\eta)$ satisfies at points where $x(s,t)\in M\times\R_+$
the equation   
\begin{equation}\label{grW}
\left.\begin{array}{c}
\partial_s x+J_t\partial_tx+b(\eta)h'(r)X(x)\\
\partial_s \eta + b'(\eta)\int_0^1h(r)dt - c'(\eta)=0. 
\end{array}\right\}
\end{equation}
We define $f:M\times\R_+\to\R$ by 
$$
   f(y,r):=r,
$$
and for a partial gradient flow line $w$ we define 
$$
   \rho(s,t):=\mathrm{ln}(f(x(s,t))) 
$$
whenever $x(s,t)\in M\times\R_+$. 
Our $L^\infty$-bounds for $x$ are based on the following inequality
for the Laplacian of $\rho$.  

\begin{lemma}\label{lem:lap1}
Let $H(x)=h(r)$ be a radial Hamiltonian and let $(x,\eta) \in 
C^\infty([-T,T]\times S^1,\hat V) \times C^\infty([-T,T],\R)$ be a partial
gradient flow line for $A_{H,b,c}$. Then at points $(s,t)$ where
$x(s,t)\in M\times\R_+$ the Laplacians of $f\circ x$ and $\rho$ satisfy
\begin{equation}\label{eq:lap-f}
   \Delta(f\circ x)
   = \langle \partial_s x,\partial_s x\rangle_t -
   \partial_s\big(h'(r)b(\eta)\big)f\circ x,
\end{equation}
\begin{equation}\label{eq:lap-rho}
   \Delta\rho \geq -\partial_s\big(h'(r)b(\eta)\big). 
\end{equation}
\end{lemma}

\begin{proof}
If $\langle \cdot \rangle_t$ denotes the Riemannian metric 
induced from the cylindrical almost complex 
structure $J_t$, we note that on $M\times\R_+$ we have 
$$
   \langle X,X \rangle_t=f.
$$
All the following computations are done at a point $(s,t)$ where
$x(s,t)\in M\times\R_+$.  
We first compute $d^c(f(x))=d(f(x)) \circ i$ using the first
equation in (\ref{grW}):
\begin{eqnarray*}
-d^c(f\circ x)&=&-(df(x)\partial_t x)ds+(df(x)\partial_s x)dt\\
&=&-\big(df(x)(J_t(x)\partial_t x)\big)dt
-\big(df(x)(J_t(x)\partial_s x)\big)ds\\
& &\big(df(x)(\partial_s x+J_t(x)\partial_t x)\big)dt
+\big(df(x)(J_t(x)\partial_s x-\partial_t x)\big)ds\\
&=&-\big\langle \nabla_t f(x),J_t(x)\partial_t x\big\rangle_t dt
-\big\langle \nabla_t f(x), J_t(x)\partial_s x \big\rangle_t ds\\
& &-h'(r)b(\eta)\big\langle \nabla_t f(x),X(x) \big\rangle_t dt\\
& &-h'(r)b(\eta)\big\langle \nabla_t f(x),J_t(x)X(x) \big\rangle_t ds\\
&=&\omega\big(X(x),\partial_t x\big)dt
+\omega\big(X(x),\partial_s x\big)ds\\
& &-h'(r)b(\eta)\big\langle 
X(x),X(x) \big\rangle_t dt\\
&=&x^* \iota_X \omega-h'(r)b(\eta)f(x) dt.
\end{eqnarray*}
Applying $d$ we obtain
$$
   \Delta(f\circ x)ds \wedge dt =-dd^c(f\circ x)
   =x^* d\iota_X \omega-\partial_s\big(h'(r)b(\eta)f(x)\big)ds \wedge
   dt. 
$$
Using again the first equation in (\ref{grW}) we find
\begin{eqnarray*}
x^* d\iota_X \omega
&=&x^*\mathcal{L}_X \omega\\
&=&x^*\omega\\ 
&=&\omega\big(\partial_s
x,J_t(x)\partial_sx+h'(r)b(\eta)J_t(x)X(x)\big) ds \wedge dt\\
&=&\langle \partial_s x,\partial_s x\rangle_tds \wedge dt
+h'(r)b(\eta)\langle \partial_s x,\nabla_t f(x) \rangle_tds \wedge dt\\
&=&\langle \partial_s x,\partial_s x\rangle_tds \wedge dt
+h'(r)b(\eta)df(x)\partial_s x ds \wedge dt\\
&=&\langle \partial_s x,\partial_s x\rangle_t ds \wedge dt
+h'(r)b(\eta)\partial_s (f\circ x)ds \wedge dt,\\ 
\end{eqnarray*}
and hence the Laplacian of $f\circ x$ is given by
$$
   \Delta(f\circ x)
   = \langle \partial_s x,\partial_s x\rangle_t -
   \partial_s\big(h'(r)b(\eta)\big)f(x). 
$$
This proves the first statement in the lemma. 

Since $J_t$ interchanges the Reeb and the Liouville vector field on
$M\times\R_+$, i.e.~$R=J_t X$, we conclude that
$$
   \langle R,X \rangle_t=0, \quad ||R||_t^2=||X||_t^2=f.
$$
In particular, we can estimate the norm of $\partial_s x$
in the following way:
\begin{eqnarray*}
||\partial_s x||_t^2 &\geq& \frac{\langle \partial_s x,X \rangle^2_t}
{||X||_t^2}+\frac{\langle \partial_s x,R \rangle^2_t}
{||R||_t^2}\\ 
&=&\frac{\langle \partial_s x,\nabla_t f \rangle^2_t}
{f(x)}+\frac{\langle -J_t\partial_t x - h'(r)b(\eta)X,R \rangle^2_t}
{f(x)}\\ 
&=&\frac{(df(x)\partial_s x)^2}
{f(x)}+\frac{\langle J_t\partial_t x,R \rangle^2_t}
{f(x)}\\
&=&\frac{(\partial_s f(x))^2}
{f(x)}+\frac{\langle\partial_t x,X \rangle^2_t}
{f(x)}\\ 
&=&\frac{(\partial_s f(x))^2}
{f(x)}+\frac{(\partial_t f(x))^2}
{f(x)}
\end{eqnarray*}
Combining the above two expressions we obtain the following
estimate for the Laplacian of $f\circ x$
\begin{equation}\label{conv1}
\Delta (f\circ x) - \frac{(\partial_s f(x))^2}
{f(x)}-\frac{(\partial_t f(x))^2}
{f(x)}+ \partial_s\big(h'(r)b(\eta)\big)f(x)\geq 0.
\end{equation}
Replacing $f$ by $e^\rho$ and dividing by $f$ we obtain
\begin{eqnarray*}
0 &\leq&
\frac{\Delta e^\rho}{e^\rho}-\frac{(\partial_s(e^\rho))^2}{e^{2\rho}}
-\frac{(\partial_t(e^\rho))^2}{e^{2\rho}}+\partial_s\big(h'(r)b(\eta)\big)\\
&=&\Delta \rho+(\partial_s \rho)^2+(\partial_t \rho)^2
-(\partial_s \rho)^2-(\partial_t \rho)^2+\partial_s\big(h'(r)b(\eta)\big)\\
&=&\Delta \rho+\partial_s\big(h'(r)b(\eta)\big). 
\end{eqnarray*}
This proves the second statement and hence Lemma~\ref{lem:lap1}. 
\end{proof}

\subsection{$L^\infty$-bounds on the loop $x$}\label{ss:x-bound1}

To draw conclusions from Lemma~\ref{lem:lap1}, we now make 
the following assumptions on the functions $(H,b,c)$:
\begin{equation}\label{eq:H-lin}
   H(y,r) = A(r-R)+E \quad\text{for }r\geq R,\qquad A,E\geq 0,\ R\geq
   1. 
\end{equation}
\begin{equation}\label{eq:bc}
   \sup_{\eta\in\R}|b'(\eta)|=B<\infty, \quad
   \sup_{\eta\in\R}|c'(\eta)|=C< \infty.
\end{equation}

The crucial observation of the following lemma
is that we can get a uniform bound from below on the Laplacian of
$\rho$ along a partial gradient flow line provided only
that the derivatives of $b$ and $c$ are uniformly bounded, but
not necessarily $b$ and $c$ itself.

\begin{lemma}\label{lap}
Suppose that $(H,b,c)$ satisfy assumptions~\eqref{eq:H-lin}
and~\eqref{eq:bc} and let $D:=|\min H|$. Suppose further that
$(x,\eta) \in 
C^\infty([-T,T]\times S^1,W) \times C^\infty([-T,T],\R)$ is a partial
gradient flow line for $A_{H,b,c}$. Then at points $(s,t)$ where
$x(s,t)\in M\times [R,\infty)$ the Laplacian of $\rho \circ 
x$ satisfies 
$$
   \Delta(\rho\circ x) \geq -A^2B^2D-ABC.
$$
\end{lemma}

\begin{proof}
Let $(s,t)$ be a point where $x(s,t)\in M\times [R,\infty)$. The
assumption on $h$ implies $h'(r)=A$ at this point, so the
estimate~\eqref{eq:lap-rho} reads
$$
   \Delta \rho \geq -A b'(\eta)\p_s\eta. 
$$
Using the second equation in (\ref{grW}), assumption~\eqref{eq:bc} and
$H(x)\geq -D$ we obtain from this the estimate
\begin{eqnarray*}
\Delta \rho 
&\geq& -A b'(\eta)\p_s\eta \\
&=& A^2 b'(\eta)^2\int_0^1 H(x)dt - A b'(\eta)c'(\eta) \\
&\geq& -A^2b'(\eta)^2D - A b'(\eta)c'(\eta) \\
&\geq& -A^2B^2D-ABC. 
\end{eqnarray*}
This proves Lemma~\ref{lap}.
\end{proof}

The crucial hypothesis for the following proposition is a uniform
bound on $x$ for {\bf nearly critical points} $(x,\eta)$. More
precisely, we assume that there exists a family of cylindrical
almost complex structures $J_t$ on $W$ such that for 
$(x,\eta) \in \mathcal{L}\times \R$ the following implication holds: 
\begin{equation}\label{tame}
||\nabla_J A_{H,b,c}(x,\eta)||_J\leq \epsilon
\quad \Longrightarrow \quad
\max_{t \in S^1}f(x(t)) \leq S.
\end{equation}
 We define the {\bf energy} of a partial gradient flow line
$w=(x,\eta)$ by 
$$
   \mathcal{E}_w := \int_{-T}^T||\nabla A_{H,b,c}(w)||^2 ds =
   A_{H,b,c}(w(T))-A_{H,b,c}(w(-T)).   
$$

\begin{proposition}\label{prop:x-bound1}
Suppose that the triple $(H,b,c) \in C^\infty(\R_+,\R)\times
C^\infty(\R)\times C^\infty(\R)$ satisfies assumptions~\eqref{eq:H-lin}
and~\eqref{eq:bc} as well as condition~\eqref{tame} for some
$(\eps,S)$. Let $w=(x,\eta) \in C^\infty([-T,T]\times S^1,W)
\times C^\infty([-T,T],\R)$ be a partial
gradient flow line of $\nabla_J A_{H,b,c}$
satisfying
\begin{equation}\label{beiT}
\sup_{t \in S^1}f(x(\pm T,t)) \leq S.
\end{equation}
Then for each
$(s,t) \in [-T,T]\times S^1$ the following estimate holds
$$
   f(x(s,t)) \leq \max(R,S) \exp\left(\frac{(A^2B^2D+ABC)
   \mathcal{E}_w^2}{2\epsilon^4}\right).
$$
\end{proposition}

\begin{proof}
Fix $s_0 \in [-T,T]$. We abbreviate
$$
   \sigma^\pm(s_0)=\inf\big\{\sigma \in [0,T\mp s_0]: 
   ||\nabla A_{H,b,c}(w)(s_0\pm \sigma)||
< \epsilon\big\}.
$$
We claim that
\begin{equation}\label{sig}
\sigma^\pm(s_0) \leq \frac{\mathcal{E}_w}{\epsilon^2}.
\end{equation}
We prove this assertion only for $\sigma^+(s_0)$. Indeed,
$$
   \mathcal{E}_w
   = \int_{-T}^T||\nabla A_{H,b,c}(w)||^2 ds 
   \geq \int_{s_0}^{s_0+\sigma^+(s_0)}||\nabla A_{H,b,c}(w)||^2 ds 
   \geq \epsilon^2 \sigma^+(s_0),
$$
proving the claim.
It follows from (\ref{beiT}), the definition of $\sigma^\pm(s_0)$ and
condition~\eqref{tame} that
\begin{equation}\label{bro}
\max_{t \in S^1}\rho\big(x(s_0\pm \sigma^\pm(s_0),t)\big) \leq \ln S.
\end{equation}
We introduce the following finite cylinder
$$
   \mathcal{Z}_{s_0}=[s_0-\sigma^-(s_0),s_0+\sigma^+(s_0)]\times S^1. 
$$
For a constant 
$$
   \nu>A^2B^2D+ABC
$$ 
we introduce the function 
$\chi \in C^\infty\big(\mathcal{Z}_{s_0}\big)$ by
$$
   \chi(s,t):=\rho(x(s,t))+\frac{\nu (s-s_0)^2}{2}.
$$
Using (\ref{sig}) and (\ref{bro}) we estimate $\chi$ at the
boundary of the cylinder:
\begin{equation}\label{bch}
   \max_{\partial \mathcal{Z}_{s_0}} \chi \leq \ln S+
   \frac{\nu \mathcal{E}_w^2}{2\epsilon^4}.
\end{equation}
Lemma~\ref{lap} yields the following implication for
$(s,t)\in\mathcal{Z}_{s_0}$: 
$$
   \chi(s,t)\geq \ln R+\frac{\nu\mathcal{E}_w^2}{2\epsilon^4}
   \Longrightarrow f(x(s,t))\geq R\Longrightarrow
   \Delta\chi(s,t)=\Delta\rho(x(s,t))+\nu>0. 
$$
Thus $\chi$ cannot have an interior maximum bigger than $\ln
R+\frac{\nu\mathcal{E}_w^2}{2\epsilon^4}$. Combining this with the
boundary estimate~(\ref{bch}) yields 
$$
   \sup_{\mathcal{Z}_{s_0}} \rho\leq \sup_{\mathcal{Z}_{s_0}}\chi
   \leq \ln\max(R,S)+\frac{\nu\mathcal{E}_w^2}{2\epsilon^4}.
$$
This finishes the proof of Proposition~\ref{prop:x-bound1}.
\end{proof}

\subsection{$L^\infty$-bounds at nearly critical points} \label{sec:nearcrit}

In order to apply Proposition~\ref{prop:x-bound1}, we need to
establish condition~\eqref{tame} for given triples $(H,b,c)$. For the
unperturbed Rabinowitz functional this was proven in~\cite{CF}:

\begin{lemma}[\cite{CF}, proof of Proposition 3.2, Step~2]\label{lem:tame1}
Suppose that $(h,b,c)$ satisfy the following conditions:
\begin{enumerate}
\item $h(r)=r-1$ for $r\in[1-\delta,1+\delta]$; 
\item $b(\eta)=\eta$;
\item $c\equiv 0$.
\end{enumerate}
Then for there exists an $\eps=\eps(\delta)>0$ such that
condition~\eqref{tame} holds with $S=1+\delta$.  
\end{lemma}

Before proving a corresponding result for other triples $(h,b,c)$, we
first make two general observations. 

Denote by $|\ |_t$ the metric
$\lambda\otimes\lambda+d\lambda(\cdot,J_t\cdot)$ on $M$ and by $\|\
\|_2$ the $L^2$-norm with respect to this metric. For
$A\notin\Spec(M,\lambda)$ denote by $\eta_A>0$ the distance from $A$
to $\Spec(M,\lambda)$. 

\begin{lemma}\label{lem:deltaA}
For each $A\notin\Spec(M,\lambda)$ there exists $\delta_A>0$ such that 
$$
   \|\dot y-aR(y)\|_2\geq\delta_A \quad\text{for all }y\in
   C^\infty(S^1,M) \text{ and }
   a\in[A-\frac{1}{2}\eta_A,A+\frac{1}{2}\eta_A].  
$$
\end{lemma}

\begin{proof}
This is an immediate consequence of the Arzela-Ascoli theorem. 
\end{proof}

\begin{lemma}\label{lem:wild}
Suppose $h'(r)=1$ for $1\leq A\leq r\leq B$ and let $b,c$ be
arbitrary. Let $(x,\eta)\in\LL\times\R$ and suppose that for $x=(y,r)$
there are $t,t'\in S^1$ with $r(t)\leq A$ and $r(t')\geq B$. Then
$$
   \|\nabla A_{H,b,c}(x,\eta)\|_J\geq\frac{B-A}{\sqrt{B}}. 
$$
\end{lemma}

\begin{proof}
Recall that on $M\times[1,\infty)$ the almost complex structure $J_t$
maps the Liouville vector field $X=r\p_r$ to the Reeb vector field $R$
and preserves the contact structure $\xi=\ker\lambda$. It follows that
the metric at $(y,r)\in M\times[1,\infty)$ is given by
$$
   |uX+vR+w|_t^2=r\Bigl[u^2+v^2+d\lambda(w,J_tw)\Bigr],\qquad
   u,w\in\R, w\in\xi. 
$$
So at points $(s,t)$ with $x(s,t)=(y(s,t),r(s,t))\in
M\times[A,B]$ we have
$$
   |\dot x-bX_H(x)|_t^2 = \frac{\dot r^2}{r}+r|\dot y-bR(y)|_t^2. 
$$
By assumption there exist $t_0<t_1$ with $r(t_0)=A$, $r(t_1)=B$ and
$r(t)\in[A,B]$ for all $t\in[t_0,t_1]$. Now the lemma follows from the
estimate 
\begin{align*}
   \|\nabla A_{H,b,c}(x,\eta)\|_J 
   &\geq \sqrt{\int_0^1|\dot x-b(\eta)X_H(x)|^2dt} \cr
   &\geq \int_0^1|\dot x-b(\eta)X_H(x)|dt \cr
   &\geq \int_{t_0}^{t_1}|\dot x-b(\eta)X_H(x)|dt \cr
   &\geq \int_{t_0}^{t_1}\frac{|\dot r|}{\sqrt{r}}dt \cr
   &\geq \int_{t_0}^{t_1}\frac{|\dot r|}{\sqrt{B}}dt \cr
   &\geq \frac{B-A}{\sqrt{B}}. 
\end{align*}
\end{proof}

Now for $A\notin\Spec(M,\lambda)$ let $b_A:\R\to\R$ be a smoothing of
the function 
\begin{equation} \label{eq:bA}
   b_A(\eta) = 
   \left\{\begin{array}{ll} 
   -A, & \eta\le-A,\\ 
   \eta, & -A\le\eta\le A,\\
   A, & A\le\eta,
   \end{array}\right.
\end{equation}
see Figure~\ref{fig:b} on page~\pageref{fig:b}. Here the smoothing is
done in such a way that $b_A'(\eta)=1$ whenever
$b_A(\eta)\in[-A+\frac{1}{2}\eta_A,A-\frac{1}{2}\eta_A]$. 

\begin{lemma}\label{lem:tame}
Suppose that $(h,b,c)$ satisfy the following conditions:
\begin{enumerate}
\item $h(r)=r-R+D$ for $r\geq R$ with constants $R\geq 1$ and $D\geq
  0$; 
\item $b=b_A$ as in~\eqref{eq:bA} with $A\notin\Spec(M,\lambda)$;
\item $|c'|\leq C$.
\end{enumerate}
Then condition~\eqref{tame} holds with 
$$
   \eps=\min(\delta_A,1/\sqrt{3}),\qquad S=3(R+C),
$$
where $\delta_A$ is the constant from Lemma~\ref{lem:deltaA}. 
\end{lemma}

\begin{proof}
Let $(x,\eta)\in\LL\times\R$ with $||\nabla
A_{H,b,c}(x,\eta)||_J<\eps$.  
Recall from the proof of Lemma~\ref{lem:wild} that 
at points $(s,t)$ with $x(s,t)=(y(s,t),r(s,t))\in
M\times[R,\infty)$ we have
$$
   |\dot x-bX_H(x)|_t^2 = \frac{\dot r^2}{r}+r|\dot y-bR(y)|_t^2. 
$$
Now we prove the lemma in two steps. Set 
$$
   E:=R+C\geq R\geq 1.
$$ 
{\bf Step 1: }Either $\max r\leq 3E$ or $\min r\geq 2E$. 

Otherwise there exist $t,t'$ with $r(t)=2E$ and $r(t')=3E$, so
Lemma~\ref{lem:wild} with $A=2E$ and $B=3E$ yields
$$
   \|\nabla A_{H,b,c}(x,\eta)\|_J 
   \geq \frac{E}{\sqrt{3E}} = \sqrt\frac{E}{3}. 
$$
This contradicts the hypothesis $||\nabla
A_{H,b,c}(x,\eta)||_J<1/\sqrt{3}\leq \sqrt{E/3}$ and proves Step 1. 

If $\max r\leq 3E=S$ we are done, so assume from now on that $\min
r\geq 2E$. 

{\bf Step 2: }$|b(\eta)|\leq A-\frac{1}{2}\eta_A$. 

By hypothesis we have
\begin{align*}
   2E\delta_A^2
   &\geq \delta_A^2 \cr
   &> \|\nabla A_{H,b,c}(x,\eta)\|_J^2 \cr
   &\geq \int_0^1r|\dot y-b(\eta)R|^2dt \cr
   &\geq 2E\int_0^1|\dot y-b(\eta)R|^2dt,
\end{align*}
hence $\|\dot y-b(\eta)R(y)\|_2<\delta_A$. By definition of
$\delta_A$ this implies $|\,|b(\eta)|-A|\geq\frac{\eta_A}{2}$ and
Step 2 follows. 

By construction of $b$, Step 2 implies $b'(\eta)=1$. Using this and
the hypothesis $|c'|\leq C$ we estimate 
\begin{align*}
   \|\nabla A_{H,b,c}(x,\eta)\|_J 
   &\geq \Bigl|-\int_0^1H(x)dt+c'(\eta)\Bigr| \cr
   &\geq \int_0^1(r-R+D)dt-C \cr
   &\geq 2E-R-C = E. 
\end{align*}
But this contradicts the hypothesis $\|\nabla
A_{H,b,c}(x,\eta)\|_J<1\leq E$, so the second case in Step 1 cannot
occur and the Lemma~\ref{lem:tame} is proved. 
\end{proof}

\subsection{$L^\infty$-bounds on the Lagrange multiplier $\eta$}

In this section we establish bounds on the Lagrange multiplier $\eta$ along
gradient flow lines. For the unperturbed Rabinowitz functional and
defining Hamiltonians that are constant at infinity, such a bound was
proven in~\cite{CF}. The following result is a refinement of this. 

\begin{proposition}\label{prop:eta-bound1}
Suppose that the triple $(h,b,c)$ satisfies the following conditions:
\begin{enumerate}
\item $h(r)=r-1$ for $r\in[1-\delta,1+\delta]$ and
  condition~\eqref{eq:H-lin};  
\item $b(\eta)=\eta$;
\item $c\equiv 0$.
\end{enumerate}
Let $w=(x,\eta):\R\to\LL\times\R$ be a gradient flow line of $\nabla_J
A_{H,\eta,0}$ with asymptotic linits $(x^\pm,\eta^\pm)$. Then $x$ and
$\eta$ are uniformly bounded in terms of constants which only depend
on the actions $A_{H,\eta,0}(x^\pm,\eta^\pm)$, the constant $\delta$ and
the constants $R,A$ in~\eqref{eq:H-lin}.  
\end{proposition}

\begin{proof}
By Lemma~\ref{lem:tame1}, condition~\eqref{tame} holds for
$S=1+\delta$ and some $\eps=\eps(\delta)>0$. Since the asymptotic
limits $x^\pm$ lie on the level set $H^{-1}(0)=M\times\{1\}$, we have 
$r(s,t)\leq S$ for $|s|$ sufficiently large. Hence
Proposition~\ref{prop:x-bound1} provides a uniform bound on $x$ in
terms of the constants $R,S,A$ and the action difference
$\mathcal{E}_w=A_{H,\eta,0}(x^+,\eta^+)-A_{H,\eta,0}(x^-,\eta^-)$. 

The uniform bound on $\eta$ now follows from Corollary 3.5
in~\cite{CF}. In fact, the result in~\cite{CF} is stated for
Hamiltonians that are constant at infinity. But inspection of the
proof shows that it only uses a bound on $|H(x)|$ along the gradient
flow line, which we just established. 
The bound for $\eta$ only
depends on this bound, the actions $A_{H,\eta,0}(x^\pm,\eta^\pm)$ and the
constant $\delta$. 
\end{proof}

Next we observe that for suitable perturbations a bound on $\eta$
becomes in fact much easier:

\begin{lemma} \label{lem:eta-bound2}
For arbitrary $H$, suppose that there exists $A >0$ such that the
following conditions hold for the perturbations $b$ and $c$:
\begin{equation}\label{bc}
 \left.\begin{array}{cc}
b'(\eta)=0, & |\eta| \geq A,\\
c'(\eta)\le0, & \eta\le -A,\\
c'(\eta)\ge0, & \eta\ge A.
\end{array}\right\}
\end{equation}
Let $(x,\eta):\R\to\LL\times\R$ be a gradient flow line of $\nabla_J
A_{H,b,c}$ with asymptotic limits $(x^\pm,\eta^\pm)$. Then
$$
   |\eta(s)|\leq \max\{|\eta^+|,|\eta^-|,A\}\qquad \text{for all
   }s\in\R. 
$$
\end{lemma}

\begin{proof}
It follows from (\ref{bc}) and the second equation in (\ref{grad})
that if $|\eta| \geq R$ then $\partial_s |\eta|\geq 0$. This implies
the stated uniform bound on $|\eta(s)|$. 
\end{proof}

\subsection{Generalized Rabinowitz Floer homologies}

Consider quintuples $Q=(H,b,c,\alpha,\beta)$, where $H=h(r)$ is a radial
Hamiltonian, $b,c:\R\to\R$ are smooth functions, and
$-\infty<\alpha<\beta<\infty$. We fix a cylindrical almost complex
structure $J$ and define the $L^2$-gradient $\nabla_J A_{H,b,c}$ as
above. 
\begin{definition}\label{def:adm}
We call $(H,b,c,\alpha,\beta)$ {\bf
admissible} if the following conditions hold:
\begin{enumerate}
\item There are no critical points of $A_{H,b,c}$ with action $\alpha$
or $\beta$, and the set $\Crit^{(\alpha,\beta)}(A_{H,b,c})$ of critical points
with action in $(\alpha,\beta)$ is compact. 
\item The loop $x$ is uniformly bounded on all gradient flow lines
$(x,\eta):\R\to\LL\times\R$ of $A_{H,b,c}$ connecting critical points
with action in $(\alpha,\beta)$.  
\item The Lagrange multiplier $\eta$ is uniformly bounded on all
gradient flow lines $(x,\eta):\R\to\LL\times\R$ of $A_{H,b,c}$
connecting critical points with action in $(\alpha,\beta)$.  
\end{enumerate}
\end{definition}
As discussed at the beginning of this section, for an admissible
quintuple $(H,b,c,\alpha,\beta)$ the space of gradient flow lines of
$A_{H,b,c}$ connecting critical points with action in $(\alpha,\beta)$
is compact modulo breaking. So we can define its Floer homology
$$
   FH^{(\alpha,\beta)}(A_{H,b,c})
$$
as in Section~\ref{sec:RFH}. We omit the almost complex structure $J$
from the notation since the Floer homology does not depend on it. 

We call a homotopy $Q=\{Q_t\}_{t\in[0,1]}$ of quintuples
$Q_t=(H_t,b_t,c_t,\alpha_t,\beta_t)$ {\bf admissible} if all the $Q_t$
are admissible, the union
$\cup_{t\in[0,1]}\Crit^{(\alpha_t,\beta_t)}(A_{H_t,b_t,c_t})$ is compact, and
the bounds on $x$ and $\eta$ can be chosen uniformly in $t$. The
following two results follow by standard arguments in Floer
homology~\cite{Sa}.

\begin{proposition}\label{prop:phi}
An admissible homotopy $Q=\{Q_t\}_{t\in[0,1]}$ of quintuples 
$Q_t=(H_t,b_t,c_t,\alpha_t,\beta_t)$ induces an isomorphism
$$
   \phi_Q:FH^{(\alpha_0,\beta_0)}(A_{H_0,b_0,c_0}) \to
   FH^{(\alpha_1,\beta_1)}(A_{H_1,b_1,c_1}).
$$
The induced isomorphisms are functorial with respect to concatenation 
$$
   (Q\#R)_t := \begin{cases}
      Q_{2t} & t\in[0,1/2], \cr
      Q_{2t-1} & t\in[1/2,1] \cr
   \end{cases}
$$
and inverse $Q^{-1}_t:=Q_{1-t}$, namely 
$$
   \phi_{Q\#R}=\phi_R\circ\phi_Q,\qquad \phi_{Q^{-1}}=\phi_Q^{-1}. 
$$
\end{proposition} 

\begin{proposition}\label{prop:iota}
Define a partial order on admissible quintuples by
$Q=(H,b,c,\alpha,\beta)\leq Q'=(H',b',c',\alpha',\beta')$ iff $\alpha\leq
\alpha'$, $\beta\leq \beta'$ and $(H,b,c)=(H',b',c')$. Then for $Q\leq Q'$ the
obvious inclusions induce homomorphisms
$$
   \iota_{QQ'}:FH^{(\alpha,\beta)}(A_{H,b,c}) \to
   FH^{(\alpha',\beta')}(A_{H,b,c}).  
$$
They are functorial in the following sense:
$$
   \iota_{QQ''}=\iota_{Q'Q''}\circ\iota_{QQ'} \text{ for }Q\leq Q'\leq
   Q'',\qquad \iota_{QQ}=\id. 
$$
Moreover, if $Q=\{Q_t\}_{t\in[0,1]}$ and $Q'=\{Q_t'\}_{t\in[0,1]}$ are
admissible homotopies with $Q_0\leq Q_0'$ and $Q_1\leq Q_1'$ then
$$
   \phi_{Q'}\circ\iota_{Q_0Q_0'} = \iota_{Q_1Q_1'}\circ\phi_Q. 
$$
\end{proposition} 

Now we first consider the case of the unperturbed Rabinowitz
functional, i.e.~$b(\eta)=\eta$ and $c\equiv 0$, with Hamiltonians
$H=h(r)$ satisfying the following condition for some constants
$A,E\geq 0$ and $R\geq 1$ (which may depend on $h$):
\begin{equation}\label{eq:h}
   h(r)=r-1 \text{ near }r=1,\qquad h(r)=A(r-R)+E \text{ for }r\geq
   R. 
\end{equation}

\begin{proposition}\label{prop:RFH-lin}
For any $h$ satisfying~\eqref{eq:h} and for any
$\alpha,\beta\notin\Spec(M,\lambda)$, we have
$$
   FH^{(\alpha,\beta)}(A_{H,\eta,0}) = RFH^{(\alpha,\beta)}(V). 
$$
In particular, for any such $h$ we have 
$$
   RFH_*(V) = \lim_{\stackrel \longrightarrow \mu} \lim_{\stackrel
   \longleftarrow \lambda} FH_*^{(-\lambda,\mu)}(A_{H,\eta,0}), \qquad
   \lambda,\mu\to\infty.
$$
\end{proposition} 

\begin{proof} 
Note that critical points of $A_{H,\eta,0}$ for $h(r)=r-1$ near $r=1$
correspond to closed Reeb orbits on $M$ and their action equals their
period, so the first condition in the definition of admissibility is
satisfied. Proposition~\ref{prop:eta-bound1} provides uniform bounds
on gradient flow lines $(x,\eta)$ between critical points with action
in $(\alpha,\beta)$ for any such triple $(h,\eta,0)$, and these bounds
are also uniform for homotopies of such triples. Hence all such
triples are admissible and connected by admissible homotopies, and the
result follows from Proposition~\ref{prop:phi}. 
\end{proof} 

\begin{remark} \rm Proposition~\ref{prop:RFH-lin} allows us to compute
Rabinowitz Floer homology using Hamiltonians that are linear at
infinity rather than constant at infinity as in the original
definition in~\cite{CF}. This last case is also included in the
statement of Proposition~\ref{prop:RFH-lin} and corresponds to $A=0$. 
\end{remark}

From now on we will always assume that $h$ satisfies
\begin{equation}\label{eq:h-lin1}
   h(r) = r-1\qquad \text{for }r\geq 2. 
\end{equation}
The second class of triples $(h,b,c)$ we wish to consider are those
satisfying the following conditions:
\begin{equation}\label{eq:hbc}
   \left\{\begin{array}{ll}
   h(r) = r-1 &\text{ for }r\geq 2, \\
   b=b_A &  \text{ as in Lemma~\ref{lem:tame} with }
   A\notin\Spec(M,\lambda), \\
   |c'(\eta)|\le C & \text{ for } \eta\in\R, \text{ and } 
    \left\{\begin{array}{ll} 
     c'(\eta)\leq 0 & \text{ for }\eta\leq -A, \\ 
     c'(\eta)\ge 0 &  \text{ for }\eta\geq A. 
    \end{array} \right.   
  \end{array} \right.
\end{equation}
By Lemma~\ref{lem:eta-bound2}, $|\eta|$ is bounded along gradient flow
lines by $\max\{|\eta^+|,|\eta^-|,A\}$. Lemma~\ref{lem:tame} and
Proposition~\ref{prop:x-bound1} provide a uniform bound on $x$ along
gradient flow lines. Hence admissibility of $(h,b,c)$ comes down to
the first condition in Definition~\ref{def:adm} and we have proved

\begin{proposition}\label{prop:RFH-flat}
A quintuple $(h,b,c,\alpha,\beta)$ with $(h,b,c)$
satisfying~\eqref{eq:hbc} is
admissible, and hence its Floer homology
$FH^{(\alpha,\beta)}(A_{H,b,c})$ is defined, provided that 
there are no critical points of $A_{H,b,c}$ with action $\alpha$
or $\beta$ and the set $\Crit^{(\alpha,\beta)}(A_{H,b,c})$ of critical points
with action in $(\alpha,\beta)$ is compact. 

A homotopy $(h_t,b_t,c_t,\alpha_t,\beta_t)$ with $(h_t,b_t,c_t)$
satisfying~\eqref{eq:hbc} is admissible, and hence
$FH^{(\alpha_0,\beta_0)}(A_{H_0,b_0,c_0})\cong 
FH^{(\alpha_1,\beta_1)}(A_{H_1,b_1,c_1})$, provided that 
there are no critical points of $A_{H_t,b_t,c_t}$ with action $\alpha_t$
or $\beta_t$ and provided the set
$\bigcup_{t\in[0,1]}\Crit^{(\alpha_t,\beta_t)}(A_{H_t,b_t,c_t})$ is compact.  
\end{proposition}

In general, the hypotheses of Proposition~\ref{prop:RFH-flat} may fail
for two reasons: 
\begin{itemize}
\item critical values may cross the end points of the intervals
  $[\alpha_t,\beta_t]$;
\item there may exist families of critical points $(x,\eta)$ with $x$
  constant and $\eta$ unbounded.  
\end{itemize}
Thus the Floer homology $FH^{(\alpha,\beta)}(A_{H,b,c})$ need not be
defined, and if it is it may depend on the quintuple
$(h,b,c,\alpha,\beta)$ even if $\alpha$ and $\beta$ are fixed. In
Section~\ref{sec:proof} we will construct specific homotopies
satisfying the hypotheses of Proposition~\ref{prop:RFH-flat} in order
to interpolate between Rabinowitz Floer homology and symplectic
homology.

\section{Preparation for the proof of the main result}\label{sec:prep}

In the previous section we studied the Floer homology of the perturbed
Rabinowitz functional $A_{H,b,c}$ in two cases: 
\begin{enumerate}
\item $b(\eta)=\eta$, $c\equiv 0$ and $h$ satisfying~\eqref{eq:h};
\item $(h,b,c)$ satisfying~\eqref{eq:hbc}. 
\end{enumerate}
The main result of this section is an interpolation between these two
classes: 

\begin{theorem}\label{thm:flatten}
Suppose $-\infty<\alpha<\beta<\infty$ satisfy
$\alpha,\beta\notin\Spec(M,\lambda)$, and $h$ satisfies $h'(r)\ge 0$
for $r>0$ and $h(r)=r-1$ for
$r\geq 1/2$. Then there exists a constant $A(\alpha,\beta)$ such
that for $b=b_A$ as in Lemma~\ref{lem:tame} with $A\geq
A(\alpha,\beta)$ we have
$$
   FH^{(\alpha,\beta)}(A_{H,\eta,0}) \cong
   FH^{(\alpha,\beta)}(A_{H,b,0}). 
$$
\end{theorem}

This result allows us to replace the original function $b(\eta)=\eta$
by a function $b_A$ which is constant at infinity. Its proof will
occupy the remainder of this section.

\subsection{An improved $L^\infty$-bound on the loop
  $x$}\label{ss:x-bound2} 

In this subsection we will derive from Lemma~\ref{lem:lap1} the
following $L^\infty$-bound on the loop $x$. In contrast to the
bounds in Section~\ref{ss:x-bound1}, this bound does not rely on
condition~\eqref{tame} and thus holds uniformly for all $b$ satisfying
$0\leq b'\leq 1$. On the other hand, this bound only works for
$c\equiv 0$. 

\begin{proposition}\label{prop:x-bound2}
Suppose the triple $(h,b,c)$ satisfies the following conditions:
\begin{enumerate}
\item $h(r)\equiv {\rm const}\in(-1,0)$ near $r=0$, $h(r)=r-1$ for
  $r\geq 2$, and $h'(r)\geq 0$ for all $r\in\R_+$;
\item $0\leq b'(\eta)\leq 1$ for all $\eta\in\R$;
\item $c\equiv 0$. 
\end{enumerate}
Let $w=(x,\eta):\R\to\LL\times\R$ be a solution of~\eqref{grad} whose
asymptotic limits $x^\pm$ are contained in $V\cup(M\times[1,2))$. Then
$x$ is uniformly bounded by a constant depending only on the action
difference $A_{H,b,c}(x^+,\eta^+)-A_{H,b,c}(x^-,\eta^-)$. 
\end{proposition}

\begin{proof}
By the assumption on $b$ inequality~\eqref{eq:lap-rho} in
Lemma~\ref{lem:lap1} simplifies for $\rho(s,t)\geq\ln 2$ to 
\begin{equation}\label{problem1}
   \Delta\rho(s,t) \geq -\p_s\beta(s),\qquad \beta(s):=b(\eta(s)). 
\end{equation}
Since $c\equiv 0$, the second equation in~\eqref{grad} yields
\begin{align}\label{problem2}
   -\p_s\beta(s) &= -b'(\eta)\p_s\eta \cr
   &= b'(\eta)^2\int_0^1h\bigl(r(s,t)\bigr)dt \cr 
   &= b'(\eta)^2\int_0^1\Bigl(e^{\rho(s,t)}-1\Bigr)dt,
\end{align}
where for the last equality we redefine $\rho$ by $\rho:=\ln(h+1)$.
(This is possible since $h>-1$ and does not change $\rho$ for
$\rho\geq 2$.) The assumption on the asymptotic limits $x^\pm$ implies
$\rho(s,t)<\ln 2$ for $|s|$ large, so
$$
   \Om:=\{(s,t)\in\R\times S^1\mid \rho(s,t)>\ln 2\}
$$
has compact closure. If $b'\bigl(\eta(s)\bigr)=0$ for some $s\in\R$, then
the second equation in~\eqref{grad} implies that $\eta$, hence $\beta$, is constant.
Equation~\eqref{problem1} becomes $\Delta\rho\geq 0$, hence $\rho\leq\ln
2$ by the maximum principle and we are done. So assume from now on
that $b'\bigl(\eta(s)\bigr)>0$ for all $s\in\R$. Then we can uniquely
write $\eta$ as a function of $\beta$, and hence
$b'\bigl(\eta(s)\bigr)^2=f\bigl(\beta(s)\bigr)$ for some smooth 
function $f:\R\to[0,1]$. Then~\eqref{problem2} becomes
$$
   -\p_s\beta(s) =
   f\bigl(\beta(s)\bigr)\int_0^1\Bigl(e^{\rho(s,t)}-1\Bigr)dt. 
$$
Hence the pair $(\rho,\beta)$ satisfies the hypotheses of
Proposition~\ref{upb} below (with $A=1$).  

Pick any $T>\ln 2$ and define the \emph{$T$-wild set} $W_T$ by 
$$
W_T=\Big\{s \in \mathbb{R}: 
\,\,\exists\,\,t,t' \in S^1, \rho(s,t) > T,\,\,
\rho(s,t') < \ln 2\Big\}.
$$
Lemma~\ref{lem:wild} with $A=2$ and $B=e^T$ yields
$$
   \|\nabla A_{H,b,c}(x,\eta)(s)\|_J\geq\frac{e^T-2}{e^{T/2}} =:C_T \
\mbox{ for all } s\in W_T. 
$$
Thus we can estimate the Lebesgue measure $|W_T|$ in terms of the
action difference by
\begin{align*}
   \Delta &:= A_{H,b,c}(x^+,\eta^+)-A_{H,b,c}(x^-,\eta^-) \cr
   &= \int_{-\infty}^\infty \|\nabla A_{H,b,c}(x,\eta)(s)\|_J^2ds \cr 
   &\geq \int_{W_T} \|\nabla A_{H,b,c}(x,\eta)(s)\|_J^2ds \cr 
   &\geq |W_T|C_T^2,
\end{align*}
hence 
$$
   |W_T|\leq \frac{\Delta}{C_T^2} =:\delta. 
$$
Now Proposition~\ref{upb} below (with $A=1$) yields the uniform
estimate 
$$
   \max_{\mathbb{R} \times S^1} \rho \leq 
   T+4 \delta^2,
$$
where the right hand side only depends on the action difference
$\Delta$ (take for example $T=2$). This concludes the proof of
Proposition~\ref{prop:x-bound2} modulo Proposition~\ref{upb} below. 
\end{proof}

\subsection{An upper bound for a Kazdan-Warner type inequality}\label{ss:KW}
Assume that $\Omega$ is an open subset of the 
infinite cylinder
$\mathbb{R} \times S^1$ whose closure is compact,
$A$ is a positive real number, and
$f \in C^\infty(\mathbb{R},[0,A])$. 
We consider solutions 
$$
   (\rho,\beta) \in C^\infty(\mathbb{R}\times S^1,\mathbb{R})
   \times C^\infty(\mathbb{R},\mathbb{R})
$$ of the following Kazdan-Warner type inequality:
\begin{equation}\label{problem}
\left.\begin{array}{cc}
\Delta \rho(s,t) \geq -\partial_s \beta(s)=f(\beta(s))
\int_0^1\Big(e^{\rho(s,\tau)}-1\Big)d\tau, &
(s,t) \in \Omega\\
\rho|_{\mathbb{R}\times S^1 \setminus\Omega} \leq \ln 2.
\end{array}\right\}
\end{equation}
Here we understand that the limit in the last equation of
(\ref{problem}) as $s\to\infty$ is uniform in the $t$-variable. 
Given a solution $(\rho,\beta)$ of
(\ref{problem}) we introduce for a real number
$T > \ln 2$ the \emph{$T$-wild set} of
$(\rho,\beta)$ 
$$W_T=W_T(\rho,\beta) 
\subset \mathbb{R}$$
by 
$$
W_T=\Big\{s \in \mathbb{R}: 
\,\,\exists\,\,t,t' \in S^1, \rho(s,t) > T,\,\,
\rho(s,t') < \ln 2\Big\}.
$$
Assume that $\delta>0$ and
$T> \ln 2$. We will consider solutions $(\rho,\beta)$ of 
(\ref{problem}) for which the Lebesgue measure
of the $T$-wild set of $(\rho,\beta)$ satisfies
\begin{equation}\label{T-tame}
   |W_T| \leq \delta.
\end{equation}
The main result of this section is the following uniform
upper bound for the first factor of solutions of
(\ref{problem}) satisfying~\eqref{T-tame}.

\begin{proposition}\label{upb}
Assume that $\delta>0$, $T>\ln 2$, and
$(\rho,\beta)$ is a solution of (\ref{problem})
satisfying~\eqref{T-tame}. Then 
$$
   \max_{\mathbb{R} \times S^1} \rho \leq 
   T+4A \delta^2.
$$
\end{proposition}

\begin{proof} 
We prove Proposition~\ref{upb} in two steps.
To formulate Step\,1 we introduce
the following superset of the $T$-wild set
$$V_T=\big\{s \in \mathbb{R}:\,\,\exists\,\,t 
\in S^1, \rho(s,t) > T\big\}.$$
Note that its closure $\mathrm{cl}(V_T)$ is compact since the
closure of $\Omega$ is compact. We abbreviate
$$\mathcal{Z}_T=\mathrm{cl}\big((V_T\times S^1) 
\cap \Omega\big)\subset \mathbb{R}\times S^1.$$

\begin{figure}[htp] 
\centering
\includegraphics[scale=1.3]{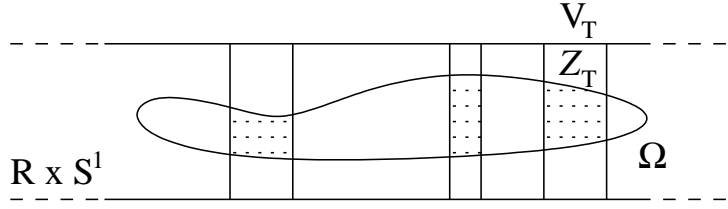}
\caption{Domains on $\R\times S^1$.} \label{fig:proof1}
\end{figure}

We further introduce the following space of functions
$$\mathcal{F}_T=\Big\{\mu \in 
C^2\big(\mathrm{cl}(V_T),[0,\infty)\big):
\Delta \rho+\partial^2_s \mu|_{\mathcal{Z}_T} \geq 0\Big\}.$$
We define the following number
$$c_T=\inf_{\mu \in \mathcal{F}_T}\max_{s \in \mathrm{cl}(V_T)}
\mu(s)
\in [0,\infty).$$
\\
\textbf{Step\,1: }$\max_{\mathbb{R}\times S^1}\rho \leq T+c_T$.
\\ \\
For $\epsilon>0$ pick $\mu \in \mathcal{F}_T$
such that
$$\max_{\mathrm{cl}(V_T)}\mu \leq c_T+\epsilon.$$
We abbreviate
$$\chi=\rho+\mu|_{\mathcal{Z}_T} \in C^2\big(\mathcal{Z}_T,
\mathbb{R}\big).$$
Note that by definition of $\mu$
$$\Delta \chi \geq 0$$
and hence, since $\mathcal{Z}_T$ is compact,
$$\max_{\mathcal{Z}_T}\chi=\max_{\partial \mathcal{Z}_T}
\chi.$$
Using this and the fact that, by definition, $\mu$ is nonnegative,
we estimate
\begin{eqnarray*}
\max_{\mathcal{Z}_T}\rho
&\leq& \max_{\mathcal{Z}_T}(\rho+\mu)\\
&=&\max_{\mathcal{Z}_T}\chi\\
&=&\max_{\partial \mathcal{Z}_T} \chi\\
&\leq& \max_{\partial \mathcal{Z}_T} \rho+
\max_{\partial \mathcal{Z}_T}\mu\\
&\leq&T+c_T+\epsilon.
\end{eqnarray*}
Since on the other hand
$$\max_{(\mathbb{R}\times S^1) \setminus \mathcal{Z}_T} \rho
\leq T \leq T+c_T+\epsilon$$
we conclude that
$$\max_{\mathbb{R}\times S^1}\rho \leq T+c_T+\epsilon.$$
Because $\epsilon$ was arbitrary Step\,1 follows. 
\\ \\
\textbf{Step\,2: }$c_T \leq 4A\delta^2$.
\\ \\
We first assume that 
$\mathrm{cl}(V_T)$ equals an interval
$[0,R]$ for $R>0$. 
We let $\mu:[0,R]\to [0,\infty)$ be the piecewise $C^2$ function which
is uniquely determined by the following conditions.
\begin{itemize}
 \item $\mu(0)=\partial_s \mu(0)=0$,
 \item $\partial^2_s \mu(s)=\max(\partial_s \beta(s),0)$ for
  every $s \in [0,R]$ at which $\partial_s \mu(s)=0$,
 \item $\partial^2_s \mu(s)=\partial_s \beta(s)$ for
  every $s \in [0,R]$ at which $\partial_s \mu(s)>0$.
\end{itemize}
Note that the function $\mu$ is monotone and hence, because
of its initial condition, nonnegative. Moreover, by construction
$\mu$ satisfies
\begin{equation} \label{eq:mu}
\partial^2_s \mu \geq \partial_s \beta
\end{equation}
and hence, by the inequality in the first line of (\ref{problem}),
we conclude that
$$\Delta \rho+\partial^2_s \mu|_{\mathcal{Z}_T} \geq 0.$$
Hence
$$\mu \in \mathcal{F}_T.$$
We set
$R^0=0$ and define recursively for $j \in \mathbb{N}$
\begin{eqnarray*}
   R_j&=&\inf\big\{s \in [R^{j-1},R]: \partial_s \beta(s)>0\big\},\\
   R^j&=&\inf\big\{s \in (R_j,R]: \beta(s)=\beta(R_j)\big\},
\end{eqnarray*}
where we understand here that the infimum of the empty
set is $R$. We refer to Figure~\ref{fig:proof2} for the construction
of the function $\mu$. We denote $I_j:=[R_j,R^j]$, $j\ge 1$, and it follows from
the definition that 
$$
\mathrm{supp}(\p_s\mu)\subset \bigcup_{j=1}^\infty I_j. 
$$
We see in particular that $\mu$ is of class $C^2$ at all points of
$[0,R]$, except for $R^j$, $j\ge 1$. However, the map $\mu$ can be
smoothened at these points while still preserving
condition~\eqref{eq:mu}. 

\begin{figure}[htp] 
\centering
\includegraphics[scale=1]{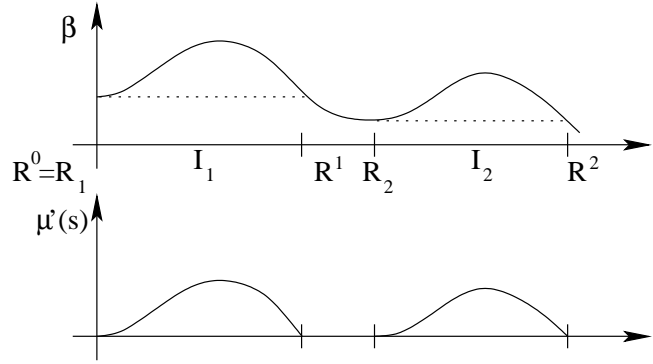}
\caption{The construction of the function $\mu$.} \label{fig:proof2}
\end{figure}

We further put
$$\Delta_j=R^j-R_j \geq 0.$$
We claim the following inequality
\begin{equation}\label{del}
\sum_{j=1}^\infty \Delta_j \leq 2\delta.
\end{equation}
To prove (\ref{del}) we introduce for $s \in \mathbb{R}$
$$
   \kappa(s):=\int_0^1\Big(e^{\rho(s,\tau)}-1\Big)d\tau.
$$
We claim that for every $j \in \mathbb{N}$
\begin{equation}\label{del1}
   \int_{R_j}^{R^j}\kappa(s) ds=0.
\end{equation}
To see this, first note that if $f(\beta(s_0))=0$ for some
$s_0\in[0,R]$, then the equality in the first line of
(\ref{problem}) implies that $\beta$ is constant on $[0,R]$, so
$R_1=R$ and there is nothing to prove. Hence we may assume
$f(\beta(s))>0$ for all $s\in[0,R]$. Pick a function $F:\R\to\R$
with $F'=1/f$. Now $\beta(R_j)=\beta(R^j)$ and 
the equality in the first line of (\ref{problem}) imply the claim:
$$
   0 = F\bigl(\beta(R_j)\bigr) - F\bigl(\beta(R^j)\bigr) 
   = -\int_{R_j}^{R^j}\frac{\p_s\beta}{f(\beta)} ds
   = \int_{R_j}^{R^j}\kappa(s) ds.
$$
Next we define the following four subsets of
$I_j=[R_j,R^j]$:
$$I_j^-=\{s \in I_j: \kappa(s)<0\},\quad
I_j^+=I_j \setminus I_j^-,\quad
I_j^w=I_j \cap W_T, \quad I_j^t=I_j \setminus I_j^w.$$
We observe that, by definition of the wild set and because $I_j\subset
\mathrm{cl}(V_T)$, we have 
\begin{equation}\label{del2}
\kappa|_{I_j^t} \geq 1.
\end{equation}
In particular,
\begin{equation}\label{del3}
I_j^t \subset I_j^+.
\end{equation}
Using (\ref{del1}--\ref{del3}) and the fact that
$\kappa \geq -1$ we estimate
\begin{equation}\label{del4}
0 =\int_{I_j} \kappa
= \int_{I_j^-}\kappa+\int_{I_j^+}\kappa
 \geq \int_{I_j^-}\kappa+\int_{I_j^t}\kappa
 \geq -|I_j^-|+|I_j^t|.
\end{equation}
Taking complements in (\ref{del3}) we obtain
\begin{equation}\label{del5}
I_j^- \subset I_j^w.
\end{equation}
Combining (\ref{del4}) and (\ref{del5}) we conclude
\begin{equation}\label{del6}
|I_j^t| \leq |I_j^w|.
\end{equation}
Hence we estimate using (\ref{del6}) and~\eqref{T-tame} 
$$\sum_{j=1}^\infty \Delta_j=\sum_{j=1}^\infty\big(|I_j^t|
+|I_j^w|\big) \leq 2\sum_{j=1}^\infty|I_j^w|
\leq 2|W_T|=2\delta,$$
proving the inequality claimed in (\ref{del}). By construction
of $\mu$ it holds that
\begin{equation}\label{supp}
\{s \, : \, \p^2_s\mu(s)>0\}\subset \mathrm{supp}\big(\partial_s
\mu\big) \subset \bigcup_{j=1}^\infty I_j.
\end{equation}
Moreover, we observe using the equation in the first line
of (\ref{problem}) and the assumption that the function
$f$ is bounded from above by $A$ that
\begin{equation}\label{ppb}
\max \big(\partial^2_s \mu\big) 
\leq \max\big(\max\, \partial_s \beta, 0\big) \leq A.
\end{equation}
From (\ref{del}),(\ref{supp}), and (\ref{ppb}) we deduce
\begin{equation}\label{pb}
\max\big(\partial_s \mu\big) \leq 2\delta A.
\end{equation}
Combining (\ref{pb}) again with (\ref{del}) and~\eqref{supp} we conclude
that
$$
\max \mu \leq 4A \delta^2,
$$
and hence
$$c_T \leq 4A \delta^2.$$
This finishes the proof of Step\,2 in the case that
$\mathrm{cl}(V_T)$ is an interval. In the general case
$\mathrm{cl}(V_T)$ is a countable union of intervals and, possibly,
accumulation points. Ignoring the latter, we can  
apply the previous construction to each component, always
setting $\mu=\p_s\mu=0$ at the left end point of each interval. This
concludes the proof of Proposition~\ref{upb} and hence of
Proposition~\ref{prop:x-bound2}.  
\end{proof}

\subsection{A useful dichotomy} \label{sec:dichotomy}
In this subsection we prove a general result which says that, under
certain hypotheses, the Morse complex of a functional
$f:X\times\R\to\R$ splits into two subcomplexes for small/large $\eta$
which are not connected by gradient flow lines. In the next subsection
we will apply this result to the Rabinowitz functional
$A_{H,b,c}:\LL\times\R\to\R$. 

Let $(X,g_X)$ be a (maybe infinite dimensional) Riemannian manifold and
$f \colon X \times \mathbb{R} \to \mathbb{R}$ a smooth function. We consider
on the manifold $X\times \mathbb{R}$ the Riemannian metric
$g=g_X \oplus g_\mathbb{R}$
where $g_\mathbb{R}$ is the standard metric on $\mathbb{R}$. We denote
by $\nabla f$ the gradient of $f$ with respect to the metric $g$ and by
$||\cdot||$ the norm with respect to the metric $g$. 
\begin{definition}\label{rabtype}
We say that $f$ is of {\bf Rabinowitz type} if there exists a 
quadruple of positive constants $(\epsilon,D,c_1,c_2)$ such that
for all $(x,\eta) \in X \times \mathbb{R}$ the following two
conditions hold
\begin{description}
 \item[\hspace{-.5cm}(i)] $||\nabla f(x,\eta)|| \leq \epsilon, \,\,
|\eta| \leq D\,\,\Longrightarrow |\eta|
\leq c_1\big(|f(x,\eta)|+||\nabla f(x,\eta)||\big)$;
\item[\hspace{-.5cm}(ii)] $|\partial_\eta f(x,\eta)| \leq c_2$ along gradient flow
  lines connecting critical points.
\end{description}
We refer to the quadruple $(\epsilon,D,c_1,c_2)$ as a 
{\bf Rabinowitz quadruple} for $f$.
\end{definition}
In the following we assume that $f$ is of Rabinowitz type and
$(\epsilon,D,c_1,c_2)$ is a fixed Rabinowitz quadruple for $f$.
Moreover, we assume without loss of generality that 
$$
D>c_1\epsilon.
$$
We define the functions $b_\pm:\R_+\to\R$, 
$$
   b_-(a):=c_1 (a+\epsilon)+\frac{2ac_2}{\epsilon^2}, \qquad
   b_+(a):=D-\frac{2ac_2}{\epsilon^2}.
$$
We further introduce the quantity
$$
   \kappa:=\frac{\epsilon^2(D-c_1\epsilon)}{c_1\epsilon^2+4c_2}
$$
and note that 
$$
   b_-(a) < b_+(a) \qquad\text{for }a<\kappa. 
$$
Now assume that $w=(x,\eta) \in C^\infty(\mathbb{R},X\times \mathbb{R})$ is
a gradient flow line of $f$, i.e. a solution of the ODE
$$
   \partial_s w(s)=\nabla f(w(s)), \quad s \in \mathbb{R}.
$$
For a gradient flow line we abbreviate
$$A_w=\sup_{s \in \mathbb{R}}\big|f(w(s))\big|.$$
We refer to $A_w$ as the {\bf absolute action} of $w$.
The following proposition shows that there is a dichotomy of gradient
flow lines of Rabinowitz type functions
of small absolute action into gradient flow lines
with large $|\eta|$ and small $|\eta|$.

\begin{proposition}\label{dicho}
Suppose that $f$ is of Rabinowitz type with Rabinowitz quadruple
$(\epsilon,D,c_1,c_2)$ and define $b_\pm(a),\kappa$ as above. Let
$w=(x,\eta) \in C^\infty(\mathbb{R},X\times \mathbb{R})$ be a gradient
flow line  of $f$ with absolute action $A_w=a<\kappa$. Then exactly one
of the following two cases holds:
\begin{description}
 \item[(i)] $|\eta(s)|\leq b_-(a)$ for all $s \in 
\mathbb{R}$;
 \item[(ii)] $|\eta(s)|> b_+(a)$ for all $s \in 
\mathbb{R}$.
\end{description}
\end{proposition}

\begin{proof} 
For $\sigma \in \mathbb{R}$ we abbreviate
$$
   \tau(\sigma)=\inf\big\{\tau \geq 0: ||\nabla f(w(\sigma+\tau))||
   \leq \epsilon\big\}.
$$
Using the gradient flow equation we estimate
\begin{eqnarray*}
   \epsilon^2 \tau(\sigma)&\leq&\int_{\sigma}^{\sigma+\tau(\sigma)}
   ||\nabla f(w(s))||^2 ds\\
   &=& \int_{\sigma}^{\sigma+\tau(\sigma)}\frac{d}{ds}f(w(s))ds\\
   &=& f\big(w(\sigma+\tau(\sigma))\big) - f\big(w(\sigma)\big)\\
   &\leq&2A_w,
\end{eqnarray*}
implying that
\begin{equation}\label{lag1}
\tau(\sigma) \leq \frac{2A_w}{\epsilon^2}.
\end{equation}
Since the metric $g$ on $X \times \mathbb{R}$ is of product
form, the gradient flow equation for $\eta$ reads
$$
   \partial_s \eta=\partial_\eta f(w).
$$
Hence using assertion~(ii) in Definition~\ref{rabtype} we obtain that
\begin{equation}\label{lag2}
|\partial_s \eta| \leq c_2.
\end{equation}
Using (\ref{lag1}) and (\ref{lag2}) we get
\begin{equation}\label{lag3}
\big|\eta(\sigma)-\eta(\sigma+\tau(\sigma))\big| \leq
\int_{\sigma}^{\sigma+\tau(\sigma)}|\partial_s \eta| ds
\leq c_2 \tau(\sigma)
\leq \frac{2c_2 A_w}{\epsilon^2}.
\end{equation}
We distinguish the following two cases.
\\ \\
\textbf{Case 1:} $\big|\eta(\sigma+\tau(\sigma))\big|>D$.
\\ \\
In this case we obtain using (\ref{lag3}) that
$$\big|\eta(\sigma)\big| \geq \big|\eta(\sigma+\tau(\sigma))\big|
-\frac{2c_1A_w}{\epsilon^2}>D-\frac{2c_2A_w}{\epsilon^2}
=b_+(A_w).$$
\textbf{Case 2:} $\big|\eta(\sigma+\tau(\sigma))\big| \leq D$.
\\ \\
In this case we estimate using the assertion~(i) in
Definition~\ref{rabtype} and again inequality (\ref{lag3})
$$
   \big|\eta(\sigma)\big| \leq \big|\eta(\sigma+\tau(\sigma))\big|
   +\frac{2c_2A_w}{\epsilon^2}\leq c_1(A_w+\epsilon)+\frac{2c_2
   A_w}{\epsilon^2} = b_-(A_w).
$$
Since $b_+(A_w)>b_-(A_w)$, continuity of $\eta$ implies that
either $|\eta(s)| \leq b_-(A_w)$ for all $s \in \mathbb{R}$ or 
$|\eta(s)| >b_+(A_w)$ for all $s \in \mathbb{R}$. This proves 
Proposition~\ref{dicho}. 
\end{proof}

\subsection{Proof of Theorem~\ref{thm:flatten}}
Let us fix $-\infty<\alpha<\beta<\infty$ such that
$\alpha,\beta\notin\Spec(M,\lambda)$, and a function $h$ satisfying $h'(r)\ge 0$
for $r>0$ and $h(r)=r-1$ for 
$r\geq 1/2$. We consider Rabinowitz functionals $A_{H,b,c}$ for 
$H=h(r)$, $c\equiv 0$, and $b$ satisfying 
\begin{equation}\label{eq:b}
   0\leq b'(\eta)\leq 1,\qquad b(\eta)=\eta \text{ for }|\eta|\leq D, 
\end{equation}
with a constant 
$$
   D > 2\max(|\alpha|,|\beta|)
$$ 
to be determined later. 

Recall that critical points $(x,\eta)$ of $A_{H,b,0}$ with
$x=(y,r)$ satisfy 
$$
\left\{\begin{array}{l} 
\dot y = b(\eta) h'(r)R, \\ 
b'(\eta) h(r) = 0.
\end{array}\right.
$$
The action of such  a critical point is
$$
A_{H,b,0}(x,\eta)=b(\eta)(rh'(r)-h(r)).
$$ 
(in the region $\widehat V\setminus M\times \R_+$, the formula is to
be read via the convention $h'(0)=0$ and $h(0)=\mathrm{ct.}$). 
There are two types of critical points as above:  

Type~1: $b'(\eta)\neq 0$. In this case $h(r)=0$ and hence $r=1$ and
$A_{H,b,0}(x,\eta)= b(\eta)$. 

Type~2: $b'(\eta)=0$. In this case suppose first $r\geq 1/2$. Then the action 
$A_{H,b,0}(x,\eta)=b(\eta)$ cannot lie in the interval
$(\alpha,\beta)$ because $b(\eta)=\eta$ on this 
interval. Hence we must have $r\le 1/2$, so that $h(r)\le -1/2$ and
therefore $rh'(r)-h(r)\ge -h(r) \ge 1/2$. This implies $|A_{H,b,0}|\ge
|b(\eta)|/2\ge D/2$, and the action does not lie in $(\alpha,\beta)$. 

We thus proved that all critical points with action in
$(\alpha,\beta)$ satisfy $r= 1$.   
Hence by Proposition~\ref{prop:x-bound2} there exists a constant $c_2$
depending only on $\beta-\alpha$ such that $r(s,t)\leq c_2$ along all
gradient flow lines of $A_{H,b,0}$ connecting critical points
with action in $(\alpha,\beta)$. 

By the proof of Proposition~3.2 in~\cite{CF} there exist constants
$\eps,c_1>0$ such that for the unperturbed Rabinowitz functional the
following holds:
$$
   ||\nabla A_{H,\eta,0}(x,\eta)|| \leq \epsilon \Longrightarrow |\eta|
   \leq c_1\big(|A_{H,\eta,0}(x,\eta)|+||\nabla
   A_{H,\eta,0}(x,\eta)||\big). 
$$
Since $b(\eta)=\eta$ for $|\eta|\leq D$, for the perturbed Rabinowitz
functional this implies
$$
   ||\nabla A_{H,b,0}(x,\eta)|| \leq \epsilon,\ |\eta|\leq D
   \Longrightarrow |\eta| 
   \leq c_1\big(|A_{H,b,0}(x,\eta)|+||\nabla
   A_{H,b,0}(x,\eta)||\big). 
$$
This shows that $A_{H,b,0}$ is of Rabinowitz type in the sense of the
previous subsection with Rabinowitz quadruple
$(\eps,D,c_1,c_2)$. Moreover, note the crucial fact that the constants
$\eps,c_1,c_2$ work for all functions $b$ satisfying~\eqref{eq:b},
independently of $D$! Thus we can choose $D=D(\alpha,\beta)$ so large
that the quantities $\kappa$ and $b_\pm(a)$ defined as in the previous
section satisfy  
$$
   \kappa>\max(|\alpha|,|\beta|),\qquad
   b_-(\kappa)>\max(|\alpha|,|\beta|). 
$$
This implies that for $a<\kappa$ we have $b_-(a)<b_+(a)$ and 
$$
   b_+(a) > b_+(\kappa)\geq b_-(\kappa) > \max(|\alpha|,|\beta|). 
$$
On the other hand, the discussion above shows that all critical points
$(x,\eta)$ with action in $(\alpha,\beta)$ satisfy
$$
   A_{H,b,0}(x,\eta) = b(\eta) = \eta \in (\alpha,\beta). 
$$
Now we apply Proposition~\ref{dicho} to the functional
$A_{H,b,0}:U\to\R$ with 
$$
   U:= \{(x,\eta)\in\LL\times\R\mid
   A_{H,b,0}(x,\eta)\in(\alpha,\beta)\}. 
$$ 
The preceding discussion shows that case~(ii) in
Proposition~\ref{dicho} does not occur, so along every gradient flow
line connecting critical points in $(\alpha,\beta)$ we have
$$
   |\eta(s)|\leq b_-(\kappa).
$$
Now a short computation shows
$$
   b_-(\kappa) = D\left(\frac{1}{2}+O(\eps)\right) + O(\eps),
$$
where the implicit constants in $O(\eps)$ depend only on $c_1$ and
$c_2$. Assuming without loss of generality that $\eps$ is sufficiently
small, we may therefore assume $b_-(\kappa) < D$. The resulting
estimate  
$$
   |\eta(s)| < D
$$
shows that gradient flow lines of $A_{H,b,0}$ connecting critical
points with action in $(\alpha,\beta)$ stay in the region where
$b(\eta)=\eta$ and therefore agree with gradient flow lines of the
unperturbed Rabinowitz functional $A_{H,\eta,0}$. Hence the chain
complexes of $A_{H,b,0}$ and $A_{H,\eta,0}$ in the action interval
$(\alpha,\beta)$ coincide and we conclude
$$
   FH^{(\alpha,\beta)}(A_{H,\eta,0}) \cong
   FH^{(\alpha,\beta)}(A_{H,b,0}). 
$$
This conclusion holds for all functions $b$ satisfying~\eqref{eq:b},
in particular for functions $b=b_A$ as in Lemma~\ref{lem:tame} with
$A\geq D(\alpha,\beta)+1$. This concludes the proof of
Theorem~\ref{thm:flatten}.  \hfill $\square$

\section{Proof of the main result}\label{sec:proof}

In this section we prove Theorem~\ref{thm:main1} in the Introduction,
i.e. the isomorphism 
\begin{equation*} 
RFH_*(V) \simeq \Check{SH}_*(V).
\end{equation*}

For this we consider deformed 
Rabinowitz action functionals of the form
$$
A_{H,b,c}(x,\eta) = \int_0^1 x^*\lambda - b(\eta) \int_0^1
H(x(t)) dt + 
c(\eta),
$$
with $H:\wh V\to\R$ and $b,c:\R\to \R$ smooth
functions. 

We start with a summary of what we achieved in
Sections~\ref{sec:perturbations} and~\ref{sec:prep}. Let
$(\alpha,\beta)$ be a fixed action interval, such that
$|\alpha|,|\beta|\notin \mathrm{Spec}(M,\lambda)$. By definition, we have 
$RFH_*^{(\alpha,\beta)}(V)=FH_*^{(\alpha,\beta)}(A_{H,\eta,0})$ for a
defining Hamiltonian $H:\widehat V\to \R$ which is constant at
infinity. In the sequel, we only consider Hamiltonians of the form
$H=h(r)$, where $r\in\R_+$ is the second coordinate in the
symplectization $M\times \R_+$. By Proposition~\ref{prop:RFH-lin}, we
also have
$$
RFH_*^{(\alpha,\beta)}(V)=FH_*^{(\alpha,\beta)}(A_{H,\eta,0}),
$$ 
for a
Hamiltonian $H=h(r)$ with $h(r)=r-1$ for $r\ge 1/2$. By
Theorem~\ref{thm:flatten} we have 
\begin{equation} \label{eq:Hbzero}
FH_*^{(\alpha,\beta)}(A_{H,\eta,0})=FH_*^{(\alpha,\beta)}(A_{H,b,0}),
\end{equation}
with $b=b_A$ as in~\eqref{eq:bA} and 
$A\notin\mathrm{Spec}(M,\lambda)$ sufficiently large. Finally,
by Proposition~\ref{prop:RFH-flat} the right hand side
of~\eqref{eq:Hbzero} is isomorphic to 
$$
FH_*^{(\alpha,\beta)}(A_{\widetilde H,\widetilde b,c})
$$
for every triple $(\widetilde H,\widetilde b,c)$ satisfying~\eqref{eq:hbc} that
can be connected to $(H,b,0)$ by a homotopy of such triples during
which the critical points do not cross the boundary of the action
interval $(\alpha,\beta)$, and form a compact set. 

In this section we construct a homotopy to a triple $(\widetilde
H,1,c)$ for which $FH_*^{(\alpha,\beta)}(A_{\widetilde
H,1,c})=\check{SH}_*^{(\alpha,\beta)}(V)$. We achieve this in five
steps, as follows.  

\noindent {\bf Step~1.} We rescale $(H,b,0)$ to $(\mu H, b/\mu,0)$ for
a suitable $\mu>0$. 

\noindent {\bf Step~2.} We replace $(\mu H, b/\mu,0)$ by $(\mu H,
b/\mu,c)$ for a suitable $c$ (Figure~\ref{fig:c}). 

\noindent {\bf Step~3.} We replace $(\mu H,b/\mu,c)$ by $(K,b/\mu,c)$,
where $K$ is a flattening of $H$ near $r=1$ (Figure~\ref{fig:HK}).

\noindent {\bf Step~4.} We homotope the action functional
$A_{K,b/\mu,c}$ to $A_{|K|,1,c}$. 

\noindent {\bf Step~5.} By our special choice of $c$ and $K$, we have 
$$
FH_*^{(\alpha,\beta)}(A_{|K|,1,c})=FH_*^{(\alpha,\beta)}(|K|)
=\check{SH}_*^{(\alpha,\beta)}(V).
$$

\begin{proof}[Proof of Theorem~\ref{thm:main1}] 
Let $-\infty < \alpha < \beta < +\infty$ be fixed, such that
$|\alpha|,|\beta|\notin \mathrm{Spec}(M,\lambda)$. Given 
$0<\delta<1$ and $\mu>0$ we define the following
Hamiltonians. Let $H=H_\delta:\wh V\to\R$ be such that, up
to smooth approximation in the neighbourhood of $\{\delta\}\times M$,
it satisfies (Figure~\ref{fig:HK}) 
$$
\left\{\begin{array}{ll}
H\equiv \delta-1 & \mbox{on } V\setminus [\delta,1]\times M, \\
H(r,x)=h(r) & \mbox{on } [\delta,\infty)\times M, \qquad \qquad \mbox{with
} h(r)=r-1.
\end{array}\right.
$$
Let $K=K_{\delta,\mu}:\wh V\to \R$ be such that, up to smooth
approximation in the neighbourhood of $\{\delta\}\times M$ and
$\{1\}\times M$, it satisfies (Figure~\ref{fig:HK}) 
$$
\left\{\begin{array}{ll}
K\equiv \mu(\delta-1) & \mbox{on } V\setminus \, [\delta,1]\times M, \\
K(r,x)=k(r) & \mbox{on } [\delta,\infty)\times M, 
\end{array}\right.
$$
with $k(1)=k'(1)=0$ and 
$$
\left\{\begin{array}{cl}
k'>0 & \mbox{for } r\in]\delta,\infty)\setminus \{1\}, \\
k(r)=\mu(r-1) & \mbox{outside small neighbourhoods of } \delta \mbox{
and } 1. 
\end{array}\right. 
$$

\begin{figure}[htp] 
\centering
\includegraphics[scale=.6]{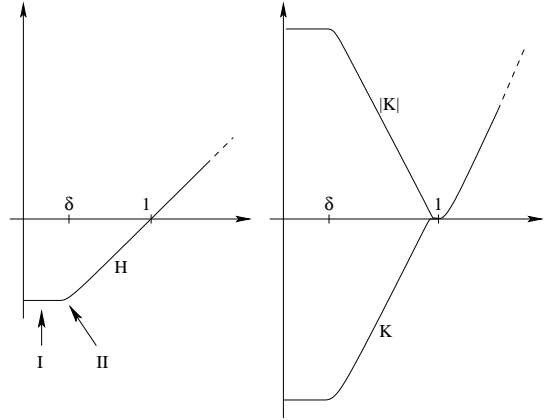}
\caption{The Hamiltonians $H$, $K$, and $|K|$.} \label{fig:HK} 
\end{figure} 

We define now $c:\R\to [0,\infty)$ to be an even function such that
(Figure~\ref{fig:c}) 
\begin{equation} \label{eq:c}
\left\{\begin{array}{l}
c(0)=0, \\
0<\sup c =  c_0<\infty,  \\
c'(\eta)=0 \mbox{ iff } \eta=0. 
\end{array}\right.
\end{equation} 

\begin{figure}[htp] 
\centering
\input{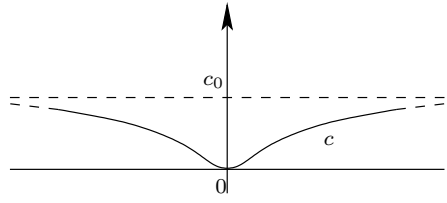}
\caption{A suitable function $c$.} \label{fig:c} 
\end{figure}

We recall that, given $A>0$, we defined in Section~\ref{sec:nearcrit} 
the function $b=b_A:\R\to \R$ such that, up to smooth
approximation in the neighbourhood of $\eta=\pm A$, it satisfies the
conditions (Figure~\ref{fig:b}) 
$$
\left\{\begin{array}{ll} 
b\equiv -A, & \eta\le-A,\\ 
b(\eta)=\eta, & -A\le\eta\le A,\\
b\equiv A, & A\le\eta.
\end{array}\right.
$$

\begin{figure}[htp] 
\centering
\input{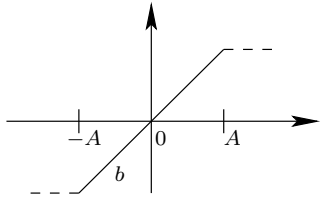}
\caption{The function $b$.} \label{fig:b} 
\end{figure} 



For any $p\notin\mathrm{Spec}(M,\lambda)$,
we recall that $\eta_p>0$ is the distance from $p$ to the closed set 
$\mathrm{Spec}(M,\lambda)$. Let $A(\alpha,\beta)>0$ be the constant
from Theorem~\ref{thm:flatten}, and let us denote $T_0=\min
\mathrm{Spec}(M,\lambda) >0$.  

\noindent {\bf Assumptions.} 
\begin{itemize}
\item \emph{(assumption on $\mu$ and $A$)} We require that
$\mu=A\notin\mathrm{Spec}(M,\lambda)$ and 
$$
\mu=A\ge 10\max\{|\alpha|,|\beta|,1,T_0,A(\alpha,\beta)\}.
$$ 
\hspace{-1.3cm} (we distinguish in notation between $\mu$ and $A$
since they play different roles).




\item \emph{(assumption on $\delta$)} We require $\delta \le \frac 1 2$.
\item \emph{(assumption on $b$)} The smoothing of $b$ takes place in small enough intervals near
$\pm A$, so that 
$$
0<b'(\eta) <1 \ \Longrightarrow \ b(\eta)\in [-A,-A+\eta_A) \, \cup \,
(A-\eta_A,A]. 
$$

\item Let 
$$
d_{A,\alpha,\beta}:=\min \{|p-q| \ : \ p\neq q\in
\mathrm{Spec}(M,\lambda) \cup \{ 0,\alpha,\beta \}, \ p,q\le A\} >0.
$$
Let $\eps_0>0$ be such that 
$$
\mu-\eta_\mu \le (1-\eps_0) \mu <\mu.
$$
Let 
\begin{equation} \label{eq:eps} 
\eps_{A,\alpha,\beta}:=d_{A,\alpha,\beta}/10A. 
\end{equation} 
It follows from the definition of $d_{A,\alpha,\beta}$ that
$\eps_{A,\alpha,\beta}$ satisfies the condition 
\begin{equation} \label{eq:eps-bis}
|r-1|<\eps_{A,\alpha,\beta} \ \Rightarrow \ |rp-p|<d_{A,\alpha,\beta},
\ \forall \ p\in \mathrm{Spec}(M,\lambda) \, \cap \, (0,A).
\end{equation} 
\item \emph{(assumption on $c$)} We require that 
$$
\sup |c|\le \min \{ \frac 1 2, \ d_{A,\alpha,\beta}/10 \}
$$ 
and 
$$
\sup |c'|\le \min\{ \frac 1 2, \ \eps_0
\eps_{A,\alpha,\beta}d_{A,\alpha,\beta}/10 A^2 \}. 
$$
\item \emph{(assumption on $k$)} Let $c_1:=\frac 1 2 \min\{|c'(\eta)| \ : \ |\eta -A| \le \eta_A
\mbox{ or } |\eta+A|\le \eta_A\}$. Let $\bar b(\eta):=b(\eta)/\mu$,
and let $c_2>0$ be such that
$$
|c'(\eta)/\bar b'(\eta)|\le c_2 \ \Rightarrow \ |\bar b(\eta)|< T_0/\mu.
$$ 
Let $c_3:=\min(c_1,c_2)$. We require that, in the neighbourhood of
$r=1$, we have 
$$
|k(r)|\ge c_3 \ \Rightarrow \ k'(r)=\mu.
$$
(This means that $k$ differs from a linear function only in
a small neighbourhood of $r=1$). 
\end{itemize} 

\bigskip 

Before proving Steps~1 to~5, we recall that critical points of
$A_{H,b,c}$ are pairs $(x,\eta)$ such that  
$$
\left\{\begin{array}{l} 
\dot x = b(\eta) X_H, \\ 
b'(\eta)\int H(x(t))dt = c'(\eta).
\end{array}\right.
$$
This is equivalent to the following: 
\renewcommand{\theenumi}{\arabic{enumi}}
\begin{enumerate} 
\item Either $b'(\eta)=0$, in which case $c'(\eta)=0$ and
  $\dot x  = b(\eta) X_H$,
\item or $b'(\eta)\neq 0$, in which case $\dot
  x=b(\eta)X_H$ and 
$$
x(t) \in H^{-1}(c'(\eta) / b'(\eta)), \quad t\in S^1.
$$
\end{enumerate} 
We refer to $(x,\eta)$ as a critical point of Type
(1), respectively Type (2). In case $H=h(r)$, critical points
$(x,\eta)$ appear on levels $r=\mathrm{ct.}$ and have action 
$$
A_{H,b,c}(x,\eta) = b(\eta)\big(rh'(r)-h(r)\big) +c(\eta).
$$
As above, we use the convention $h'(0)=0$ and $h(0)=H(x)$. 

\medskip\noindent  

\medskip\noindent{\bf Step~1.} 
{\it We have a canonical identification of complexes
$$
CF_*^{(\alpha,\beta)}(A_{H,b,0}) = CF_*^{(\alpha,\beta)}(A_{\mu
H,b/\mu,0}) 
$$
}

This follows directly from the equality $A_{H,b,0}=A_{\mu H,b/\mu,0}$.
\hfill{$\square$} 

\medskip\noindent  

{\bf Remark.} From now on the functions $\mu H$ and $b/\mu$ remain unchanged
  outside a compact set, so that moduli spaces of gradient
  trajectories are compact modulo breaking
  (Proposition~\ref{prop:RFH-flat}) and the Rabinowitz-Floer
  homology groups are well-defined.

\medskip\noindent{\bf Step~2.} 
{\it There is a chain map inducing an isomorphism in homology
$$
CF_*^{(\alpha,\beta)}(A_{\mu H,b/\mu,0}) \to
CF_*^{(\alpha,\beta)}(A_{\mu H,b/\mu,c}).  
$$
}

We denote $\bar H:=\mu H$, $\bar h:=\mu h$, $\bar b:=b/\mu$, and
consider the homotopy 
$A_s=A_{\bar H,\bar b,\rho(s)c}$, with $s\in\R$ and $\rho:\R\to [0,1]$
a smooth increasing function, equal to $0$ near $-\infty$ and equal to
$1$ near $+\infty$. To prove that the induced chain map is a
quasi-isomorphism it is enough to examine the action of the critical
points of each $A_s$, and show that none of these actions crosses the
boundary of the interval $(\alpha,\beta)$. We prove that the actions
of critical points of $A_{\bar H,\bar b,c}$ stay away from $\alpha$
and $\beta$ (the same computation shows that this holds true for any
$A_s$ since $0\le \rho\le 1$).  

A critical point $(x,\eta)$ of $A_{\bar H,\bar b,c}$ with $x$ on level
$r$ satisfies 
$$
\left\{\begin{array}{l} 
\dot x=\bar b(\eta)X_{\bar H},\\ 
\bar b'(\eta)\bar h(r)=c'(\eta),
\end{array}\right. 
$$
so that $|\eta|< A$ and $\bar b'(\eta)\neq 0$. 
We distinguish several cases according to the value of the strictly
increasing function $f(\eta)=c'(\eta)/\bar b'(\eta)$ on the interval
$(-A,A)$ (see Figure~\ref{fig:bc}). We
have $\bar h(r)=f(\eta)$ and the cases that we distinguish correspond to
the various types of orbits of $\bar H$. 

Case~1. $f(\eta)=\mu(\delta -1)\le -\frac \mu 2$. Then we must
have $\eta\in[-A,-A+\eta_A]$, because $f(\eta)=\mu c'(\eta)>-\frac \mu
2$ for $\eta\in[-A+\eta_A,0]$. As a consequence $\bar
b(\eta)$ is close to $-1$ (here we use for the first time that
$\mu=A$). Moreover 
$X_{\bar H}=0$, $x$ is constant and the action is $A_{\bar H,\bar
b,c}(x,\eta)=-\bar b(\eta)\bar h(r)+c(\eta)\le -\frac \mu 2 + \frac 1
2 < \alpha-1$.   

Case~2. $f(\eta)$ is close to $\mu(\delta -1)$ and the slope of
$\bar h$ varies between $0$ and $\mu$, so that $x$ lives on a level
$r$ close to $\delta$. As in Case~1 we get that $\eta\in[-A,-A+\eta_A]$,
$\bar 
b(\eta)$ is close to $-1$ and the action satisfies $A_{\bar H,\bar
b,c}(x,\eta) = \bar b(\eta) (\bar h'(r)r-h(r))+c(\eta)< -\bar
b(\eta)\bar h(r) +c(\eta) < \alpha-1$. 

%

Case~3. $f(\eta)$ is bigger than $\mu(\delta-1)$, and 
$\bar h'=\mu$. Then $\eta\in [-A+\eta_A,A-\eta_A]$, $\bar
b(\eta)=\eta/\mu$ and $\bar b'(\eta)=1/\mu$ (otherwise $\bar
b(\eta)\simeq \pm 1$ and, since $\mu\notin \mathrm{Spec}(M,\lambda)$,
there are no solutions of $\dot x=\bar b(\eta)\bar h'(x)R_\lambda = \bar
b(\eta)\mu R_\lambda$). Such a critical point $(x,\eta)$ satisfies
$$
\left\{\begin{array}{l} 
\dot x=\eta R_\lambda,\\ 
\bar H(x)=\mu c'(\eta),
\end{array}\right. 
$$
and the set of critical points in Case~5 is in bijective
correspondence with the set of critical points of $A_{H,b,0}=A_{\bar
H,\bar b,0}$ with action in $[-A+\eta_A,A-\eta_A]$ (the correspondence
assigns to $(x,\eta)$ the pair $(\widetilde x,\eta)$, where
$\widetilde x$ represents the same characteristic as $x$, but located
on the level $r=1$ instead of $r=\bar h^{-1}(\mu c'(\eta))=1+c'(\eta)$). The
action is $A_{\bar H,\bar b,c}(x,\eta) = \eta r - \eta c'(\eta)
+c(\eta)=\eta+c(\eta)$.  
By assumption we have $\sup |c|\le d_{A,\alpha,\beta}/2$, and this
ensures that  
$$
|A_{\bar H,\bar b,c}(x,\eta) - A_{\bar H, \bar b, 0}(\widetilde
x,0)|\le d_{A,\alpha,\beta}. 
$$
The action of the critical points therefore varies by at most
$d_{A,\alpha,\beta}$ during the deformation, and cannot cross the
bounds of the interval $(\alpha,\beta)$.  \hfill{$\square$}

\begin{figure}[htp] 
\centering
\input{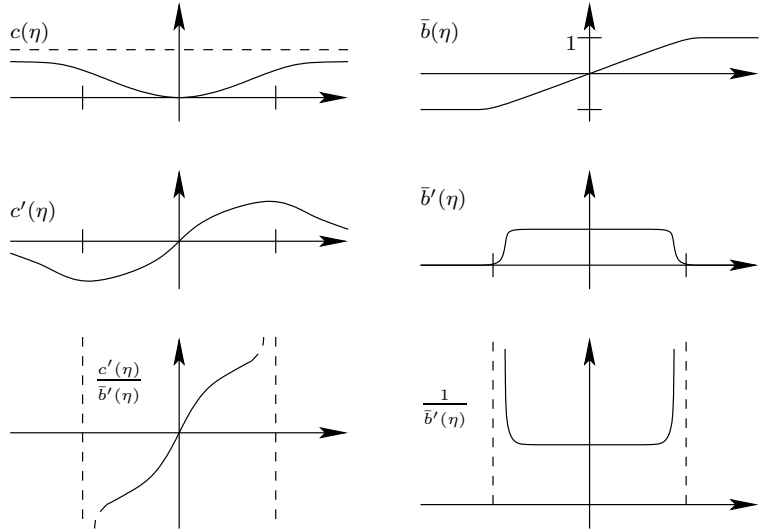}
\caption{The function $c'/\bar b'$.} \label{fig:bc} 
\end{figure}

\medskip\noindent  

\medskip\noindent{\bf Step~3.} 
{\it There is a chain map inducing an isomorphism in homology
$$
CF_*^{(\alpha,\beta)}(A_{\mu H,b/\mu,c}) \to
CF_*^{(\alpha,\beta)}(A_{K,b/\mu,c}).  
$$
}

We use the notation $\bar H=\mu H$, $\bar h=\mu h$, $\bar b=b/\mu$
from Step~2, and recall that $K=k(r)$ on $[\delta,\infty)$. Let 
$\rho:\R\to[0,1]$ be a smooth increasing function equal to $0$ near
$-\infty$ and equal to $1$ near $+\infty$. We consider the
homotopy $A_s=A_{H_s,\bar b,c}$, $s\in \R$ with $H_s=(1-\rho(s))\bar 
H+\rho(s)K$. To prove that the induced chain map is a
quasi-isomorphism we show that, under our assumptions on the
perturbation $k$ of $\bar h$, the critical 
points of $A_s$ stay fixed during the homotopy together with their
action. We prove this for $A_\infty=A_{K,\bar b,c}$, and the same
computations work for any $s\in\R$ since we do not use the fact that
$k'(1)=0$.  

A critical point $(x,\eta)$ of $A_{K,\bar b,c}$ with $x$ on level $r$
satisfies  
$$
\left\{\begin{array}{l} \dot x=\bar b(\eta)X_K,\\
\bar b'(\eta) k(r)=c'(\eta),
\end{array}\right. 
$$
so that $|\eta|<A$ and $\bar b'(\eta)\neq 0$. As in Step~2 we
distinguish several cases according to the value of
$f(\eta)=c'(\eta)/\bar b'(\eta)$, and it is clear that critical points
falling in Cases~1 and~2 of Step~2 are the same, with the same action
throughout the homotopy. 

We now examine critical points $(x,\eta)$ such that
$f(\eta)$ is bigger than $\mu(\delta-1)$. The critical points such
that $k'(r)=\mu$ are the same as those of $A_{\bar H,\bar b,c}$, and
have the same action. Thus, the relevant new situation
is when the value of $k'(r)$ is strictly smaller than $\mu$. 

We first claim that $\eta\in[-A+\eta_A,A-\eta_A]$. Otherwise $\bar
b(\eta)\simeq\pm 1$ and $|k(r)|=|c'(\eta)/\bar 
b'(\eta)|\ge c_1$. Then, by assumption, we have $k'(r)=\mu\notin
\mathrm{Spec}(M,\lambda)$, so 
that there are no solutions of $\dot x = \bar b(\eta) X_K=\bar b(\eta)
\mu R_\lambda$. 

We now claim that the only critical points $(x,\eta)$ such that $x$
lies on a level $r>\delta$ with $k'(r)<\mu$ are of the
form $(x,0)$ and satisfy $r=1$ (moreover, they have the same
action along the homotopy). Indeed, by our assumption on $k$
we must have $|k(r)|\le c_2$, so that $|\bar b(\eta)|\mu$ is smaller
than $T_0$ and any solution of the equation $\dot x = \bar
b(\eta)k'(r)R_\lambda$ must be constant. Thus, we either have
$k'(r)=0$ and $r=1$, or $\bar b(\eta)=0$ and $\eta=0$.  These
conditions are equivalent because $\bar b'(\eta)k(r)=c'(\eta)$, and
this completes the proof.  
\hfill{$\square$}

\medskip\noindent  

\medskip\noindent{\bf Step~4.} 
{\it There is a chain map inducing an isomorphism in homology
$$
CF_*^{(\alpha,\beta)}(A_{K,b/\mu,c}) \to
CF_*^{(\alpha,\beta)}(A_{|K|,1,c}).  
$$
}

For $0\le \eps \le 1$ we denote 
\begin{eqnarray*}
A_\eps(x,\eta)&:=&(1-\eps)A_{K,\bar b,c}(x,\eta)+\eps
A_{|K|,1,c}(x,\eta) \\
&=& \int x^*(r\lambda)-(1-\eps)\bar b(\eta)\int k(r) - \eps \int
|k(r)| + c(\eta).
\end{eqnarray*}
Let $\rho:\R\to[0,1]$ be a smooth increasing function, equal to $0$
near $-\infty$ and equal to $1$ near $+\infty$. We prove the following
claim, which implies that the chain map induced by the homotopy
$A_{\rho(s)}$, $s\in \R$ is a quasi-isomorphism. 

\begin{center}
\emph{For each critical point $(x^\eps,\eta^\eps)$ of $A_\eps$ there is a
critical point $(x^1,0)$ of $A_1$ such that 
$|A_\eps(x^\eps,\eta^\eps)-A_1(x^1,0)|\le d_{A,\alpha,\beta}$.}
\end{center}

The claim is obvious for $\eps=1$, so that we can assume without loss
of generality that $0\le \eps<1$. The equations for a critical point
$(x^\eps,\eta^\eps)$ of $A_\eps$ on level $r^\eps$ are
$$
\left\{\begin{array}{l}
\dot x^\eps = \big((1-\eps)\bar b
(\eta^\eps)k'(r^\eps)+\eps|k|'(r^\eps)\big)R_\lambda, \\
(1-\eps)\bar b'(\eta^\eps)k(r^\eps)=c'(\eta^\eps). 
\end{array}\right. 
$$
We denote $\bar b_\eps:=(1-\eps)\bar b + \eps$ and distinguish three
cases. 

{\it Case~I.} $\eta^\eps=0$. Then $r^\eps =1$, $k'(r^\eps)=|k|'(r^\eps)=0$,
$x^\eps=\mathrm{ct.}$ and $A_\eps(x^\eps,0)=0$. The claim holds with
$x^1:=x^\eps$.   

{\it Case~II.} $\eta^\eps>0$. Then $r^\eps>1$, $|k|(r^\eps)=k(r^\eps)$, we
have 
$$
\left\{\begin{array}{l} \dot
    x^\eps = \bar b_\eps (\eta^\eps)k'(r^\eps) R_\lambda, \\ (1-\eps)\bar
    b'(\eta^\eps) k(r^\eps)=c'(\eta^\eps), 
  \end{array}\right.
$$ 
and the action is $A_\eps(x^\eps,\eta^\eps)=\bar
b_\eps(\eta^\eps)(r^\eps k'(r^\eps)-k(r^\eps))+c(\eta^\eps)$. 

{\it Case~II.i.} If $0\le \eps < 1-\eps_0$ we distinguish the following
cases. 

{\it Case~II.i.1.} $\eta^\eps\in(0,A-\eta_A]$. Then $c'/(1-\eps)\bar b'\le\mu
\sup |c'| /\eps_0$ is so small that $\bar b_\eps
(\eta^\eps)k(r^\eps)\le \mu k(r^\eps)\le
d_{A,\alpha,\beta}/10$. Moreover, we have $\mu(r^\eps-1)\simeq
k(r^\eps) = c(\eta^\eps)/(1-\eps)\bar b'(\eta^\eps) \le \mu \sup
|c'|/\eps_0 \le d_{A,\alpha,\beta}/10\mu$, which implies $1< r^\eps \le
1+d_{A,\alpha,\beta}/10\mu^2 \le 1+\eps_{A,\alpha,\beta}$, where
$\eps_{A,\alpha,\beta} = d_{A,\alpha,\beta}/10\mu$ was introduced
in~\eqref{eq:eps}. Since $1\le r^\eps \le 1+\eps_{A,\alpha,\beta}$ we
obtain  $A_\eps(x^\eps,\eta^\eps)\simeq \bar 
b_\eps(\eta_\eps)k'(r^\eps)$ by~\eqref{eq:eps-bis}. The claim
therefore holds with $x^1$ being the 
$1$-periodic orbit of $k$ situated on a level close to $1$ and
corresponding to the positively parametrized characteristic underlying
$x^\eps$.  

{\it Case~II.i.2.} $\eta^\eps\in[A-\eta_A,A)$. We claim that this case
is impossible. Indeed, we would have $\bar b(\eta^\eps)\simeq
1$ and $\bar b_\eps(\eta^\eps)\simeq 1$, and in particular $x^\eps$ 
cannot be constant. There are now two cases. Either $k'(r^\eps)=\mu$,
which is excluded since $\mu\notin \mathrm{Spec}(M,\lambda)$, or
$k'(r^\eps)\in(0,\mu)$. We would then have $k(r^\eps)<c_2$, hence
$c'(\eta^\eps)/\bar b'(\eta^\eps)<(1-\eps)c_2\le c_2$, so that $\bar
b(\eta^\eps)<T_0/\mu$, which is again impossible if $x^\eps$ is not
constant. This proves the claim. 

\medskip 

{\it Case~II.ii.} If $1-\eps_0\le \eps<1$ then $1\ge\bar b_\eps(\eta^\eps)\ge
1-\eps_0$. Then, by our assumption on $\eps_0$, we must have 
$r^\eps\simeq 1$ (indeed, we
cannot have $k'(r^\eps)=\mu$ since $\bar b_\eps(\eta^\eps)\mu\notin
\mathrm{Spec}(M,\lambda)$). The claim therefore holds with $x^1$ chosen
as in Case~II.i.1.  

\medskip 

{\it Case~III.} $\eta^\eps<0$. This case is treated similarly to
Case~II. We have $r^\eps<1$, $|k|(r^\eps)=-k(r^\eps)$, and 
$$
\left\{\begin{array}{l} \dot
    x^\eps = (\bar b_\eps (\eta^\eps)-2\eps)k'(r^\eps) R_\lambda, \\
(1-\eps)\bar 
    b'(\eta^\eps) k(r^\eps)=c'(\eta^\eps), 
  \end{array}\right.
\Leftrightarrow
\left\{\begin{array}{l} \dot
    x^\eps = -(\bar b_\eps (\eta^\eps)-2\eps)|k|'(r^\eps) R_\lambda,
\\ (1-\eps)\bar 
    b'(\eta^\eps) k(r^\eps)=c'(\eta^\eps).
  \end{array}\right.
$$ 
The action is 
\begin{eqnarray*}
A_\eps(x^\eps,\eta^\eps)&=&(\bar
b_\eps(\eta^\eps)-2\eps)(r^\eps k'(r^\eps)-k(r^\eps))+c(\eta^\eps) \\
& = & -(\bar
b_\eps(\eta^\eps)-2\eps)(r^\eps |k|'(r^\eps)-|k|(r^\eps))+c(\eta^\eps).
\end{eqnarray*}

{\it Case~III.i.} If $0\le \eps < 1-\eps_0$ we distinguish the following
cases. 

{\it Case~III.i.1.} $\eta^\eps\in[-A+\eta_A,0)$. Then
$|c'|/(1-\eps)\bar b'\le\mu 
\sup |c'| /\eps_0$ is so small that $r^\eps\simeq 1$, $k(r^\eps)\simeq
0$ and $A_\eps(x^\eps,\eta^\eps)\simeq (\bar
b_\eps(\eta_\eps)-2\eps)k'(r^\eps)$. This last quantity belongs to
$\mathrm{Spec}(M,\lambda)$ because it is the action of $x^\eps$, and
the claim holds with $x^1$ being the 
$1$-periodic orbit of $|k|$ situated on a level close to $1$ and
smaller than $1$, corresponding to the negative parametrization of the
characteristic underlying $x^\eps$.  

{\it Case~III.i.2.} $\eta^\eps\in(-A,-A+\eta_A]$. Then $\bar
b(\eta^\eps)\simeq -1$ and $\bar b_\eps(\eta^\eps)-2\eps\simeq -1$. If
$x^\eps$ is not constant then, arguing as in II.i.2., we see that
$r^\eps\simeq \delta$ because $\mu\notin
\mathrm{Spec}(M,\lambda)$. The claim therefore holds with $x^1$
being the $1$-periodic orbit of $|k|$ situated on a level close to $\delta$ and
corresponding to the negative parametrization of the characteristic underlying
$x^\eps$. If $x^\eps$ is constant (and hence $r^\eps\le \delta$), the claim
holds with $x^1:=x^\eps$.  

\medskip 

{\it Case~III.ii.} If $1-\eps_0\le \eps<1$ then $-1\le \bar
b_\eps(\eta^\eps)-2\eps\le 
-1+\eps_0$. Then, by the assumption on $\eps_0$, we must have either 
$r^\eps\simeq 1$, or $r^\eps \simeq \delta$, or $r^\eps\le \delta$. In
all three cases the claim holds with $x^1$ chosen as in Case~III.i. 
\hfill{$\square$}

\medskip \noindent {\bf Step~5.}
{\it For any Hamiltonian $\widetilde H$, there is a canonical identification of
Floer complexes 
$$
CF_*^{(\alpha,\beta)}(A_H)\cong CF_*^{(\alpha,\beta)}(A_{H,1,c}).
$$
This holds in particular for $\widetilde H=|K|$.
}

To identify the
generators of the two complexes we note that the critical points
of $A_{\widetilde H,1,c}$ are pairs $(x,\eta)$ satisfying 
$\dot x = X_{\widetilde H}$ and $c'(\eta)=0$, 
while the second equation is equivalent to $\eta=0$. A gradient line
for $A_{\widetilde H,1,c}$ is a pair $(u,\eta)$ with $u$ a  
gradient line of $A_{\widetilde H}$ and $\eta:\R\to\R$ a map satisfying $\dot
\eta = c'(\eta)$ and $\lim_{s\to\pm\infty} \eta(s)=0$. Since
$c'(\eta)\neq 0$ for $\eta\neq 0$ we deduce $\eta\equiv 0$, which
shows that the Floer differentials are also canonically identified.  
\hfill{$\square$}

We have thus proved that $RFH_*^{(\alpha,\beta)}(V)\simeq
FH_*^{(\alpha,\beta)}(A_{|K|})$. On the other hand $|K|-\eps\in
\Check{\Ad}^0(\widehat V)$ for any $\eps>0$ and, because the slope
$\mu$ is bigger than $\max(|\alpha|,|\beta|)$, we infer that 
$FH_*^{(\alpha,\beta)}(A_{|K|})\simeq
\Check{SH}_*^{(\alpha,\beta)}(V)$. This proves the Theorem. 
\end{proof}

{\bf Remarks on the proof of Theorem~\ref{thm:main1}.} 

1. The phenomenon underlying Step~4 is that {\it the
action oscillates close to the period.} This principle holds
for all the action estimates in the proof. 

2. Another recurrent phenomenon is that {\it the
$\eta$-component of the critical points coagulates at $\pm A$,
respectively in $(-A+\eta_A,A-\eta_A)$}. These two cases correspond to
$r\simeq \delta$, respectively $r\simeq 1$. 
In all our action estimates we repeatedly used that $\bar
b(\eta)\simeq 1$ for $\eta\simeq A$, respectively $\bar
b'(\eta)=1/\mu$ on the interval $(-A+\eta_A,A-\eta_A)$.


\end{document}